\documentclass[12pt,a4paper]{amsart}

\usepackage[utf8]{inputenc}
\usepackage[margin=1in]{geometry}
\usepackage{thmtools} 
\usepackage{hyperref}
\usepackage{etoolbox} 
\usepackage[capitalize,noabbrev]{cleveref}
\usepackage{amsmath}
\usepackage{amsfonts}
\usepackage{amsthm}
\usepackage{amssymb}
\usepackage{xcolor}
\usepackage{graphicx}
\usepackage{tikz-cd} 
\usepackage{enumerate} 
\usepackage{subcaption} 
\usepackage{microtype} 
\usepackage{mathrsfs} 
\usepackage{enumitem}
\usepackage{comment}
\usepackage{multicol}
\usepackage{url}
\usepackage{stackengine}

\theoremstyle{plain}
\newtheorem{thrm}{Theorem}[section]
\def\bthm{\begin{thrm}}
\def\ethm{\end{thrm}}

\newtheorem{theoremx}{Theorem}

\newtheorem{conjx}{Conjecture}

\newtheorem{prop}[thrm]{Proposition}
\def\bprop{\begin{prop}}
\def\eprop{\end{prop}}

\newtheorem{ques}[thrm]{Question}
\def\bques{\begin{ques}}
\def\eques{\end{ques}}

\newtheorem{cjec}[thrm]{Conjecture}
\def\bcjec{\begin{cjec}}
\def\ecjec{\end{cjec}}

\newtheorem{cor}[thrm]{Corollary}
\def\bcor{\begin{cor}}
\def\ecor{\end{cor}}

\newtheorem{fact}[thrm]{Fact}
\def\bfact{\begin{fact}}
\def\efact{\end{fact}}

\newtheorem{lem}[thrm]{Lemma}
\newtheorem{lemma}[thrm]{Lemma}
\def\blem{\begin{lem}}
\def\elem{\end{lem}}

\theoremstyle{definition}
\newtheorem{defn}[thrm]{Definition}
\def\bdefn{\begin{defn}}
\def\edefn{\end{defn}}

\newtheorem{nota}[thrm]{Notation}
\def\bnota{\begin{nota}}
\def\enota{\end{nota}}

\newtheorem*{conc}{Conclusion}
\def\bconc{\begin{conc}}
\def\econc{\end{conc}}

\newtheorem{setup}[thrm]{Setup}

\newtheorem{alg}[thrm]{Algorithm}
\def\balg{\begin{alg}}
\def\ealg{\end{alg}}

\def\bproof{\begin{proof}}
\def\eproof{\end{proof}}

\theoremstyle{remark}

\newtheorem{rem}[thrm]{Remark}
\def\brem{\begin{rem}}
\def\erem{\end{rem}}

\newtheorem{ex}[thrm]{Example}
\def\bex{\begin{ex}}
\def\eex{\end{ex}}

\newtheorem{exs}[thrm]{Examples}
\def\bexs{\begin{exs}}
\def\eexs{\end{exs}}

\newtheorem{obs}{Observation}
\def\bobs{\begin{obs}}
\def\eobs{\end{obs}}	

\let\ra=\rightarrow

\let\hra=\hookrightarrow

\let\bra=\mapsto
\let\iso=\cong

\let\tensor=\otimes
\let\setminus=\smallsetminus

\let\epsilon=\varepsilon

\def\A{\mathcal{A}}

\def\N{\mathbb{N}}
\def\Z{\mathbb{Z}}

\def\K{\mathbb{K}}
\def\a{\mathfrak{a}}

\def\k{\Bbbk}
\def\kk{\Bbbk}

\def\p{\mathfrak{p}}
\def\m{\mathfrak{m}}

\DeclareMathOperator{\id}{id}

\DeclareMathOperator{\Ass}{Ass}
\DeclareMathOperator{\Supp}{Supp}

\DeclareMathOperator{\hgt}{ht}

\DeclareMathOperator{\fchar}{char}

\DeclareMathOperator{\n}{\mathfrak{n}}

\newcommand{\mc}[1]{\mathcal{ #1 }}

\newcommand{\card}[1]{\left \lvert #1 \right\rvert}

\newcommand{\ol}[1]{\overline{#1}}
\newcommand{\ul}[1]{\underline{#1}}

\renewcommand{\subset}{\subseteq}
\renewcommand{\supset}{\supseteq}

\AtBeginDocument{%
   \def\MR#1{}
}

\makeatletter
\newcommand\xleftrightarrow[2][]{\ext@arrow 0099{\longleftrightarrowfill@}{#1}{#2}}
\def\longleftrightarrowfill@{\arrowfill@\leftarrow\relbar\rightarrow}
\makeatother

\newcommand\xrowht[2][0]{\addstackgap[.5\dimexpr#2\relax]{\vphantom{#1}}}

\begin{document}

\title{F-Purity of Binomial Edge Ideals}
\author[A. LaClair]{Adam LaClair}
\address{University of Nebraska-Lincoln, Department of Mathematics, Lincoln, NE, USA}
\email{alaclair2@unl.edu}
\author[J. McCullough]{Jason McCullough}
\address{Iowa State University, Department of Mathematics, Ames, IA, USA}
\email{jmccullo@iastate.edu}



\begin{abstract}
In 2012, K.\ Matsuda introduced the class of weakly closed graphs and investigated when binomial edge ideals are F-pure. He proved that weakly closed binomial edge ideals are F-pure whenever the base field has positive characteristic. He conjectured that: (i) when the base field has characteristic two, every F-pure binomial edge ideal comes from a weakly closed graph; and (ii) that every binomial edge ideal is F-pure provided that the characteristic of the residue field is sufficiently large. 

In this paper, we resolve both of Matsuda's conjectures.  We confirm Matsuda's first conjecture, showing that the binomial edge ideal of a graph defines an F-pure quotient in characteristic 2 if and only if the graph is weakly closed.  We also show that Matsuda's second conjecture is false in a very strong way by showing that graphs containing asteroidal triples, such as the net, define non-F-pure binomial edge ideals in any positive characteristic.  Our results yield a complete classification of F-pure binomial edge ideals of chordal graphs as well as large families of standard graded algebras that are F-injective but neither F-pure nor F-rational in all characteristics. 
\end{abstract}


\maketitle

\section{Introduction}

Let $I \subset R := \k[x_{1},\ldots,x_{n}]$ be an ideal in a polynomial ring where $\fchar(\k) = p > 0$. We say that $R/I$ is F-pure if $F_{\ast}(R/I)$ is a pure R-module, where $F$ denotes the Frobenius map $F:r \mapsto r^p$. F-purity first appeared in the work of Hochster and Roberts, who utilized characteristic $p$ techniques to prove that rings of invariants are Cohen--Macaulay \cite{hochster1974rings}. Since its introduction, F-purity has assumed a prominent position among the various classes of F-singularities, due in part to its important connections to log canonical singularities in birational geometry \cite{hara2002f}, \cite{mustata2005FThresholds}, \cite{takagi2004fpure}. Within commutative algebra and algebraic geometry, F-pure rings constitute a mild class of F-singularities that possess many desirable properties, such as vanishing results on graded components of local cohomology modules \cite{hochster1976purity}, and satisfying the stable Harbourne conjecture \cite{grifo2019symbolic}. One active direction of research is to provide a characterization of F-pure rings under combinatorial or topological hypotheses on the ideal $I$. Such results include classifying F-pure rings of: dimension 1 \cite{goto1977structure}, dimension 2 normal Gorenstein \cite{watanabe1988study}, Stanley--Reisner rings, Schubert determinantal ideals \cite{brion2007frobenius}, ladder determinantal ideals \cite{de2023ladder}, and partial results on binomial edge ideals \cite{matsuda2018weakly}, \cite{ohtani2011graphs}. In this paper, we investigate the F-purity of binomial edge ideals.

Let $G$ be a simple graph on $n$ vertices, i.e. $G$ has no repeated edges or loops. Let $\k$ be a field and $R$ the polynomial ring $\k[x_{1},\ldots,x_{n},y_{1},\ldots,y_{n}]$. With this setup, Herzog, Hibi, Hreinsd\'{o}ttir, Kahle, and Rauh \cite{herzog2010binomial}, and independently Ohtani \cite{ohtani2011graphs}, associated to $G$ the binomial edge ideal $J_{G} \subset R$ defined as
\begin{align*}
    J_{G} := \left( \{ f_{i,j} \mid \{i,j\} \in E(G) \} \right) \subset R,
\end{align*}
where $f_{i,j} := x_{i}y_{j} - x_{j}y_{i}$. We set $R_{G} := R/J_{G}$. Since their introduction, various authors have studied the relationship between the algebraic properties of $R_{G}$ and the combinatorial structure of $G$. 

The study of the F-singularities of $R_{G}$ was pioneered by Ohtani \cite{ohtani2013binomial}, who proved that if $G$ is a complete multipartite graph and the base field $\k$ has positive characteristic, then $R_{G}$ is F-pure. Subsequently, Matsuda \cite{matsuda2018weakly} generalized this result of Ohtani via the introduction of weakly closed graphs (see Definition \ref{defn:weakly_closed}), and he proved
that every weakly closed  graph defines an F-pure binomial edge ideal (Theorem \ref{thm:weakly_closed_are_F_pure}); he conjectured that the converse should hold when the base field has characteristic two 

\begin{conjx}[Matsuda {\cite[Conjecture 2.8]{matsuda2018weakly}}]
\label{conj:weakly_closed_equiv_F_pure_in_char_2}
Let $G$ be a graph. Assume that the characteristic of the base field $\k$ is two. Then, $G$ is weakly closed if and only if $R_{G}$ is F-pure.
\end{conjx}

\noindent Work of Gonz{\'a}lez-Mart{\'\i}nez \cite[Theorem 4.3]{gonzalez2021gorenstein} and Koley--Varbaro \cite[Corollary 4.11]{koley2023grobner} established that $R_{G}$ is F-injective for every graph $G$, since there is a monomial order for which the initial ideal of $J_{G}$ is squarefree. Later, Seccia \cite{seccia2023binomial} gave a new proof of Matsuda's result, showing that binomial edge ideals of weakly closed graphs were Knutson ideals, and extended Matsuda's theorem (Theorem \ref{thm:weakly_closed_are_F_pure}) to generalized binomial edge ideals.  The first author investigated the F-pure threshold of binomial edge ideals in \cite{laclair2023invariants}. However, Matsuda's weakly closed conjecture has remained completely open. In this paper, we prove Matsuda's Conjecture~\ref{conj:weakly_closed_equiv_F_pure_in_char_2}.

\begin{theoremx}
\label{thm:matsudas_conjecture_is_true}
Conjecture \ref{conj:weakly_closed_equiv_F_pure_in_char_2} is true.
\end{theoremx}

Our new observation is that for binomial edge ideals the property of being F-pure descends from a graph to its induced subgraphs (Lemma \ref{lem:F_split_descends_to_induced_subgraph}). Thus, if an induced subgraph does not define an F-pure ideal, then neither does the original graph. A celebrated theorem of Gallai (Theorem \ref{thm:gallai}) together with an observation of Matsuda (Theorem \ref{prop:reduction_matsuda_to_finite_list}) shows that any graph which is not weakly closed must contain an induced subgraph belonging to: (1) one of ten sporadic graphs, (2) one of four ``regular" families\footnote{We use the term regular and co-regular here informally to indicate that the graphs in these families possess a high degree of regularity - not that the graphs themselves are regular or co-regular graphs.}, or (3) one of four ``co-regular" families. Hence to prove Matsuda's conjecture it suffices to show that every subgraph in this list does not define an F-pure binomial edge ideal (Proposition \ref{prop:reduction_matsuda_to_finite_list}). 

Matsuda observed that the 5-cycle $C_5$, which is not weakly closed, defines a non-F-pure binomial edge ideal in characteristic $2$; however, in characteristics $3,5,$ and $7$, $R_{C_5}$ is F-pure.  Perhaps based on this evidence, Matsuda also made the following conjecture.

\begin{conjx}[Matsuda {\cite[Conjecture 2.9]{matsuda2018weakly}}]
\label{conj:bei_are_eventually_F_pure}
Let $G$ be a graph, and let $p$ denote the characteristic of the base field. Then $R_{G}$ is F-pure for all sufficiently large $p > 0$.
\end{conjx}

\noindent Our second main result shows that Conjecture~\ref{conj:bei_are_eventually_F_pure} is false in a very strong way.

\begin{theoremx}[Theorem \ref{thm:non_F_purity}, Corollary \ref{cor:at_triple_implies_non_f_purity}]
\label{thm:intro_non_f_purity}
If $G$ contains an asteroidal triple, then $R_{G}$ is not F-pure in any positive characteristic.  In particular, Conjecture~\ref{conj:bei_are_eventually_F_pure} is false.
\end{theoremx}

\noindent It follows from Theorem~\ref{thm:intro_non_f_purity} that every sporadic graph and every graph among the regular families appearing in Gallai's list of forbidden subgraphs in Theorem~\ref{thm:gallai} defines a quotient ring that is not F-pure in any positive characteristic. Moreover, Theorem \ref{thm:intro_non_f_purity} exhibits a large family of combinatorially defined, standard graded algebras that are F-injective in every characteristic, non-normal (and hence not F-rational either), and not F-pure in any characteristic. In fact, $R_G$ is normal if and only if $R_G$ is a domain if and only if $G$ is a disjoint union of complete graphs. 
 Some members of the family are Cohen--Macaulay, like the net graph $XF_2^1$ in Figure~\ref{fig:reg_graphs}(B), while others are not, like the big claw $T_2$ in Figure~\ref{fig:finite_AT_graphs}(A). See \cite{fedder1983} and \cite{kurano2009multigraded} for further results in this direction. Finally, Theorems \ref{thm:matsudas_conjecture_is_true} and \ref{thm:intro_non_f_purity} completely characterize F-purity of $R_G$ when $G$ is chordal (Corollary \ref{cor:char_F_pure_bei}).

The rest of the paper is structured as follows.  In Sections \ref{sec:background} and \ref{sec:background_F_pure_bei}, we review the relevant background for this paper. In Section \ref{sec:non_f_pure_bei}, we prove Theorem \ref{thm:intro_non_f_purity}. In Section \ref{sec:Matsuda_conj}, we introduce Gallai's theorem and its application to Matsuda's weakly closed conjecture. In Sections \ref{sec:computation_colon_ideals}, \ref{sec:cycling}, \ref{sec:reduction_step_fedder_criterion}, and \ref{sec:co_regular_not_f_pure}, we introduce the necessary tools and prove Theorem \ref{thm:matsudas_conjecture_is_true}. In Section \ref{sec:application_matsuda_theorem}, we provide applications of Matsuda's weakly closed theorem to the study of binomial edge ideals. We prove that every unmixed graph of K\"onig type is weakly closed (Corollary \ref{cor:unmixed_Konig_type_are_wc}), and we prove that weakly closed graphs are closed under vertex completion and vertex deletion (Proposition \ref{prop:wc_preserved_under_completion_deletion}). In Section \ref{sec:further_questions}, we conclude with further questions raised by this paper. In Section \ref{sec:M2_computations}, we include a link to the Github repository containing the Macaulay2 computations utilized in this paper.

\section{Background}
\label{sec:background}

\subsection{F-Pure and F-Split Rings}

Let $R$ be a ring of characteristic $p$ for some prime $p > 0$. We denote by $F : R \ra R$ the Frobenius ring endomorphism sending $r$ to $r^{p}$. We will denote the codomain by $F_{\ast}R$. The module $F_{\ast}R$ has an $R$-module structure coming from Frobenius given by $r \cdot F_{\ast} s := F_{\ast} (r^{p}s)$, and the map $F : R \ra F_{\ast}R : r \bra F_{\ast} r^{p}$ is $R$-linear. 

\bdefn
\label{defn:F_pure}
The ring $R$ is \textbf{F-pure} if the map of $R$-modules $F : R \ra F_{\ast}R$ is pure, i.e., for every $R$-module $M$, the map
\begin{align*}
    M \ra M \tensor_{R} F_{\ast} R
\end{align*}
is injective.
\edefn 

An important criterion for testing whether a ring is F-pure is given by Fedder's criterion. 

\bthm[{\cite[Theorem 1.12]{fedder1983}}]
\label{thm:fedder_criterion}
Let $S = \k[x_{1},\ldots,x_{n}]$ be a standard graded polynomial ring with $\fchar \k = p > 0$, let $\m := (x_{1},\ldots,x_{n})$ be the graded maximal ideal, and let $I \subset S$ be a graded ideal. Then $S/I$ is F-pure if and only if
\begin{align*}
    I^{[p]} : I \not\subset \m^{[p]},
\end{align*}
where $I^{[p]} := ( \{f^{p} \mid f \in I\} )$.
\ethm 

A related notion is the F-split property. 
\bdefn
\label{defn:f_split}
The ring $R$ is \textbf{F-split} if there exists an $R$-linear map $\varphi : F_{\ast}R \ra R$ such that $\varphi \circ F = \id_{R}$, where $F : R \ra F_{\ast}R$ is the Frobenius map.
\edefn 

It is well-known that if $R$ is F-split, then $R$ is F-pure. In general, it need not be the case that F-pure rings are F-split; see e.g. \cite[Proposition 3.71]{Rotman09}. However, when $R$ is \textbf{F-finite}, i.e., $F_{\ast}R$ is a finite $R$-module, then F-pure rings are F-split. The property of a ring being F-finite is preserved under adjoining variables, passing to a quotient ring, and localizing at a multiplicative subset. Hence, if $\k$ is a perfect field, then the ring $\k[x_{1},\ldots,x_{n}]/I$ is F-finite for any ideal $I \subset \k[x_{1},\ldots,x_{n}]$. We refer the reader to \cite{ma2021f} for further information regarding the terms defined here.

\subsection{Binomial Edge Ideals}

Let $n$ be a positive integer, and let $[n]$ denote the set of integers $1$ through $n$, inclusive. Let $G$ be a graph on the vertices $[n]$, and let $E$ be the set of edges of $G$. All graphs considered here are simple, meaning they do not have loops or repeated edges. To any such graph, we associate the \textbf{binomial edge ideal}
\begin{align*}
J_{G} := \left( \{ f_{i,j} \mid \{i,j\} \in E \} \right) \subset R := \k[x_{1},\ldots,x_{n},y_{1},\ldots,y_{n}],
\end{align*}
where $\k$ is any field and $f_{i,j} := x_{i}y_{j} - x_{j}y_{i}$ for $1 \leq i < j \leq n$. We denote by $R_{G}$ the quotient ring $R/J_{G}$. In the influential paper by Herzog, Hibi, Hreinsd\'{o}ttir, Kahle, and Rauh, the authors showed that $J_{G}$ is radical \cite[Corollary 2.2]{herzog2010binomial} and provided a combinatorial characterization of its minimal primes.
\bprop[{\cite[Lemma 3.1, Theorem 3.2, Corollary 3.9]{herzog2010binomial}}]
\label{prop:min_primes_binomial_edge_ideal}
Let $G$ be a connected simple graph on $[n]$. For a subset $S \subset [n]$, define 
\begin{align*}
    P_{G}(S) := (\bigcup_{i \in S} \{x_{i},y_{i}\},J_{\tilde{G}_{1}},\ldots,J_{\tilde{G}_{c(S)}}),
\end{align*}
where $G_{1},\ldots,G_{c(S)}$ denote the connected components of $G \setminus S$, and $\tilde{G}_{i}$ denotes the complete graph on the vertex set $V(G_{i})$. Then,
\begin{enumerate}
    \item $P_{G}(S)$ is a prime ideal,
    \item $\hgt (P_{G}(S)) = 2 \#S + (n - c(S))$,
    \item $J_{G} = \bigcap_{S \subset [n]} P_{G}(S)$,
    \item $P_{G}(S)$ is a minimal prime of $J_{G}$ if and only if $S = \varnothing$, or $S \neq \varnothing$ and for each $i \in S$ one has $c(S\setminus \{i\}) < c(S)$. \label{item:min_primes_3}
\end{enumerate}
\eprop

\bdefn
With notation as in Proposition \ref{prop:min_primes_binomial_edge_ideal}, we say that $S$ is a \textbf{cut set} whenever $S$ satisfies condition \eqref{item:min_primes_3}.
\edefn

We note that $R$ carries two different gradings under which $J_{G}$ is a graded ideal. First, $J_{G}$ is a graded ideal when $R$ is standard graded, i.e., $\deg(x_{i}) = \deg(y_{i}) = 1$. Second, $J_{G}$ is a graded ideal when $R$ is $\N^{n}$ graded, where $\deg(x_{i}) = \deg(y_{i}) = e_{i}$ for $1 \leq i \leq n$, and $e_{i}$ is the $i$-th standard basis element of $\N^{n}$. 

We next recall several useful graph constructions that will be relevant for our purposes. We define the neighborhood of a vertex $v$ of $G$, denoted $N_{G}(v)$, as follows:
\begin{align*}
    N_{G}(v) := \{u \in V(G) \mid \{u,v\} \in E(G) \}.
\end{align*}
For a vertex $v$ of $G$, we define the \textbf{deletion of $G$ at $v$}, denoted $G \setminus v$, by $$V(G \setminus \{v\}) := V(G)\setminus \{v\}$$ and $$E(G \setminus v) := E(G) \setminus \{ \{u,v\} \mid u \in N_{G}(v) \}.$$ For a vertex $v$ of $G$, we define the \textbf{completion of $G$ at $v$}, denoted $G_{v}$, by $$V(G_{v}) := V(G)$$ and $$E(G_{v}) := E(G) \cup \{ \{a,b \} \mid a,b \in N_{G}(v) \text{ and } \{a,b\} \notin E(G) \}.$$ 
For graphs $H$ and $G$, we say that $H$ is an \textbf{induced subgraph} of $G$ if:
\begin{enumerate}
    \item $V(H) \subset V(G)$, 
    \item $E(H) \subset E(G)$, and 
    \item whenever $a \in V(H)$, $b \in V(H)$, and $\{a,b\} \in E(G)$, then $\{a,b\} \in E(H)$. 
\end{enumerate}
For set $A \subseteq V(G)$, we write $G[A]$ for the induced subgraph of $G$ on vertex set $A$.  Finally, given a graph $G$, we define the \textbf{complement} of $G$, denoted $\ol{G}$, as the graph determined by $$V(\ol{G}) := V(G)$$ and $$E(\ol{G}) := \{ \{i,j\} \mid \{i,j\} \notin E(G) \}.$$

\section{F-Purity of Binomial Edge Ideals}\label{sec:background_F_pure_bei}

In this section, we gather some relevant facts on the F-purity of binomial edge ideals that will be relevant for the remainder of this paper. 

The first result shows that F-purity of a binomial edge ideal over a field of characteristic $p > 0$ can be verified over the field $\Z/p\Z$.

\blem 
\label{lem:reduction_to_Z_mod_2}
Let $\kk = \Z/p\Z$, and let $\K$ be any field of characteristic $p > 0$. Let $R := \k[x_{1},\ldots,x_{n},y_{1},\ldots,y_{n}]$ and $S := \K[x_{1},\ldots,x_{n},y_{1},\ldots,y_{n}]$. Let $G$ be a graph on $n$ vertices, and let $J_{G} \subset R$ the binomial edge ideal. Then, $R/J_{G}$ is F-pure if and only if $S/J_{G}S$ is F-pure.
\elem 

\bproof
By Fedder's criterion (Theorem \ref{thm:fedder_criterion}), we have that $R/J_{G}$ is F-pure if and only if 
\begin{equation}
    \label{eqn:lem:reduction_to_Z_mod_2_1}
    J_{G}^{[p]} :_{R} J_{G} \not\subset \m^{[p]}.
\end{equation}
Since the map $\k \hra \K$ is faithfully flat and $S \iso R \tensor_{\k} \K$, it follows that equation \eqref{eqn:lem:reduction_to_Z_mod_2_1} is equivalent to 
\begin{equation}
    \label{eqn:lem:reduction_to_Z_mod_2_2}
    (J_{G}S)^{[p]} :_{S} J_{G}S \not\subset (\m S)^{[p]}.
\end{equation}
By applying Fedder's criterion (Theorem \ref{thm:fedder_criterion}), equation \eqref{eqn:lem:reduction_to_Z_mod_2_2} is equivalent to $S/J_{G}S$ being F-pure. The result now follows by observing that $J_{G}S$ is precisely the binomial edge ideal of $G$ defined over the ring $S$.
\eproof

Henceforth, we will always assume that the base field is $\Z/p\Z$ unless specified otherwise. The next lemma shows that the property of the binomial edge ideal being F-pure descends along induced subgraphs.

\blem
\label{lem:F_split_descends_to_induced_subgraph}
Let $G$ be a graph, and let $H$ be an induced subgraph of $G$. If $R_{G}$ is F-pure, then $R_{H}$ is F-pure.
\elem

\bproof
Since $\k = \Z/p\Z$ is F-finite, the property of $R_{G}$ or $R_{H}$ being F-pure is equivalent to $R_{G}$ or $R_{H}$ being F-split. Thus, we may suppose that $R_{G}$ is F-split, and it suffices to show that $R_{H}$ is F-split. Since $H$ is an induced subgraph of $G$, we have that $R_{H}$ is a $\k$-algebra retract of $R_{G}$; in other words, we have that $R_{H} \subset R_{G}$ as $\k$-algebras and that there is a map $\epsilon : R_{G} \ra R_{H}$ which is the identity on $R_{H}$. Because $R_{G}$ is F-split, there is a map $\phi : F_{\ast} R_{G} \ra R_{G}$ which splits the Frobenius map $F : R_{G} \ra F_{\ast} R_{G}$. Thus we have the following commutative diagram.
\begin{center}
\begin{tikzcd}[row sep=huge, column sep=huge, text height=1.5ex, text depth=0.25ex]
R_{H} \arrow[hook,r,swap, "\iota"] \arrow[d,"F"] & \arrow[bend left=-30, swap, l,"\epsilon"] R_{G} \arrow[d,swap,"F"] \\
F_{\ast} R_{H} \arrow[hook,r,"F_{\ast} \iota"] & \arrow[bend left=-30, swap, u,"\phi"] F_{\ast} R_{G} 
\end{tikzcd}
\end{center}
We claim that the composition of maps
\begin{align*}
    \epsilon \circ \phi \circ F_{\ast}\iota : F_{\ast} R_{H} \ra R_{H}
\end{align*}
splits the map $F : R_{H} \ra F_{\ast} R_{H}$. Indeed, we have that for $r \in R_{H}$ that
\begin{align*}
    \epsilon \circ \phi \circ F_{\ast}\iota (F_{\ast} r^{p}) 
    &= \epsilon \circ \phi \circ F_{\ast} (\iota (r)^{p}) \\
    &= \epsilon \circ (\iota (r)) \\
    &= r.
\end{align*}
We note that each of the maps $F_{\ast}\iota$, $\phi$, and $\epsilon$ are $R_{H}$-linear, and hence their composition is also $R_{H}$-linear.
\eproof

The next result establishes that F-purity of binomial edge ideals is preserved under vertex completion and deletion.

\blem
\label{lem:F_pure_preserved_vertex_deletion_completion}
Let $G$ be a graph, and let $v \in V(G)$. If $R_{G}$ is F-pure, then the graphs $G \setminus v$, $G_{v}$, and $G_{v} \setminus v$ realize F-pure binomial edge ideals.
\elem

\bproof
First, we show that $R_{G\setminus v}$ is F-pure. Lemma \ref{lem:F_split_descends_to_induced_subgraph} implies that $G \setminus v$ defines an F-pure binomial edge ideal since $G \setminus v$ is an induced subgraph of $G$. 

Next, we show that $R_{G_{v}}$ is F-pure. By Fedder's criterion (Theorem \ref{thm:fedder_criterion}), it suffices to show that 
\begin{equation}
    \label{eqn:F_pure_preserved_vertex_deletion_completion_0}
    J_{G_{v}}^{[p]}:J_{G_{v}} \not \subset m^{[p]}.
\end{equation}
We will show that 
\begin{equation}
\label{eqn:F_pure_preserved_vertex_deletion_completion_1} 
J_{G}^{[p]}:J_{G} \subset J_{G_{v}}^{[p]}:J_{G_{v}}.
\end{equation}
Since $R_{G}$ is F-pure, we have that 
\begin{align*}
    J_{G}^{[p]}:J_{G} \not \subset \m^{[p]}.
\end{align*}
Hence, equation \eqref{eqn:F_pure_preserved_vertex_deletion_completion_1} will imply equation \eqref{eqn:F_pure_preserved_vertex_deletion_completion_0}, which will complete the proof.

Let $r \in J_{G}^{[p]}:J_{G}$, and let $\{i,j\} \in E(G_{v})$. If $\{i,j\} \in E(G)$, then we have that 
\begin{align*}
    r \cdot f_{ij} \in J_{G}^{[p]} \subset J_{G_{v}}^{[p]}.
\end{align*}
Suppose that $\{i,j\} \in E(G_{v}) \setminus E(G)$. Then, $i,j \in N_{G}(v)$. By Pl\"{u}cker's relations (see e.g., \cite{bruns2022determinants}), we have that $x_{v} \cdot f_{ij} = \pm x_{i} \cdot f_{vj} \mp x_{j} \cdot f_{vi}$. Because $\{v,j\}$ and $\{v,i\}$ are edges of $G$, it follows from the previous case that 
\begin{align*}
    r \cdot x_{v} \cdot f_{ij} \in J_{G_{v}}^{[p]}.
\end{align*}
Observe that
\begin{align*}
    \Ass_{R}(R/J_{G_{v}}^{[p]}) = \Ass_{R}(R/J_{G_{v}})  = \{ P_{G}(S) \mid v \notin S \}
\end{align*}
by \cite{peskine1973dimension} and {\cite[Lemma 4.5]{bolognini2022cohen}}. Proposition \ref{prop:min_primes_binomial_edge_ideal} implies that $x_{v}$ is a non-zerodivisor on $R/J_{G_{v}}^{[p]}$. Consequently, $r \cdot f_{ij} \in J_{G_{v}}^{[p]}$. This proves that $r \in J_{G_{v}}^{[p]}:J_{G_{v}}$, and thus establishes equation \eqref{eqn:F_pure_preserved_vertex_deletion_completion_1}, which completes the proof.

Finally, we observe that the above two cases imply that $G_{v} \setminus v$ defines an F-pure binomial edge ideal.
\eproof

\section{Non-F-Pure Binomial Edge Ideals} \label{sec:non_f_pure_bei}


In this section, we exhibit a large family of graphs whose binomial edge ideal do not realize an F-pure quotient ring for any positive characteristic $p > 0$. Thus, we answer Matsuda's Conjecture \ref{conj:bei_are_eventually_F_pure} in the negative.  

We begin with some preparatory results.

\bthm[Lucas's Theorem {\cite{lucas1878theorie}}]
\label{thm_Lucas}
Let $p$ be a prime integer and $m$ and $n$ positive integers. Choose a positive integer $e$ so that we can write $m = \sum_{i=0}^{e} m_{i}p^{i}$ and $n = \sum_{i=0}^{e} n_{i} p^{i}$ where $0 \leq m_{i}, n_{i} < p$ for $0 \leq i \leq e$. Then
\begin{align*}
\binom{n}{m} \equiv \prod_{i=0}^{e} \binom{n_{i}}{m_{i}} \mod{p},
\end{align*}
with the convention that $\binom{a}{b} = 0$ if $a < b$.
\ethm

\noindent For a modern treatment of Lucas's Theorem, see \cite[Theorem 1]{Fine47}.

\begin{lemma}
\label{lem:binom_coeff_mod_p}
Let $p$ be a prime. For all $0 \leq i \leq p-1$, we have that
\begin{align*}
    \binom{p-1}{i} \equiv (-1)^i \mod{p}.
\end{align*}

\end{lemma}

\begin{proof}
The proof is an exercise in elementary number theory. A write-up can be found in \cite{binomialIdentity}.
\end{proof}

\begin{lemma}
\label{lem:switching_terms}
Let $p$ be a positive prime number. If $a$, $b$, and $c$ are distinct positive integers, then
\begin{align*}
(f_{a,b}f_{b,c})^{p-1} \equiv (x_by_bf_{a,c})^{p-1}    \mod{(x_{b}^{p}, y_{b}^{p})}.
\end{align*}
\end{lemma}

\bproof
The result can be verified by hand when $p = 2$. Hence, we may suppose that $p > 2$ is an odd positive prime number. By the Pigeon hole principle, it is clear that every non-zero term of $(f_{a,b}f_{b,c})^{p-1}$ is divisible by $x_b^{p}$, $y_b^{p}$, or $x_b^{p-1}y_b^{p-1}$. The term of $(f_{a,b}f_{b,c})^{p-1}$ divisible by $x_b^{p-1}y_b^{p-1}$ is
\begin{equation}
\label{eqn:switching_terms_1}
\left(\sum_{i=0}^{p-1} \binom{p-1}{i}^2 (x_ay_c)^i\cdot (x_cy_a)^{p-1-i}\right) x_b^{p-1}y_b^{p-1},
\end{equation}
which by Lemma \ref{lem:binom_coeff_mod_p} is equal to 
\begin{equation}
\label{eqn:switching_terms_1.5}
\left(\sum_{i=0}^{p-1} (x_ay_c)^i\cdot (x_cy_a)^{p-1-i}\right) x_b^{p-1}y_b^{p-1}.
\end{equation}
By the binomial theorem, the polynomial $f_{a,c}^{p-1}$ can be expanded as
\begin{equation}
\label{eqn:switching_terms_2}
    \sum_{i=0}^{p-1} (-1)^{p-i-1} \binom{p-1}{i} (x_ay_c)^i\cdot (x_cy_a)^{p-1-i}.
\end{equation}
As $p$ is in particular an odd number, Lemma \ref{lem:binom_coeff_mod_p} implies that expression \eqref{eqn:switching_terms_2} is equal to 
\begin{equation}
    \label{eqn:switching_terms_3}
    \sum_{i=0}^{p-1} (x_ay_c)^i\cdot (x_cy_a)^{p-1-i}.
\end{equation}
The result now follows by substituting $f_{a,c}^{p-1}$ into expression \eqref{eqn:switching_terms_1.5}.
\eproof

\begin{lemma}\label{lem:Lucas}
Let $R = \k[x_1,x_2,x_3,y_1,y_2,y_3]$ and $\m$ the graded maximal ideal, where $\fchar(\k) = p > 0$. Then
\begin{align*}
    (f_{1,2}f_{2,3}f_{1,3})^{p-1} \in \m^{[p]}.
\end{align*}
\end{lemma}

\begin{proof}
By Lemma \ref{lem:switching_terms}, we have that
\begin{align*}
    (f_{1,2}f_{2,3}f_{1,3})^{p-1} \equiv (x_2y_2)^{p-1} (f_{1,3})^{2p-2} \mod{\m^{[p]}}.
\end{align*}
Every term of $(f_{1,3})^{2p-2}$ belongs to $\m^{[p]}$ except for the term $(x_1y_3x_3y_1)^{p-1}$. The coefficient on this term is
\begin{align*}
    \binom{2p-2}{p-1}.
\end{align*}
By Theorem \ref{thm_Lucas}, we have that 
\begin{align*}
    \binom{2p-2}{p-1}  \equiv \binom{1}{0}\cdot \binom{p-2}{p-1} \equiv 0 \mod{p}.
\end{align*}
\end{proof}

\begin{lemma}\label{lem:swap}
Let $G$ be a connected graph, and let $S$ be a cut set of $G$. Let $i,j,k \in V(G)$. Denote by $\p = P_G(S)$ the minimal prime of $J_{G}$ corresponding to the cut set $S$. Then $$\p^{[p]}:(f_{i,j},f_{j,k}) \subseteq \p^{[p]}:(f_{i,k}).$$
\end{lemma}

\begin{proof}
Let $r \in \p^{[p]}:(f_{i,j},f_{j,k})$ in which case $rf_{i,j}$ and $rf_{j,k}$ belong to $\p^{[p]}$. From the following  relation 
\begin{align*}
    x_kf_{i,j} - x_jf_{i,k} + x_if_{j,k} = 0,
\end{align*}
we see that
\begin{align*}
    rx_jf_{i,k} = rx_kf_{i,j} + rx_jf_{j,k} \in \p^{[p]}.
\end{align*}
Since $\sqrt{\p^{[p]}} = \p$ does not contain $x_j$, it follows that $rf_{i,k} \in \p^{[p]}$. Thus, $r \in \p^{[p]}:(f_{i,k})$ as claimed.
\end{proof}

\bdefn \label{defn_minimal_vertex_isolators}
Let $G$ be a connected graph, and let $S$  be a cut set of $V(G)$. 
\begin{enumerate}
    \item We say that $S$ is a \textbf{minimal vertex separator} of $G$ if the induced subgraph of $G$ on vertex set $V(G) \setminus S$ consists of two connected components.
    \item In addition, we say that $S$ is a \textbf{minimal vertex isolator} if one of the connected components of $G[V(G) \setminus S]$ is a singleton vertex.
\end{enumerate}
For disjoint subsets $S$ and $A$ of $V(G)$, we denote by $P(S,A)$ the ideal 
\begin{align*}
    (x_i,y_i \mid i \in S) + J_{\tilde{G_A}},
\end{align*}
where $\tilde{G}_A$ is the complete graph on $A$.
\edefn

\begin{lemma}\label{lem:one:colon}
Let $G$ be a connected graph, and let $S$ be a minimal vertex isolator with $V-S = A \sqcup \{v\}$.  Let $i,j \in A$ be distinct.  Then, \[P(S,A)^{[p]}:J_G \subseteq (x_k,y_k\mid k \in (S \cup A - \{i,j\}))^{[p]} + (f_{i,j})^{p-1}.\]
\end{lemma}

\begin{proof}
    Since $G[A]$ is connected by assumption, there is a path $i = v_0,v_1,\ldots,v_r = j$ in $G[A]$ with $r \geq 1$.  Thus $f_{v_{k-1},v_k} \in J_G$ for $1 \le k \le r$.  Therefore, $P(S,A)^{[p]}:J_G \subseteq P(S,A)^{[p]}:(f_{v_0,v_1},f_{v_1,v_2},\ldots,f_{v_{r-1},v_r})$.  
    
    First we show that $P(S,A)^{[p]}:J_G \subseteq P(S,A)^{[p]}:(f_{i,j})$. We observe that
    \begin{align*}
        &\phantom{=(} P(S,A)^{[p]}:(f_{v_0,v_1},\ldots,f_{v_{r-1},v_r}) \\&= \left( P(S,A)^{[p]}:(f_{v_0,v_1},f_{v_1,v_2}) \right) \cap \left( P(S,A)^{[p]}:(f_{v_2,v_3},\ldots,f_{v_{r-1},v_r})\right) \\
        &\subseteq  \left( P(S,A)^{[p]}:(f_{v_0,v_2})\right) \cap \left(P(S,A)^{[p]}:(f_{v_2,v_3},\ldots,f_{v_{r-1},v_r}) \right)\\
        &= P(S,A)^{[p]}:(f_{v_0,v_2},f_{v_2,v_3}\ldots,f_{v_{r-1},v_r})
    \end{align*}
    where the first containment follows from Lemma~\ref{lem:swap}. By repeating this argument it follows that $P(S,A)^{[p]}:J_G \subseteq P(S,A)^{[p]}:(f_{i,j})$.
    
    Next, we observe that for $g \in P(S,A)^{[p]}:(f_{i,j})$
    \begin{align*}
        g \cdot f_{i,j} \in P(S,A)^{[p]} \subset (x_k,y_k\mid k \in (S \cup A - \{i,j\}))^{[p]} + (f_{i,j})^{p}.
    \end{align*}
    The conclusion now follows from the fact that $\{x_l,y_l \mid l \in S \cup A - \{i,j\}\} \cup \{f_{i,j}\}$ form a regular sequence.
\end{proof}

\begin{prop}\label{prop:reg:sequence}
    Let $R = \kk[x_1,\ldots,x_n]$ be a polynomial ring, and let $f_1,\ldots,f_t \in R$ be a homogeneous regular sequence.  For $\alpha \in \mathbb{N}^t$, write $f^\alpha = \prod_{i = 1}^t f_i^{\alpha_i}$.  Then \[(f^{\alpha_1},\ldots,f^{\alpha_r}) \cap (f^{\beta_1},\ldots,f^{\beta_s}) = (\mathrm{lcm}(f^{\alpha_i},f^{\beta_j}) \mid 1 \le i \le r, 1 \le j \le s).\]
\end{prop}

\begin{proof}
    First, note that the subring $S = \kk[f_1,\ldots,f_t]$ is isomorphic to a polynomial ring in $t$ variables.  In $S$, the intersection amounts to a statement about the intersection of monomial ideals, which is well-known; see e.g. \cite[Proposition 2.1.5]{MRSW2018}.  The conclusion follows since the extension $S \to R$ is flat.
\end{proof}

\bthm
\label{thm:non_F_purity}
Let $G$ be a connected graph.  Suppose $G$ has three mutually nonadjacent vertices $v_1,v_2,v_3$ such that $N(v_i)$ is a minimal vertex isolator for $i = 1,2,3$ and such that $G \setminus N(v_{i}) = \{v_{i}\} \sqcup (G\setminus N[v_i])$. Then $R_G$ is not F-pure for any positive characteristic $p > 0$.
\ethm

\begin{proof}
By Fedder's Criterion Theorem (\ref{thm:fedder_criterion}), it suffices to show that $J_G^{[p]}:J_G \subseteq \m^{[p]}$.  First, we observe that the hypothesis that $N(v_{i})$ is a minimal vertex isolator of $G$ implies that
\begin{align*}
    \p_i := P(N(v_i),G-N(v_i)), \quad i = 1,2,3
\end{align*}
is a minimal prime of $J_{G}$. Lemma~\ref{lem:one:colon} implies that 
\begin{align*}
\p_i^{[p]}:J_G \subseteq (x_l,y_l \mid l \in (V(G) - \{v_i,v_j,v_k\}))^{[p]} + (f_{v_j,v_k})^{p-1}
\end{align*}
for distinct $i,j,k \in \{1,2,3\}$.  Write $\eta = (x_l,y_l \mid l \in (V(G) \setminus \{v_i,v_j,v_k\}))$. Then, we compute that
\begin{align}
    J_G^{[p]}:J_G &\subseteq \left(\bigcap_{l = 1}^3 \p_i\right)^{[p]}:J_G \label{step1}\\
    &= \left(\bigcap_{l = 1}^3 \p_i^{[p]}\right):J_G\label{step2}\\
    &= \bigcap_{l = 1}^3 \left( \p_i^{[p]}:J_G\right)\label{step3}\\
    &\subseteq ( \eta^{[p]} + (f_{v_i,v_j})^{p-1}) \cap ( \eta^{[p]} + (f_{v_i,v_k})^{p-1}) \cap ( \eta^{[p]} + (f_{v_j,v_k})^{p-1})\label{step4}\\
    &\subseteq ( \eta^{[p]} + (f_{v_i,v_j}f_{v_i,v_k})^{p-1})  \cap ( \eta^{[p]} + (f_{v_j,v_k})^{p-1})\label{step5}\\
    &\subseteq ( \eta^{[p]} + (f_{v_i,v_j}f_{v_i,v_k}f_{v_j,v_k})^{p-1})\label{step6}\\
    &\subseteq \m^{[p]}\label{step7}.
    \end{align}

We have that: step~(\ref{step1}) follows from $J_G^{[p]} \subseteq \bigcap_{l = 1}^3 \p_i$; step~(\ref{step2}) follows from the definition of Frobenius power; step~(\ref{step3}) follows from the definition of colon ideals; step~(\ref{step4}) follows from Lemma~\ref{lem:one:colon}; step~(\ref{step5}) follows from Proposition~\ref{prop:reg:sequence} since $\{x_l,y_l \mid l \in (V(G) - \{v_i,v_j,v_k\})\} \cup \{f_{v_i,v_j},f_{v_i,v_k}\}$ forms a regular sequence; and step~(\ref{step6}) also follows from Proposition~\ref{prop:reg:sequence} since $\{x_l,y_l \mid l \in (V(G) - \{v_i,v_j,v_k\})\} \cup \{f_{v_i,v_j} \cdot f_{v_i,v_k},f_{v_j,v_k}\}$ is a regular sequence.  The final step~(\ref{step7}) follows from Lemma~\ref{lem:Lucas}.
\end{proof}

\subsection{Asteroidal Triples}

We next utilize Theorem \ref{thm:non_F_purity} to show that any graph containing an asteroidal triple is not F-pure in any positive characteristic.

\bdefn[\cite{isgciATFree},\cite{LB62}]
Three vertices of a graph form an \textbf{asteroidal triple} if any two of them are connected by a path avoiding the neighborhood of the third. A graph containing an asteroidal triple is called \textbf{asteroidal}.  A graph is \textbf{AT-free} if it does not contain an asteroidal triple.
\edefn 

There is a characterization of AT-free graphs via forbidden induced subgraphs.

\bthm[\cite{isgciATFree}]
\label{thm:forbidden_subgraph_AT_free}
A graph is AT-free if and only if the graph does not contain any of the following subgraphs as an induced subgraph:
\begin{multicols}{3}
\begin{enumerate}
\item $C_{n}$ for $n \geq 6$
\item $T_{2}$
\item $X_{2}$
\item $X_{3}$
\item $X_{30}$
\item $X_{31}$
\item $X_{32}$
\item $X_{33}$
\item $X_{34}$
\item $X_{35}$
\item $X_{36}$
\item $X_{37}$
\item $X_{38}$
\item $X_{39}$
\item $X_{40}$
\item $X_{41}$
\item $XF_{2}^{n+1}$ for $n \geq 0$
\item $XF_{3}^{n}$ for $n \geq 0$
\item $XF_{4}^{n}$ for $n \geq 0$
\end{enumerate}
\end{multicols}
\ethm

\bex 
The graphs in Theorem \ref{thm:forbidden_subgraph_AT_free} are depicted in Figures \ref{fig:finite_AT_graphs} and \ref{fig:reg_graphs}. 
For the graph $C_{n}$ as depicted in Figure \ref{fig:reg_graphs}, the vertex set $\{1,3,5\}$ is an asteroidal triple whenever $n \geq 6$. For the remaining graphs shown in Figures \ref{fig:finite_AT_graphs} and \ref{fig:reg_graphs}, the vertices $\{v_1,v_2,v_3\}$ form an asteroidal triple.
\eex 



\def\edTT{2}
\def\scale{0.75}

\def\ptsTT{1*\edTT/1*\edTT/v_1/180,
3*\edTT/1*\edTT/v_2/0,
2*\edTT/2*\edTT/v_3/0}
    
\def\ptsTTUnlabel{1*\edTT/0*\edTT/1,
2*\edTT/0*\edTT/2,
3*\edTT/0*\edTT/3,
2*\edTT/1*\edTT/5}

\def\edgesTT{
    v_1/1,
    1/2,
    2/3,
    3/v_2,
    2/5,
    5/v_3}



\def\edXT{2}
\def\scale{0.75}

\def\ptsXT{0*\edXT/0*\edXT/v_1/270,
3*\edXT/0*\edXT/v_2/270,
1.5*\edXT/1.866*\edXT/v_3/0}
    
\def\ptsXTUnlabel{1*\edXT/0*\edXT/1,
1.5*\edXT/-0.866*\edXT/2,
2*\edXT/0*\edXT/3,
1.5*\edXT/0.866*\edXT/5}

\def\edgesXT{
    v_1/1,
    1/2,
    1/5,
    2/3,
    3/v_2,
    3/5,
    5/v_3}



\def\edXR{2}
\def\scale{0.75}

\def\ptsXR{0*\edXR/0*\edXR/v_1/270,
2*\edXR/0*\edXR/v_2/270,
1*\edXR/2*\edXR/v_3/0}
    
\def\ptsXRUnlabel{
1*\edXR/0*\edXR/1,
0*\edXR/1*\edXR/3,
1*\edXR/1*\edXR/4,
2*\edXR/1*\edXR/5,
1*\edXR/1*\edXR/6}

\def\edgesXR{
    v_1/1,
    v_1/3,
    1/4,
    1/v_2,
    v_2/5,
    3/4,
    4/5,
    4/v_3}



\def\edXTZ{2}
\def\scale{0.75}

\def\ptsXTZ{0*\edXTZ/0*\edXTZ/v_1/270,
3*\edXTZ/0*\edXTZ/v_2/270,
1.5*\edXTZ/1.866*\edXTZ/v_3/0}
    
\def\ptsXTZUnlabel{1*\edXTZ/0*\edXTZ/1,
2*\edXTZ/0*\edXTZ/2,
1*\edXTZ/1*\edXTZ/3,
2*\edXTZ/1*\edXTZ/4}

\def\edgesXTZ{
    v_1/1,
    1/2,
    1/3,
    2/v_2,
    2/4,
    3/4,
    3/v_3,
    4/v_3}



\def\edXTO{2}
\def\scale{0.75}

\def\ptsXTO{0*\edXTO/0*\edXTO/v_1/270,
2*\edXTO/0*\edXTO/v_2/270,
1*\edXTO/2*\edXTO/v_3/0}
    
\def\ptsXTOUnlabel{
1*\edXTO/0*\edXTO/1,
0*\edXTO/1*\edXTO/2,
1*\edXTO/1*\edXTO/3,
2*\edXTO/1*\edXTO/4,
1*\edXTO/1*\edXTO/5}

\def\edgesXTO{
    v_1/1,
    v_1/2,
    1/3,
    1/v_2,
    v_2/4,
    2/3,
    3/4,
    3/v_3,
    1/2,
    1/4}



\def\edXTT{2}
\def\scale{0.75}

\def\ptsXTT{0*\edXTT/0*\edXTT/v_1/270,
2*\edXTT/0*\edXTT/v_2/270,
1*\edXTT/2*\edXTT/v_3/0}
    
\def\ptsXTTUnlabel{
1*\edXTT/0*\edXTT/1,
0*\edXTT/1*\edXTT/2,
1*\edXTT/1*\edXTT/3,
2*\edXTT/1*\edXTT/4,
1*\edXTT/1*\edXTT/5}

\def\edgesXTT{
    v_1/1,
    v_1/2,
    1/3,
    1/v_2,
    v_2/4,
    2/3,
    3/4,
    3/v_3,
    1/2}



\def\edXTR{2}
\def\scale{0.75}

\def\ptsXTR{0*\edXTR/0*\edXTR/v_1/270,
2*\edXTR/0*\edXTR/v_2/270,
0.5*\edXTR/1.866*\edXTR/v_3/0}
    
\def\ptsXTRUnlabel{
1*\edXTR/0*\edXTR/1,
0*\edXTR/1*\edXTR/2,
1*\edXTR/1*\edXTR/3,
2*\edXTR/1*\edXTR/4,
1*\edXTR/1*\edXTR/5}

\def\edgesXTR{
    v_1/1,
    v_1/2,
    1/3,
    1/v_2,
    v_2/4,
    2/3,
    3/4,
    3/v_3,
    2/v_3}
    


\def\edXTF{2}
\def\scale{0.75}

\def\ptsXTF{0*\edXTF/0*\edXTF/v_1/270,
2*\edXTF/0*\edXTF/v_2/270,
1*\edXTF/1.732*\edXTF/v_3/0}
    
\def\ptsXTFUnlabel{
1*\edXTF/0*\edXTF/1,
1.5*\edXTF/0*\edXTF/2,
0.5*\edXTF/0.866*\edXTF/3,
1.5*\edXTF/0.866*\edXTF/4}

\def\edgesXTF{
    v_1/1,
    v_1/3,
    1/3,
    1/2,
    v_2/4,
    2/3,
    3/4,
    3/v_3,
    4/v_3,
    2/v_2,
    1/4}



\def\edXTV{2}
\def\scale{0.75}

\def\ptsXTV{0*\edXTV/0*\edXTV/v_1/270,
2*\edXTV/0*\edXTV/v_2/270,
0.5*\edXTV/1.866*\edXTV/v_3/0}
    
\def\ptsXTVUnlabel{
1*\edXTV/0*\edXTV/1,
0*\edXTV/1*\edXTV/2,
1*\edXTV/1*\edXTV/3,
2*\edXTV/1*\edXTV/4,
1*\edXTV/1*\edXTV/5}

\def\edgesXTV{
    v_1/1,
    v_1/2,
    1/3,
    1/v_2,
    v_2/4,
    2/3,
    3/4,
    3/v_3,
    2/v_3,
    1/2}



\def\edXTX{2}
\def\scale{0.75}

\def\ptsXTX{0*\edXTX/0*\edXTX/v_1/270,
2*\edXTX/0*\edXTX/v_2/270,
1*\edXTX/1.732*\edXTX/v_3/0}
    
\def\ptsXTXUnlabel{
0.5*\edXTX/0*\edXTX/1,
1.5*\edXTX/0*\edXTX/2,
0.5*\edXTX/0.866*\edXTX/3,
1.5*\edXTX/0.866*\edXTX/4}

\def\edgesXTX{
    v_1/1,
    v_1/3,
    1/2,
    v_2/4,
    3/4,
    3/v_3,
    4/v_3,
    2/v_2,
    2/3,
    1/4}



\def\edXTS{2}
\def\scale{0.75}

\def\ptsXTS{0*\edXTS/0*\edXTS/v_1/270,
2*\edXTS/0*\edXTS/v_2/270,
1*\edXTS/1.732*\edXTS/v_3/0}
    
\def\ptsXTSUnlabel{
1*\edXTS/0*\edXTS/1,
0.5*\edXTS/0.866*\edXTS/3,
1.5*\edXTS/0.866*\edXTS/4}

\def\edgesXTS{
    v_1/1,
    v_1/3,
    v_2/4,
    3/4,
    3/v_3,
    4/v_3,
    1/v_2}



\def\edXTE{2}
\def\scale{0.75}

\def\ptsXTE{0*\edXTE/0*\edXTE/v_1/270,
2*\edXTE/0*\edXTE/v_2/270,
0.5*\edXTE/1.732*\edXTE/v_3/180}
    
\def\ptsXTEUnlabel{
1*\edXTE/0*\edXTE/1,
0.5*\edXTE/0.866*\edXTE/3,
1.5*\edXTE/0.866*\edXTE/4,
1.5*\edXTE/1.732*\edXTE/5}

\def\edgesXTE{
    v_1/1,
    v_1/3,
    v_2/4,
    3/4,
    v_3/3,
    v_3/5,
    4/5,
    1/v_2}



\def\edXTN{2}
\def\scale{0.75}

\def\ptsXTN{0*\edXTN/0*\edXTN/v_1/270,
2*\edXTN/0*\edXTN/v_2/270,
1*\edXTN/1.732*\edXTN/v_3/0}
    
\def\ptsXTNUnlabel{
0.5*\edXTN/0*\edXTN/1,
1.5*\edXTN/0*\edXTN/2,
0.5*\edXTN/0.866*\edXTN/3,
1.5*\edXTN/0.866*\edXTN/4}

\def\edgesXTN{
    v_1/1,
    v_1/3,
    1/2,
    v_2/4,
    3/v_3,
    4/v_3,
    2/v_2,
    2/3,
    1/4}



\def\edXF{2}
\def\scale{0.75}

\def\ptsXF{0*\edXF/0*\edXF/v_1/270,
2*\edXF/0*\edXF/v_2/270,
1*\edXF/1.732*\edXF/v_3/0}
    
\def\ptsXFUnlabel{
0.5*\edXF/0*\edXF/1,
1*\edXF/0*\edXF/2,
0.5*\edXF/0.866*\edXF/3,
1.5*\edXF/0.866*\edXF/4}

\def\edgesXF{
    v_1/1,
    v_1/3,
    1/2,
    v_2/4,
    3/v_3,
    4/v_3,
    2/v_2,
    2/3,
    1/4,
    2/4}



\def\edXFO{2}
\def\scale{0.75}

\def\ptsXFO{0*\edXFO/0*\edXFO/v_1/270,
3*\edXFO/0*\edXFO/v_2/270,
1.5*\edXFO/1.866*\edXFO/v_3/0}
    
\def\ptsXFOUnlabel{1*\edXFO/0*\edXFO/1,
2*\edXFO/0*\edXFO/2,
1*\edXFO/1*\edXFO/3,
2*\edXFO/1*\edXFO/4}

\def\edgesXFO{
    v_1/1,
    1/2,
    1/3,
    2/v_2,
    2/4,
    3/v_3,
    4/v_3}

%
%
\begin{figure}[!htbp]
\begin{subfigure}{.3\linewidth}
\begin{center}
\begin{tikzpicture}
	\foreach \x/\y/\z/\w in \ptsTT {
		\draw[fill = black!50] (\scale*\x,\scale*\y) circle [radius = 0.1] node[label = {[label distance = 0.05 cm]\w: $\z$}] (\z) {}; 
	}
 
\foreach \x/\y/\z in \ptsTTUnlabel {
		\draw[fill = black!50] (\scale*\x,\scale*\y) circle [radius = 0.1] node (\z) {}; 
	}
 
\foreach \x/\y in \edgesTT { \draw (\x) -- (\y); }  

\end{tikzpicture}
\subcaption{The graph $T_{2}$} \label{graph:TT}
\end{center}
\end{subfigure}
\hfill
%
%
%
\begin{subfigure}{.3\linewidth}
\begin{center}
\begin{tikzpicture}
	\foreach \x/\y/\z/\w in \ptsXT {
		\draw[fill = black!50] (\scale*\x,\scale*\y) circle [radius = 0.1] node[label = {[label distance = 0.05 cm]\w: $\z$}] (\z) {}; 
	}
 
\foreach \x/\y/\z in \ptsXTUnlabel {
		\draw[fill = black!50] (\scale*\x,\scale*\y) circle [radius = 0.1] node (\z) {}; 
	}

\foreach \x/\y in \edgesXT { \draw (\x) -- (\y); }  


\end{tikzpicture}
\subcaption{The graph $
X_{2}$} \label{graph:XT}
\end{center}
\end{subfigure}
\hfill 
%
%
%
\begin{subfigure}{.3\linewidth}
\begin{center}
\begin{tikzpicture}
	\foreach \x/\y/\z/\w in \ptsXR {
		\draw[fill = black!50] (\scale*\x,\scale*\y) circle [radius = 0.1] node[label = {[label distance = 0.05 cm]\w: $\z$}] (\z) {}; 
	}
 
\foreach \x/\y/\z in \ptsXRUnlabel {
		\draw[fill = black!50] (\scale*\x,\scale*\y) circle [radius = 0.1] node (\z) {}; 
	}

\foreach \x/\y in \edgesXR { \draw (\x) -- (\y); }  


\end{tikzpicture}
\subcaption{The graph $
X_{3}$} \label{graph:XR}
\end{center}
\end{subfigure}
\hfill
%
%
%
\begin{subfigure}{.3\linewidth}
\begin{center}
\begin{tikzpicture}
	\foreach \x/\y/\z/\w in \ptsXTZ {
		\draw[fill = black!50] (\scale*\x,\scale*\y) circle [radius = 0.1] node[label = {[label distance = 0.05 cm]\w: $\z$}] (\z) {}; 
	}
 
\foreach \x/\y/\z in \ptsXTZUnlabel {
		\draw[fill = black!50] (\scale*\x,\scale*\y) circle [radius = 0.1] node (\z) {}; 
	}

\foreach \x/\y in \edgesXTZ { \draw (\x) -- (\y); }  


\end{tikzpicture}
\subcaption{The graph $
X_{30}$} \label{graph:XTZ}
\end{center}
\end{subfigure}
\hfill
%
%
%
\begin{subfigure}{.3\linewidth}
\begin{center}
\begin{tikzpicture}
	\foreach \x/\y/\z/\w in \ptsXTO {
		\draw[fill = black!50] (\scale*\x,\scale*\y) circle [radius = 0.1] node[label = {[label distance = 0.05 cm]\w: $\z$}] (\z) {}; 
	}
 
\foreach \x/\y/\z in \ptsXTOUnlabel {
		\draw[fill = black!50] (\scale*\x,\scale*\y) circle [radius = 0.1] node (\z) {}; 
	}

\foreach \x/\y in \edgesXTO { \draw (\x) -- (\y); }  


\end{tikzpicture}
\subcaption{The graph $
X_{31}$} \label{graph:XTO}
\end{center}
\end{subfigure}
\hfill
%
%
%
\begin{subfigure}{.3\linewidth}
\begin{center}
\begin{tikzpicture}
	\foreach \x/\y/\z/\w in \ptsXTT {
		\draw[fill = black!50] (\scale*\x,\scale*\y) circle [radius = 0.1] node[label = {[label distance = 0.05 cm]\w: $\z$}] (\z) {}; 
	}
 
\foreach \x/\y/\z in \ptsXTTUnlabel {
		\draw[fill = black!50] (\scale*\x,\scale*\y) circle [radius = 0.1] node (\z) {}; 
	}

\foreach \x/\y in \edgesXTT { \draw (\x) -- (\y); }  


\end{tikzpicture}
\subcaption{The graph $
X_{32}$} \label{graph:XTT}
\end{center}
\end{subfigure}
\hfill
%
%
%
\begin{subfigure}{.3\linewidth}
\begin{center}
\begin{tikzpicture}
	\foreach \x/\y/\z/\w in \ptsXTR {
		\draw[fill = black!50] (\scale*\x,\scale*\y) circle [radius = 0.1] node[label = {[label distance = 0.05 cm]\w: $\z$}] (\z) {}; 
	}
 
\foreach \x/\y/\z in \ptsXTRUnlabel {
		\draw[fill = black!50] (\scale*\x,\scale*\y) circle [radius = 0.1] node (\z) {}; 
	}

\foreach \x/\y in \edgesXTR { \draw (\x) -- (\y); }  


\end{tikzpicture}
\subcaption{The graph $
X_{33}$} \label{graph:XTR}
\end{center}
\end{subfigure}
\hfill
%
%
%
\begin{subfigure}{.3\linewidth}
\begin{center}
\begin{tikzpicture}
	\foreach \x/\y/\z/\w in \ptsXTF {
		\draw[fill = black!50] (\scale*\x,\scale*\y) circle [radius = 0.1] node[label = {[label distance = 0.05 cm]\w: $\z$}] (\z) {}; 
	}
 
\foreach \x/\y/\z in \ptsXTFUnlabel {
		\draw[fill = black!50] (\scale*\x,\scale*\y) circle [radius = 0.1] node (\z) {}; 
	}

\foreach \x/\y in \edgesXTF { \draw (\x) -- (\y); }  


\end{tikzpicture}
\subcaption{The graph $
X_{34}$} \label{graph:XTF}
\end{center}
\end{subfigure}
\hfill
%
%
%
\begin{subfigure}{.3\linewidth}
\begin{center}
\begin{tikzpicture}
	\foreach \x/\y/\z/\w in \ptsXTV {
		\draw[fill = black!50] (\scale*\x,\scale*\y) circle [radius = 0.1] node[label = {[label distance = 0.05 cm]\w: $\z$}] (\z) {}; 
	}
 
\foreach \x/\y/\z in \ptsXTVUnlabel {
		\draw[fill = black!50] (\scale*\x,\scale*\y) circle [radius = 0.1] node (\z) {}; 
	}

\foreach \x/\y in \edgesXTV { \draw (\x) -- (\y); }  


\end{tikzpicture}
\subcaption{The graph $
X_{35}$} \label{graph:XTV}
\end{center}
\end{subfigure}
\hfill
%
%
%
\begin{subfigure}{.3\linewidth}
\begin{center}
\begin{tikzpicture}
	\foreach \x/\y/\z/\w in \ptsXTX {
		\draw[fill = black!50] (\scale*\x,\scale*\y) circle [radius = 0.1] node[label = {[label distance = 0.05 cm]\w: $\z$}] (\z) {}; 
	}
 
\foreach \x/\y/\z in \ptsXTXUnlabel {
		\draw[fill = black!50] (\scale*\x,\scale*\y) circle [radius = 0.1] node (\z) {}; 
	}

\foreach \x/\y in \edgesXTX { \draw (\x) -- (\y); }  


\end{tikzpicture}
\subcaption{The graph $
X_{36}$} \label{graph:XTX}
\end{center}
\end{subfigure}
\hfill
%
%
%
\begin{subfigure}{.3\linewidth}
\begin{center}
\begin{tikzpicture}
	\foreach \x/\y/\z/\w in \ptsXTS {
		\draw[fill = black!50] (\scale*\x,\scale*\y) circle [radius = 0.1] node[label = {[label distance = 0.05 cm]\w: $\z$}] (\z) {}; 
	}
 
\foreach \x/\y/\z in \ptsXTSUnlabel {
		\draw[fill = black!50] (\scale*\x,\scale*\y) circle [radius = 0.1] node (\z) {}; 
	}

\foreach \x/\y in \edgesXTS { \draw (\x) -- (\y); }  


\end{tikzpicture}
\subcaption{The graph $
X_{37}$} \label{graph:XTS}
\end{center}
\end{subfigure}
\hfill
%
%
%
\begin{subfigure}{.3\linewidth}
\begin{center}
\begin{tikzpicture}
	\foreach \x/\y/\z/\w in \ptsXTE {
		\draw[fill = black!50] (\scale*\x,\scale*\y) circle [radius = 0.1] node[label = {[label distance = 0.05 cm]\w: $\z$}] (\z) {}; 
	}
 
\foreach \x/\y/\z in \ptsXTEUnlabel {
		\draw[fill = black!50] (\scale*\x,\scale*\y) circle [radius = 0.1] node (\z) {}; 
	}

\foreach \x/\y in \edgesXTE { \draw (\x) -- (\y); }  


\end{tikzpicture}
\subcaption{The graph $
X_{38}$} \label{graph:XTE}
\end{center}
\end{subfigure}
\hfill
%
%
%
\begin{subfigure}{.3\linewidth}
\begin{center}
\begin{tikzpicture}
	\foreach \x/\y/\z/\w in \ptsXTN {
		\draw[fill = black!50] (\scale*\x,\scale*\y) circle [radius = 0.1] node[label = {[label distance = 0.05 cm]\w: $\z$}] (\z) {}; 
	}
 
\foreach \x/\y/\z in \ptsXTNUnlabel {
		\draw[fill = black!50] (\scale*\x,\scale*\y) circle [radius = 0.1] node (\z) {}; 
	}

\foreach \x/\y in \edgesXTN { \draw (\x) -- (\y); }  


\end{tikzpicture}
\subcaption{The graph $
X_{39}$} \label{graph:XTN}
\end{center}
\end{subfigure}
\hfill
%
%
%
\begin{subfigure}{.3\linewidth}
\begin{center}
\begin{tikzpicture}
	\foreach \x/\y/\z/\w in \ptsXF {
		\draw[fill = black!50] (\scale*\x,\scale*\y) circle [radius = 0.1] node[label = {[label distance = 0.05 cm]\w: $\z$}] (\z) {}; 
	}
 
\foreach \x/\y/\z in \ptsXFUnlabel {
		\draw[fill = black!50] (\scale*\x,\scale*\y) circle [radius = 0.1] node (\z) {}; 
	}

\foreach \x/\y in \edgesXF { \draw (\x) -- (\y); }  


\end{tikzpicture}
\subcaption{The graph $
X_{40}$} \label{graph:XF}
\end{center}
\end{subfigure}
\hfill
%
%
%
\begin{subfigure}{.3\linewidth}
\begin{center}
\begin{tikzpicture}
	\foreach \x/\y/\z/\w in \ptsXFO {
		\draw[fill = black!50] (\scale*\x,\scale*\y) circle [radius = 0.1] node[label = {[label distance = 0.05 cm]\w: $\z$}] (\z) {}; 
	}
 
\foreach \x/\y/\z in \ptsXFOUnlabel {
		\draw[fill = black!50] (\scale*\x,\scale*\y) circle [radius = 0.1] node (\z) {}; 
	}

\foreach \x/\y in \edgesXFO { \draw (\x) -- (\y); }  


\end{tikzpicture}
\subcaption{The graph $
X_{41}$} \label{graph:XFO}
\end{center}
\end{subfigure}
\caption{The finite graphs appearing in Theorem \ref{thm:forbidden_subgraph_AT_free}}
\label{fig:finite_AT_graphs}
\end{figure}


\def\edCn{2}
\def\scale{0.75}

\def\ptsCn{-.433884*\edCn/-.900969*\edCn/1/180,
.433884*\edCn/-.900969*\edCn/2/0,
.974928*\edCn/-.222521*\edCn/3/0,
.781831*\edCn/.62349*\edCn/4/0,
-.781831*\edCn/.62349*\edCn/n-1/180,
-.974928*\edCn/-.222521*\edCn/n/180
}
    
\def\ptsCnUnlabel{0*\edCn/1*\edCn/5}

\def\edgesCn{
    1/2,
    1/n,
    2/3,
    3/4,
    5/n-1,
    n-1/n}



\def\edKn{1.5}

\def\ptsKn{-2*\edKn/0*\edKn/1/270,
-1*\edKn/0*\edKn/2/270,
0*\edKn/0*\edKn/3/270,
2*\edKn/0*\edKn/n+1/270,
-3*\edKn/0*\edKn/v_1/270,
3*\edKn/0*\edKn/v_2/270,
0*\edKn/1.732050*\edKn/v_3/0}
    
\def\ptsKnUnlabel{
0*\edKn/0.866025*\edKn/-1}

\def\edgesKn{
    v_1/1,
    1/2,
    2/3,
    3/4,
    4/n+1,
    n+1/v_2,
    -1/1,
    -1/2,
    -1/3,
    -1/4,
    -1/n+1,
    v_3/-1}



\def\edLn{5}

\def\ptsLn{
-3*.125*\edLn/0*\edLn/1/270,
-1*.125*\edLn/0*\edLn/2/270,
3*.125*\edLn/0*\edLn/n+1/270,
-5*.125*\edLnO/.375*\edLnO/v_1/270,
5*.125*\edLnO/.375*\edLnO/v_2/270,
0*\edLnO/6*.125*\edLnO/v_3/0}
    
\def\ptsLnUnlabel{
-2*.125*\edLn/.375*\edLn/6,
2*.125*\edLn/.375*\edLn/7}

\def\edgesLn{
    1/2,
    2/3,
    3/n+1,
    v_1/6,
    6/7,
    7/v_2,
    6/v_3,
    7/v_3,
    1/6,
    1/7,
    2/6,
    2/7,
    3/6,
    3/7,
    n+1/6,
    n+1/7,
    1/v_1,
    n+1/v_2}



\def\edLnO{5}

\def\ptsLnO{
-3*.125*\edLnO/0*\edLnO/1/270,
-1*.125*\edLnO/0*\edLnO/2/270,
3*.125*\edLnO/0*\edLnO/n+1/270,
-5*.125*\edLnO/.375*\edLnO/v_1/270,
5*.125*\edLnO/.375*\edLnO/v_2/270,
0*\edLnO/6*.125*\edLnO/v_3/0}
    
\def\ptsLnOUnlabel{
-2*.125*\edLnO/.375*\edLnO/6,
2*.125*\edLnO/.375*\edLnO/7}

\def\edgesLnO{
    1/2,
    2/3,
    3/n+1,
    v_1/6,
    7/v_2,
    6/v_3,
    7/v_3,
    1/6,
    1/7,
    2/6,
    2/7,
    3/6,
    3/7,
    n+1/6,
    n+1/7,
    1/v_1,
    n+1/v_2}


\begin{figure}[ht]
\begin{subfigure}{.45\linewidth}
\begin{center}
\begin{tikzpicture}
	\foreach \x/\y/\z/\w in \ptsCn {
		\draw[fill = black!50] (\scale*\x,\scale*\y) circle [radius = 0.1] node[label = {[label distance = 0.05 cm]\w: $\z$}] (\z) {}; 
	}
 
\foreach \x/\y/\z in \ptsCnUnlabel {
		\draw[fill = black!50] (\scale*\x,\scale*\y) circle [radius = 0.1] node (\z) {}; 
	}
 
\foreach \x/\y in \edgesCn { \draw (\x) -- (\y); }  

\draw (4) -- (5) node [pos = 0.5, sloped, above] {{\large $\mathbf{\cdots}$}};

\end{tikzpicture}
\subcaption{The graph $C_{n}$, $n\geq 6$} \label{graph:Cn}
\end{center}
\end{subfigure}
%
%
%
\begin{subfigure}{.45\linewidth}
\begin{center}
\begin{tikzpicture}
	\foreach \x/\y/\z/\w in \ptsKn {
		\draw[fill = black!50] (\scale*\x,\scale*\y) circle [radius = 0.1] node[label = {[label distance = 0.05 cm]\w: $\z$}] (\z) {}; 
	}
 
\foreach \x/\y/\z in \ptsKnUnlabel {
		\draw[fill = black!50] (\scale*\x,\scale*\y) circle [radius = 0.1] node (\z) {}; 
	}

 \draw[fill = black!50] (\scale*1*\edKn,\scale*0*\edKn) circle [radius = 0.1] node[label = {[label distance = 0.05 cm]270: $\cdots$}] (4) {};
 
\foreach \x/\y in \edgesKn { \draw (\x) -- (\y); }  


\end{tikzpicture}
\subcaption{The graph $XF_2^{n}$, $n\geq 1$} \label{graph:Kn}
\end{center}
\end{subfigure}
\hfill 
%
%
%
\begin{subfigure}{.45\linewidth}
\begin{center}
\begin{tikzpicture}
	\foreach \x/\y/\z/\w in \ptsLn {
		\draw[fill = black!50] (\scale*\x,\scale*\y) circle [radius = 0.1] node[label = {[label distance = 0.05 cm]\w: $\z$}] (\z) {}; 
	}
 
\foreach \x/\y/\z in \ptsLnUnlabel {
		\draw[fill = black!50] (\scale*\x,\scale*\y) circle [radius = 0.1] node (\z) {}; 
	}

 \draw[fill = black!50] (\scale*.125*\edLn,\scale*0*\edLn) circle [radius = 0.1] node[label = {[label distance = 0.05 cm]270: $\cdots$}] (3) {};
  
\foreach \x/\y in \edgesLn { \draw (\x) -- (\y); }  


\end{tikzpicture}
\subcaption{The graph $XF_3^{n}$, $n\geq 0$} \label{graph:Ln}
\end{center}
\end{subfigure}
%
%
\begin{subfigure}{.45\linewidth}
\begin{center}
\begin{tikzpicture}
	\foreach \x/\y/\z/\w in \ptsLnO {
		\draw[fill = black!50] (\scale*\x,\scale*\y) circle [radius = 0.1] node[label = {[label distance = 0.05 cm]\w: $\z$}] (\z) {}; 
	}
 
\foreach \x/\y/\z in \ptsLnOUnlabel {
		\draw[fill = black!50] (\scale*\x,\scale*\y) circle [radius = 0.1] node (\z) {}; 
	}

 \draw[fill = black!50] (\scale*.125*\edLnO,\scale*0*\edLnO) circle [radius = 0.1] node[label = {[label distance = 0.05 cm]270: $\cdots$}] (3) {};
  
\foreach \x/\y in \edgesLnO { \draw (\x) -- (\y); }  


\end{tikzpicture}
\subcaption{The graph $XF_4^{n}$, $n\geq 0$} \label{graph:LnO}
\end{center}
\end{subfigure}
\caption{Infinite families appearing in Theorem \ref{thm:forbidden_subgraph_AT_free}}
\label{fig:reg_graphs}
\end{figure}

\bcor
\label{cor:at_triple_implies_non_f_purity}
If $R_{G}$ is F-pure in some positive characteristic $p > 0$, then $G$ is AT-free.
\ecor

\bproof
We prove the contrapositive. Let $G$ be a graph containing an asteroidal triple. Theorem \ref{thm:forbidden_subgraph_AT_free} implies that $G$ contains one of the listed subgraph as an induced subgraph. It is readily verified that every graph appearing in the list in Theorem \ref{thm:forbidden_subgraph_AT_free} satisfies the hypotheses of Theorem \ref{thm:non_F_purity}. Consequently, the binomial edge ideal of every graph listed in Theorem \ref{thm:forbidden_subgraph_AT_free} is not F-pure in any positive characteristic. Lemma \ref{lem:F_split_descends_to_induced_subgraph} implies that $S/J_{G}$ is not F-pure in any positive characteristic.
\eproof

\bcor
Matsuda's Conjecture \ref{conj:bei_are_eventually_F_pure} is false.
\ecor 

\section{Matsuda's Weakly Closed Conjecture}
\label{sec:Matsuda_conj}

In this section, we show that to prove Conjecture \ref{conj:weakly_closed_equiv_F_pure_in_char_2}, it suffices to demonstrate that every graph appearing in Theorem \ref{thm:gallai} is not F-pure. We then show that the sporadic graphs and the graphs belonging to the regular families in Theorem \ref{thm:gallai} do not realize an F-pure binomial edge ideal in any positive characteristic. In Section \ref{sec:co_regular_not_f_pure}, we establish that the binomial edge ideals coming from the co-regular families are not F-pure.

\subsection{Weakly Closed Graphs and Gallai's Theorem}

In \cite{herzog2010binomial}, the authors gave an algebraic characterization of \textbf{closed graphs} {\cite[Theorem 1.1]{herzog2010binomial}}, a notion that had independently appeared in graph theory under the name \textbf{proper interval graphs} \cite[Theorem 2.4]{crupi2014closedProperInteval}. They showed that the generators $f_{i,j}$ form a reduced Gr\"obner basis of $J_{G}$ if and only if $G$ is a closed graph. Various equivalent characterizations of $G$ being closed are well-known.

\bthm
\label{thm:closed_graphs}
Let $G$ be a graph on the vertices $[n]$, then the following are equivalent:
\begin{enumerate}
    \item $G$ is closed,
    \item Let $i < j$ with $\{i,j\} \in E(G)$. Then, for all integers $k$ with $i < k < j$, $\{i,k\} \in E(G)$ and $\{k,j\} \in E(G)$. (See {\cite[Proposition 4.8]{crupi2010binomial}}.)
    \item $G$ is a proper interval graph. (See {\cite[Theorem  2.4]{crupi2014closedProperInteval}}.)
    \item $G$ is chordal, claw-free, net-free, and tent-free. (See {\cite[Theorem 7.10]{herzog2018binomial}}.)
\end{enumerate}
\ethm 

Matsuda introduced the notion of weakly closed graphs and showed that weakly closed graphs define F-pure ideals.

\bdefn[\cite{matsuda2018weakly}]
\label{defn:weakly_closed}
Let $G$ be a graph. We say that $G$ is \textbf{weakly closed} if and only if there exists a labeling of the vertices of $G$ such that for every edge $\{i,j\}$ of $G$ with $i < j$, it satisfies that for every integer $k$ with $i < k < j$ that $\{i,k\}$ is an edge of $G$ or $\{k,j\}$ is an edge of $G$.
\edefn 

\bthm[{\cite[Theorem 2.3]{matsuda2018weakly}}]
\label{thm:weakly_closed_are_F_pure}
Let $G$ be a graph, and let $\k$ be any field of positive characteristic. If $G$ is a weakly closed graph, then $R/J_{G}$ is F-pure.
\ethm


The forward implication of Conjecture \ref{conj:weakly_closed_equiv_F_pure_in_char_2} follows from Theorem \ref{thm:weakly_closed_are_F_pure}. However, the reverse implication has remained open.

Matsuda provided an analogue of Theorem \ref{thm:closed_graphs}, which had previously been observed by Kratsch and Stewart \cite[p. 402]{KS93}; see also \cite[Proposition 5.1]{COS93}.

\bthm[{\cite[Theorem 1.9]{matsuda2018weakly}}]
\label{thm:weakly_closed_equivalent_co_comparable}
Let $G$ be a graph. Then, $G$ is weakly closed if and only if $G$ is co-comparability (i.e., $\ol{G}$ is a comparability graph). 
\ethm

\noindent A graph is a \textbf{comparability} graph if there is a partial order on the vertices such that edges exactly correspond to distinct comparable vertices in the partial order. 
 Matsuda remarks that comparability, and hence co-comparability, graphs are characterized by excluding a minimal list of induced subgraphs. This is the celebrated result of Gallai \cite{gallai1967transitiv}. A good presentation of Gallai's result is given in Trotter \cite{trotter1992combinatorics}. A readily accessible online presentation of the forbiden subgraph characterization of co-comparability graphs can be found at \cite{isgci}. In Theorem \ref{thm:gallai} below, we follow the notation for the graphs as they appear in \cite{isgci} for ease of reference; see Figures~\ref{fig:finite_AT_graphs},  \ref{fig:reg_graphs}, \ref{graph:Jn}, \ref{graph:JnO}, and \ref{graph:JnT}.

\bthm[{Gallai's Theorem \cite{gallai1967transitiv}, \cite[Theorem 2.1]{trotter1992combinatorics}},\cite{isgci}]
\label{thm:gallai}
Let $G$ be a graph. Then, $G$ is co-comparability if and only if $G$ does not contain any of the following graphs as induced subgraphs:
\begin{enumerate}
    \item the sporadic graphs: 
        \begin{multicols}{3}
            \begin{enumerate}
                \item $T_{2}$
                \item $X_{31}$
                \item $X_{2}$
                \item $X_{30}$
                \item $X_{3}$
                \item $X_{32}$
                \item $X_{33}$
                \item $X_{35}$
                \item $X_{34}$
                \item $X_{36}$
                \item[] 
                \item[]
            \end{enumerate}
        \end{multicols}
    \item the regular families:
        \begin{enumerate}
            \item $C_{n}$ for $n\geq 6$
            \item $XF_{2}^{n+1}$ for $n \geq 1$
            \item $XF_{3}^{n+1}$ for $n \geq 0$ 
            \item $XF_{4}^{n+1}$ for $n \geq 0$\\
        \end{enumerate}
    \item the co-regular families:
        \begin{enumerate}
            \item the complement of $C_{2n+1}$ for $n \geq 2$ (odd anti-hole)
            \item $\mathrm{co\!-\!XF}_{1}^{2n+3}$ for $n \geq 0$ 
            \item $\mathrm{co\!-\!XF}_{5}^{2n+3}$ for $n \geq 0$
            \item $\mathrm{co\!-\!XF}_{6}^{2n+2}$ for $n \geq 0$ 
        \end{enumerate}
\end{enumerate}
\ethm

\bprop 
\label{prop:reduction_matsuda_to_finite_list}
If every graph belonging to the list appearing in Theorem \ref{thm:gallai} is not F-pure, then Matsuda's Conjecture \ref{conj:weakly_closed_equiv_F_pure_in_char_2} holds.
\eprop 

\bproof
Mastuda's Theorem \ref{thm:weakly_closed_are_F_pure} implies that every weakly closed graph defines an F-pure ring in any characteristic, and hence, in particular, in characteristic $2$. To prove the converse of Conjecture \ref{conj:weakly_closed_equiv_F_pure_in_char_2}, it suffices to show that if a graph $G$ is not weakly closed, then $R_{G}$ is not an F-pure ring in characteristic $2$. By Theorems \ref{thm:weakly_closed_equivalent_co_comparable} and \ref{thm:gallai}, we have that if $G$ is not weakly closed, then $G$ contains as an induced subgraph $H$ where $H$ belongs to the list of graphs in Theorem \ref{thm:gallai}. By Lemma \ref{lem:F_split_descends_to_induced_subgraph}, if $R_{H}$ is not an F-pure ring in characteristic $2$, then $R_{G}$ cannot be an F-pure ring in characteristic $2$ either.
\eproof

We can now prove our first main result, delaying until the later sections some technical computations.

\bproof[Proof of Theorem \ref{thm:matsudas_conjecture_is_true}]
The graphs appearing in section (1) and (2) of Theorem \ref{thm:gallai} (the ``sporadic graphs" and ``regular families", respectively)  are not F-pure in any positive characteristic, since each of these graphs contain an asteroidal triple (Theorem \ref{thm:forbidden_subgraph_AT_free}, Corollary \ref{cor:at_triple_implies_non_f_purity}). 

Theorems \ref{thm:Cn_not_F_pure}, \ref{thm:Jn_not_F_pure}, \ref{thm:Jn1_not_F_pure}, and \ref{thm:Jn2_not_F_pure} establish that the graphs appearing in section (3) of Theorem \ref{thm:gallai} (the ``co-regular families") are not F-pure in characteristic $2$. Matsuda's conjecture \ref{conj:weakly_closed_equiv_F_pure_in_char_2} now follows from Proposition \ref{prop:reduction_matsuda_to_finite_list}.
\eproof

Proving that the co-regular families define non-F-pure binomial edge ideals in characteristic 2 will occupy us for most of the remainder of this paper.  First let us point out that our results completely classify the F-purity of all chordal graphs in every characteristic.

\begin{cor}
\label{cor:char_F_pure_bei}
    Let $G$ be a chordal graph.  The following are equivalent:
    \begin{enumerate}
     \item $R_G$ is F-pure in all characteristics $p > 0$.
    \item $R_G$ is F-pure in some characteristic $p > 0$.
    \item $G$ is AT-free.
    \item $G$ is weakly closed.
    \end{enumerate}
\end{cor}

\begin{proof}
    $\mathit{(1) \Rightarrow (2)}$ is obvious.  $\mathit{(2) \Rightarrow (3)}$ follows from Corollary~\ref{cor:at_triple_implies_non_f_purity}.  $\mathit{(3) \Rightarrow (4)}$ follows from Gallai's Theorem~\ref{thm:gallai} and Matsuda's Theorem~\ref{thm:weakly_closed_equivalent_co_comparable}, since chordal graphs are AT-free if and only if they are co-comparability \cite{isgciATFreeANDChordal}, while $\mathit{(4) \Rightarrow (1)}$ follows from Matsuda's Theorem~\ref{thm:weakly_closed_are_F_pure}.
\end{proof}

The five cycle shows that the chordal hypothesis cannot be removed.  The authors thank Matthew Mastroeni for pointing out this observation to us.  We also note that by \cite[Theorem 3]{LB62}, chordal AT-free graphs are exactly intervals graphs.

\section{Computation of Colon Ideals}
\label{sec:computation_colon_ideals}

In this section, we introduce two families of ideals that will appear among the minimal primes of the co-regular families of binomial edge ideals. Then we compute various colon ideals that are relevant for the application of Fedder's criterion. Of the results in this section, only Theorem~\ref{thrm:colon_minl_prime} will be used in later sections; the remaining technical statements serve solely to establish this theorem.

\nota 
Let $A$ and $B$ be disjoint subsets of $[n]$. As above, we denote by $P(A,B)$ the ideal
\begin{align*}
(\{x_{i},y_{i}\}_{i\in A}) + ( \{f_{i,j} \}_{i,j \in B} )
\end{align*}
where $f_{i,j} := x_{i}y_{j} - x_{j}y_{i}$. We define the monomial $\omega_{A}$ as 
\begin{align*}
\omega_{A} := \prod_{i \in A} x_{i} y_{i}.
\end{align*}
When $A = \varnothing$, we utilize the convention that $\omega_{A} = 1$. When $B = \{a,c,d\}$ with $a$, $c$, and $d$ distinct integers, we define the ideal $L_{P}$ as
\begin{align*}
    L_{P} := (f_{a,c}f_{a,d}\, \omega_{A},f_{a,c}f_{c,d}\, \omega_{A},f_{a,d}f_{c,d} \, \omega_{A}).
\end{align*}

We need the following preparatory lemma.
\blem
\label{lem:descent}
Fix an integer $n$ and integer $b \in [n]$. Let $A$ and $B$ be disjoint subsets of $[n]$. We denote by $P$ the prime ideal $P(A,B)$. Pick $k \in A$ and $r \in \k[x_{i},y_{i} \mid i \in [n] \setminus \{k\}]$. If 
\begin{align*}
    \alpha \cdot r \in P^{[2]}
\end{align*}
for some $\alpha \in \{x_{k}, y_{k}, x_{k}y_{k}\}$, then 
\begin{align*}
    r \in P(A\setminus \{k\},B)^{[2]}.
\end{align*}
If $b \notin A \cup B$ and 
\begin{align*}
    f_{k,b} \cdot r \in P^{[2]},
\end{align*}
then
\begin{align*}
    r \in P(A\setminus \{k\},B)^{[2]}.
\end{align*}
\elem 

\bproof
Since the element $\alpha \cdot r$ is not divisible by $x_{k}^{2}$ or by $y_{k}^{2}$, it follows that 
\begin{equation*}
    \label{eqn:descent_0}
    \alpha \cdot r \in P(A\setminus \{k\},B)^{[2]}.
\end{equation*}
This still holds after evaluating $x_{k}$ and $y_{k}$ at one. When $\alpha \in \{x_{k},y_{k},x_{k}y_{k}\}$, the lemma follows. When $\alpha = f_{k,b}$, we obtain that 
\begin{align*}
    (x_{b} + y_{b}) \cdot r \in P(A\setminus \{k\},B)^{[2]}.
\end{align*}
Since $x_{b} + y_{b}$ is a regular element on $P(A\setminus \{k\},B)^{[2]}$, it follows that $r \in P(A\setminus \{k\},B)^{[2]}$, which completes the proof.
\eproof

\blem
\label{lem:colon_computation_T_empty}
Fix a positive integer $n$ and distinct integers $a$, $b$, and $c$ belonging to $[n]$. Let $A \subset [n] \setminus \{a,b,c\}$, and let $B \subset [n] \setminus (A \cup \{b\})$. Denote by $P$ the ideal $P(A,B)$. Let $G$ be a graph satisfying the following conditions:
\begin{enumerate}
\item $A \subset V(G) \subset [n]$,
\item For all $k \in A$, $\{k,b\} \in E(G)$,
\item For all $k \in A$, $\{k,a\} \in E(G)$ or $\{k,c\} \in E(G)$,
\item For every edge $e$ of $G$, $e = \{k,j\}$ for some vertex $k \in A$ and $j \in V(G) \setminus \{k\}$.
\end{enumerate}
Then, we have that
\begin{equation}
\label{eqn:colon_ideal_of_minl_prime_one_comp_w_graph_0}
P^{[2]}:J_{G} = P^{[2]} + (\omega_{A}).
\end{equation}
\elem 

\bproof
We first prove the reverse inclusion. It suffices to prove for every edge $e$ of $G$ that $f_{e} \, \omega_{A} \in P^{[2]}$. Since $e$ is of the form  $\{k,j\}$ for some vertex $k \in A$ and $j \in V(G) \setminus \{k\}$, it follows that $f_{k,j} \omega_{A} \in (x_{k}^{2},y_{k}^{2}) \subset P^{[2]}$.

We prove the forward inclusion by induction on $\#A$. If $\#A = 0$, then $\omega_{A} = 1$ and $J_{G} = (0)$ by condition (4), in which case equation 
\eqref{eqn:colon_ideal_of_minl_prime_one_comp_w_graph_0} holds. Hence, we may suppose that $\#A > 0$, and that equation \eqref{eqn:colon_ideal_of_minl_prime_one_comp_w_graph_0} holds for all graphs $G'$ satisfying conditions (1), (2), (3), and (4) and for all prime ideals $P(A',B')$ as in the statement of the lemma whenever $\# A' < \# A$.  

Pick $r \in P^{[2]}:J_{G}$. After subtracting from $r$ those monomials belonging to the support of $r$ and contained in the ideal $( \{x_{j}^{2},y_{j}^{2} \mid j \in A\})$, we may suppose that every monomial belonging to the support of $r$ is not contained in $P^{[2]}$. Next, pick $k \in A$. Since every monomial in the support of $r$ is not divisible by $x_{k}^{2}$ or $y_{k}^{2}$, we may express $r$ as:
\begin{equation}
    \label{eqn:colon_ideal_of_minl_prime_one_comp_w_graph_1}
    r = x_{k}y_{k} r_{1}  + x_{k}r_{2} + y_{k} r_{3} + r_{4}
\end{equation}
for some $r_{1}, r_{2}, r_{3}, r_{4} \in \k[x_{i},y_{i} \mid i \in [n] \setminus \{k\}]$.
The ideals $J_{G}$ and $P^{[2]}$ are both $\N^{n}$-graded ideals, where $n = \# V(G)$ and $\deg(x_{i}) = \deg(y_{i}) = e_{i}$ is the $i$-th standard basis element. Hence, the ideal $P^{[2]}:J_{G}$ is also $\N^{n}$-graded. Thus, for $s \in \{x_{k}y_{k} r_{1}, x_{k}r_{2} + y_{k} r_{3}, r_{4} \}$, we have that
\begin{align*}
    s \in P^{[2]}:J_{G}.
\end{align*}
It now suffices to show that each element of $\{x_{k}y_{k} r_{1}, x_{k}r_{2} + y_{k} r_{3}, r_{4} \}$ belongs to $P^{[2]} + (\omega_{A})$.

First, we show that $x_{k}y_{k} r_{1} \in P^{[2]} + (\omega_{A})$. By applying Lemma \ref{lem:descent} to $x_{k}y_{k}r_{1}\cdot f_{i,j}$ for every edge $\{i,j\} \in E(G \setminus \{k\})$, it follows that 
\begin{align*}
    r_{1} \in P(A\setminus\{k\},B)^{[2]}:J_{G\setminus k}.
\end{align*}
Since $G \setminus k$, $A' := A\setminus k$, and $B' := B$ satisfy the hypotheses of Lemma \ref{lem:colon_computation_T_empty}, it follows by induction hypothesis that 
\begin{align*}
    r_{1} \in P(A\setminus \{k\},B)^{[2]} + (\omega_{A\setminus \{k\}}).
\end{align*}
Thus, $x_{k}y_{k}r_{1} \in P^{[2]} + (\omega_{A})$. 

Next, we show that $r_{4} \in P^{[2]} + (\omega_{A})$. We observe that 
\begin{align*}
    r_{4} \cdot f_{k,b}  \in P^{[2]}.
\end{align*}
By Lemma \ref{lem:descent}, it follows that $r_{4} \in P^{[2]}$.

Finally, we show that $x_{k}r_{2} + y_{k} r_{3} \in P^{[2]} + (\omega_{A})$. We consider the element 
\begin{align*}
    (x_{k}r_{2} + y_{k} r_{3})\cdot f_{k,b} \in P^{[2]}.
\end{align*}
After expanding this product and subtracting terms divisible by $x_{k}^{2}$ or by $y_{k}^{2}$, we obtain that 
\begin{align*}
    x_{k}y_{k} (x_{b}r_{2} + y_{b}r_{3}) \in P^{[2]}.
\end{align*}
Lemma \ref{lem:descent} implies that 
\begin{align*}
    x_{b}r_{2} + y_{b}r_{3} \in P(A\setminus \{k\},B)^{[2]}.
\end{align*}
Write 
\begin{equation}
\label{eqn:colon_computation_T_empty_mixed_case_0}
\begin{aligned}
    r_{2} &= y_{b}r_{2}' + r_{2}'' \\
    r_{3} &= x_{b}r_{3}' + r_{3}''
\end{aligned}
\end{equation}
where every monomial in the support of $r_{2}''$ (respectively $r_{3}''$) is not divisible by $y_{b}$ (respectively $x_{b}$). Hence, we have that 
\begin{equation}
    \label{eqn:colon_computation_T_empty_mixed_case}
    x_{b}y_{b}(r_{2}' + r_{3}') + x_{b}r_{2}'' + y_{b}r_{3}'' \in P(A \setminus \{k\},B)^{[2]}.
\end{equation}
In  \eqref{eqn:colon_computation_T_empty_mixed_case}, setting $x_{b} = 0$ implies that 
\begin{align*}
    y_{b}r_{3}'' \in P(A\setminus \{k\},B)^{[2]}.
\end{align*}
Since $y_{b}$ is a regular element on $P(A\setminus \{k\},B)^{[2]}$, it follows that 
\begin{equation}
    \label{eqn:colon_computation_T_empty_mixed_case_2}
    r_{3}'' \in P(A\setminus \{k\},B)^{[2]}.
\end{equation}
Likewise, setting $y_{b} = 0$ in  \eqref{eqn:colon_computation_T_empty_mixed_case} and using that $x_{b}$ is a regular element on $P(A\setminus \{k\},B)^{[2]}$, it follows that 
\begin{equation}
    \label{eqn:colon_computation_T_empty_mixed_case_3}
    r_{2}'' \in P(A\setminus \{k\},B)^{[2]}.
\end{equation}
Combining \eqref{eqn:colon_computation_T_empty_mixed_case}, \eqref{eqn:colon_computation_T_empty_mixed_case_2}, and \eqref{eqn:colon_computation_T_empty_mixed_case_3} together with the fact that $x_{b}y_{b}$ is a regular element on $P(A\setminus \{k\},B)^{[2]}$ implies
\begin{equation}
    \label{eqn:colon_computation_T_empty_mixed_case_4}
    r_{2}' + r_{3}' \in P(A\setminus \{k\},B)^{[2]}.
\end{equation}
From \eqref{eqn:colon_computation_T_empty_mixed_case_0}, \eqref{eqn:colon_computation_T_empty_mixed_case_2},  \eqref{eqn:colon_computation_T_empty_mixed_case_3}, and \eqref{eqn:colon_computation_T_empty_mixed_case_4} we have that after subtracting terms belonging to $P(A\setminus \{k\},B)^{[2]}$ that we may assume that
\begin{align*}
    r_{2} &= y_{b} r_{2}' \\
    r_{3} &= -x_{b} r_{2}'.
\end{align*}
Without loss of generality, we may assume by condition (3) that $\{a,k\} \in E(G)$. It follows that
\begin{align*}
     (x_{k} r_{2} + y_{k} r_{3}) f_{a,k} &\in P^{[2]} \\
    \implies  (x_{k} y_{b} r_{2}' - y_{k} x_{b} r_{2}') f_{a,k} &\in P^{[2]} \\
    \implies x_{k}y_{k} (x_{a} y_{b} - y_{a} x_{b}) r_{2}' &\in P^{[2]}.
\end{align*}
Since $f_{a,b}$ is a regular element on $P^{[2]}$, it follows that
\begin{align*}
    x_{k}y_{k} r_{2}' \in P^{[2]}.
\end{align*}
Lemma \ref{lem:descent} implies that $r_{2}' \in P(A \setminus \{k\}, B)^{[2]} \subset P^{[2]}$. It thus follows that $r_{2}$ and $r_{3}$ belong to $P^{[2]}$, which completes the proof.
\eproof

\blem 
\label{lem:intersection_ideals}
Fix an integer $n$. Let $A$ and $B$ be disjoint subsets of $[n]$. We denote by $P$ the ideal $P(A,B)$ of $R := \k[x_{i},y_{i} \mid i \in [n]]$. We denote by $Q$ any ideal of $R$ which can be realized as the extension of an ideal belonging to $\k[x_{i},y_{i} \mid i \in B]$. Then,
\begin{align*}
    \left( P^{[2]} + Q \right) \cap \left( P^{[2]} + (\omega_{A}) \right) = P^{[2]} + Q \cdot (\omega_{A}).
\end{align*}
\elem 

\bproof
We prove the forward inclusion, as the reverse inclusion is clear. Let $r \in P^{[2]} + (\omega_{A})$. Then we may express $r$ as $r = r_{1} + p \, \omega_{A}$, with $p \in \kk[x_{j},y_{j} \mid j \in [n] \setminus A]$ and $r_{1} \in P^{[2]}$. Since $r \in P^{[2]} + Q$, it follows that $p \, \omega_{A} \in P^{[2]} + Q$. For degree reasons, it follows that
\begin{align*}
    p \, \omega_{A} \in P(\varnothing,B)^{[2]} + Q.
\end{align*}
Since $\omega_{A}$ is a regular element on $P(\varnothing,B)^{[2]} + Q$, it follows that $p \in P^{[2]} + Q$, which completes the proof.
\eproof

\blem 
\label{lem:reduction_of_colon_computation}
Fix an integer $n$. Let $A$ and $B$ be disjoint subsets of $[n]$. We denote by $P$ the ideal $P(A,B)$ of $R := \k[x_{i},y_{i} \mid i \in [n]]$. Let $G$ be any graph such that $V(G) \subset B$. Let $J_{G}$ denote the extension of the binomial edge ideal of $G$ to $R$. Then,
\begin{align*}
    P^{[2]}:J_{G} = P^{[2]} + P(\varnothing,B)^{[2]}:J_{G}.
\end{align*}
\elem 

\bproof
We prove the forward inclusion, as the reverse inclusion is clear. Let $r \in P^{[2]}:J_{G}$. Subtract from $r$ all monomials belonging to $P(A,\varnothing)^{[2]}$. Since $V(G) \subset B$, it follows for degree reasons that for all $\{i,j\} \in E(G)$ that $r\cdot f_{i,j} \in P(\varnothing,B)^{[2]}$, which completes the proof.
\eproof

\blem
\label{lem:colon_three_binomials}
Let $a$, $c$, and $d$ denote distinct positive integers. Let $R := \k[x_{i},y_{i} \mid i \in \{a,c,d\}]$, then
\begin{equation}
    \label{eqn:colon_three_binomials_1}
    (f_{a,c}^{2}):(f_{a,c}) = (f_{a,c}),
\end{equation}
and 
\begin{equation}
    \label{eqn:colon_three_binomials_2}
    (f_{a,c},f_{a,d},f_{c,d})^{[2]} : (f_{a,c},f_{c,d}) = (f_{a,c},f_{a,d},f_{c,d})^{[2]} + (f_{a,c}f_{a,d}, f_{a,c}f_{c,d}, f_{a,d}f_{c,d}).
\end{equation}
\elem 

\bproof
The proof of equation \eqref{eqn:colon_three_binomials_1} is clear. The proof of equation \eqref{eqn:colon_three_binomials_2} can be verified via the computer program Macaulay2.
\eproof

\brem
Macaulay2 computations show that Equation \eqref{eqn:colon_three_binomials_2} need not hold when $\fchar(\k) \neq 2$.
\erem 

\bthm
\label{thrm:colon_minl_prime}
Fix a positive integer $n$ and distinct integers $a$, $b$, $c$, and $d$ belonging to $[n]$. Let $A \subset [n] \setminus \{a,b,c,d\}$. Let $B$ be either the set $\{a,c\}$ or the set $\{a,c,d\}$. Denote by $P$ the ideal $P(A,B)$. Let $G$ be a graph satisfying the following conditions:
\begin{enumerate}
\item $A \subset V(G) \subset [n]$,
\item For all $k \in A$, $\{k,b\} \in E(G)$,
\item For all $k \in A$, $\{k,a\} \in E(G)$ or $\{k,c\} \in E(G)$,
\item $\{a,c\} \in E(G)$ and $\{c,d\} \in E(G)$,
\item For every edge $e$ of $G$, $e$ has one of the following forms: $e = \{a,c\}$, $e = \{c,d\}$, or $e = \{k,j\}$ for some vertex $k \in A$ and $j \in V(G) \setminus \{k\}$.
\end{enumerate}
If $B = \{a,c\}$, then
\begin{equation}
\label{eqn:colon_minl_prime_1}
P^{[2]}:J_{G} = P^{[2]} + (f_{a,c} \, \omega_{A}).
\end{equation}
If $B = \{a,c,d\}$, then
\begin{equation}
\label{eqn:colon_minl_prime_2}
P^{[2]}:J_{G} = P^{[2]} + (f_{a,c}f_{a,d} \, \omega_{A}, f_{a,c}f_{c,d} \, \omega_{A}, f_{a,d}f_{c,d} \, \omega_{A}).
\end{equation}
\ethm 

\bproof
We will prove the case that $B = \{a,c,d\}$ since the proof of the case where $B = \{a,c\}$ is similar. Consider the following subgraphs of $G$: the graph $H_{1}$ obtained as the induced subgraph of $G$ on the vertices $\{a,c,d\}$; and $H_{2}$, the graph obtained from $G$ after deleting the edges $E(H_{1})$. Since the edges of $H_{1}$ and $H_{2}$ cover the edges of $G$, it follows from elementary properties of colon ideals that
\begin{align*}
    P^{[2]}:J_{G} = \left( P^{[2]}:J_{H_{1}} \right) \cap \left( P^{[2]}:J_{H_{2}} \right).
\end{align*}
Lemmas \ref{lem:reduction_of_colon_computation} and \ref{lem:colon_three_binomials} imply that 
\begin{align*}
    P^{[2]}:J_{H_{1}} = P^{[2]} + (f_{a,c}f_{a,d}, f_{a,c}f_{c,d}, f_{a,d}f_{c,d}).
\end{align*}
Lemma \ref{lem:colon_computation_T_empty} implies that
\begin{align*}
    P^{[2]}:J_{H_{2}} = P^{[2]} + (\omega_{A}).
\end{align*}
The claim now follows from Lemma \ref{lem:intersection_ideals}.
\eproof

\section{Cycling and Preliminary Computations}
\label{sec:cycling}

In this section, we establish a series of technical lemmas that will prove useful in computing $J_{G}^{[2]}:J_{G}$ in Section \ref{sec:co_regular_not_f_pure}, when $G$ belongs to one of the co-regular families of graphs.

We will utilize the following setup for the remainder of this section.
\begin{setup}
Given $n$ a positive integer, we set $R = \Z/2\Z[x_{i},y_{i} \mid i \in [n]]$. We denote by $m$ the following monomial
\begin{align*}
    m := \prod_{i=1}^{n} x_{i}y_{i}.
\end{align*}
\end{setup}

The following ``cycling'' lemma will be crucial for us.

\blem[Cycling]
\label{lem:cycling}
Let $n$ be a positive integer. Let $a$ and $c$ be distinct integers in $[n]$. Put $B := \{a,c\}$, and choose $A \subset [n] \setminus \{a,c\}$. Let $r \in P(A,B)^{[2]} + (f_{a,c} \, \omega_{A})$. Let $p$ and $q$ be relatively prime squarefree monomials such that $p$ and $q$ are not divisible by $x_{j}$ or by $y_{j}$ for all $j \in A \cup \{a,c\}$. Then,
\begin{align*}
\frac{x_{a}}{y_{a}} \cdot \frac{p}{q} \cdot m \in \Supp(r) \quad (\text{respectively, } \frac{y_{a}}{x_{a}} \cdot \frac{p}{q} \cdot m \in \Supp(r)),
\end{align*}
if and only if 
\begin{align*}
\frac{x_{c}}{y_{c}} \cdot \frac{p}{q} \cdot m \in \Supp(r) \quad (\text{respectively, } \frac{y_{c}}{x_{c}} \cdot \frac{p}{q} \cdot m \in \Supp(r)).
\end{align*}
\elem 

\bproof
Since $r \in P(A,B)^{[2]} + (f_{a,c} \, \omega_{A})$, we may express $r$ as:
\begin{align}
    \label{eqn:f_in_P_i_i+2}
    r = \sum_{j \in A} (\alpha_{j} \cdot x_{j}^{2} + \beta_{j} \cdot y_{j}^{2}) + \gamma \cdot f_{a,c}^{2} + \delta \cdot f_{a,c} \, \omega_{A}
\end{align}
for some $\alpha_{j},\beta_{j},\gamma,\delta \in R$ and all $j \in A$. The monomial $w := \frac{x_{a}}{y_{a}} \cdot \frac{p}{q} \cdot m$ is not divisible by $x_{j}^{2}$ or by $y_{j}^{2}$ for all $j \in A \cup \{c\}$, since $p$ is not divisible by these variables. It follows that 
\begin{align*}
    w := \frac{x_{a}}{y_{a}} \cdot \frac{p}{q} \cdot m \in \Supp(\delta \cdot f_{a,c} \, \omega_{A}).
\end{align*}
Since $w$  is not divisible by $y_{a}$, it follows that there exists $\delta' \in \Supp(\delta)$ such that 
\begin{align*}
    \frac{x_{a}}{y_{a}} \cdot \frac{p}{q} \cdot m = \delta' \cdot  x_{a}y_{c} \, \omega_{A}.
\end{align*}
Consequently, $\delta' = \frac{1}{y_{a}y_{c}} \cdot \frac{p}{q} \cdot \frac{m}{\omega_{A}}$. The assumption that $q$ is not divisible by $x_{j}$ or by $y_{j}$ for all $j \in A \cup \{a,c\}$ implies that $\delta'$ is a bona fide monomial. We compute that
\begin{align}
    \label{eqn:new_monomial}
    \delta' \cdot f_{a,c} \, \omega_{A}
    = \frac{x_{a}}{y_{a}} \cdot \frac{p}{q} \cdot m + \frac{x_{c}}{y_{c}} \cdot \frac{p}{q} \cdot m.
\end{align}
The monomial $w' := \frac{x_{c}}{y_{c}} \cdot \frac{p}{q} \cdot m$ appearing in equation \eqref{eqn:new_monomial} is not divisible by $x_{j}^{2}$ or by $y_{j}^{2}$ for $j \in A \cup \{a\}$, since $p$ is not divisible by these variables. Consequently, the computation in Equation \eqref{eqn:new_monomial} is the unique way to obtain $w'$ as a term appearing in the expansion in Equation \eqref{eqn:f_in_P_i_i+2}. Thus, the coefficient of $w$ in $r$ is equal to the coefficient of $w'$ in $r$. Thus, $w \in \Supp(r)$ if and only if $w' \in \Supp(r)$.
\eproof

\blem
\label{lem:Jn_m_in_Supp_f}
Let $n \ge 3$ be an integer. Let $a$, $c$, and $d$ be distinct integers in $[n]$. Set $B := \{a,c,d\}$, and choose $A \subset [n] \setminus \{a,c,d\}$. Set $P := P(A,B)$, and let $r \in P^{[2]} + L_{P}$, where $L_{P} := (f_{a,c}f_{a,d}\, \omega_{A},f_{a,c}f_{c,d}\, \omega_{A},f_{a,d}f_{c,d} \, \omega_{A})$. Define the monomials 
\begin{align*}
    \alpha := \frac{x_{a}y_{c}}{y_{a}x_{c}} \cdot m  \quad \text{ and } \quad
    \beta := \frac{y_{a}x_{c}}{x_{a}y_{c}} \cdot m.
\end{align*}
Then,
\begin{align*}
    \card{\{m,\alpha,\beta\} \cap \Supp(r) } \equiv 0 \mod{2}.
\end{align*}
\elem

\bproof
Since $r \in P^{[2]} + L_{P}$, we have that
\begin{align*}
r = \gamma + \mu \cdot f_{a,c}^{2} + \mu' \cdot f_{a,c} f_{a,d} \, \omega_{A} + \mu'' \cdot f_{a,c}  f_{c,d} \, \omega_{A} + \mu''' \cdot f_{a,d}  f_{c,d} \, \omega_{A},
\end{align*}
where $\gamma \in P(A,\varnothing)^{[2]} + (f_{a,d}^{2},f_{c,d}^{2})$ and $\mu$, $\mu'$, $\mu''$, and $\mu'''$ are polynomials in $R$. We observe that $m \in \Supp(r)$ implies
\begin{equation}
    \label{eqn:Jn_m_in_Supp_f}
    m \in \Supp(\mu' \cdot f_{a,c} f_{a,d} \, \omega_{A} + \mu'' \cdot f_{a,c}  f_{c,d} \, \omega_{A} + \mu''' \cdot f_{a,d}  f_{c,d} \, \omega_{A}).
\end{equation}
After expanding out the binomials in equation \eqref{eqn:Jn_m_in_Supp_f} and comparing terms, we can determine the monomials belonging to $\mu'$, $\mu''$, and $\mu'''$ for which equation \eqref{eqn:Jn_m_in_Supp_f} is satisfied. We find that for $\tilde{\mu} \in \Supp(\mu')$
\begin{align*}
    m \in \Supp(\tilde{\mu} \cdot f_{a,c} f_{a,d} \, \omega_{A})
\end{align*}
 if and only if $\tilde{\mu} \in \{\mu_{2},\mu_{3}\}$, where 
\begin{align*}
    \mu_{2} := \frac{1}{x_{a}y_{a}y_{c}x_{d}} \cdot \frac{m}{\omega_{A}} \quad \text{ and } \quad \mu_{3} := \frac{1}{x_{a}y_{a}x_{c}y_{d}}\cdot \frac{m}{\omega_{A}}.
\end{align*}
Similarly, for $\tilde{\mu} \in \Supp(\mu'')$, we have that
\begin{align*}
    m \in \Supp(\tilde{\mu} \cdot f_{a,c}\cdot f_{c,d} \cdot \omega_{A})
\end{align*}
if and only if $\tilde{\mu} \in \{\mu_{4},\mu_{5}\}$ where 
\begin{align*}
    \mu_{4} := \frac{1}{x_{c}y_{c}y_{a}x_{d}}\cdot \frac{m}{\omega_{A}} \quad \text{ and } \quad \mu_{5} := \frac{1}{x_{c}y_{c}x_{a}y_{d}}\cdot \frac{m}{ \omega_{A}}.
\end{align*}
Finally, for $\tilde{\mu} \in \Supp(\mu''')$, we have that
\begin{align*}
    m \in \Supp(\tilde{\mu} \cdot f_{a,d}\cdot f_{c,d} \cdot \omega_{A})
\end{align*}
if and only if $\tilde{\mu} \in \{\mu_{6},\mu_{7}\}$, where
\begin{align*}
    \mu_{6} := \frac{1}{x_{d}y_{d}y_{a}x_{c}}\cdot \frac{m}{\omega_{A}} \quad \text{ and } \quad \mu_{7} := \frac{1}{x_{d}y_{d}x_{a}y_{c}}\cdot \frac{m}{\omega_{A}}.
\end{align*}
Next, we compute that 
\begin{equation}
\label{eqn:exp_prod_one}
\mu_{2} \cdot f_{a,c}f_{a,d} \, \omega_{A} = \left( 
1 + \frac{x_{a}y_{d}}{y_{a}x_{d}} + \frac{x_{c}y_{d}}{y_{c}x_{d}} + \frac{x_{c}y_{a}}{y_{c}x_{a}},
\right) \cdot m
\end{equation}
and that
\begin{equation}
\label{eqn:exp_prod_two}
\mu_{3} \cdot f_{a,c}f_{a,d} \, \omega_{A} = \left(
1 + \frac{x_{a}y_{c}}{y_{a}x_{c}} + \frac{x_{d}y_{c}}{y_{d}x_{c}} + \frac{x_{d}y_{a}}{y_{d}x_{a}}
\right) \cdot m.
\end{equation}
The polynomials $\mu_{i} \cdot f_{a,c}f_{c,d} \, \omega_{A}$ and $\mu_{j} \cdot f_{a,d}f_{c,d} \, \omega_{A}$ can be computed for $i = 4,5$ and $j = 6,7$ by an appropriate change of variables applied to equations \eqref{eqn:exp_prod_one} and \eqref{eqn:exp_prod_two}.  

A straightforward computation shows that if 
\begin{align*}
    \frac{x_{a}y_{c}}{y_{a}x_{c}} \cdot m \in \Supp(\tilde{\mu}\cdot f_{i,j}f_{i,k} \, \omega_{A})
\end{align*}
for $\tilde{\mu}$ a monomial and $i,j,k \in \{a,c,d\}$ distinct, then $\tilde{\mu} \in \{\mu_{3},\mu_{4},\mu_{6}\}$. Likewise, if 
\begin{align*}
    \frac{y_{a}x_{c}}{x_{a}y_{c}} \cdot m \in \Supp(\tilde{\mu} \cdot f_{i,j}f_{i,k} \, \omega_{A})
\end{align*}
for $\tilde{\mu}$ a monomial and $i,j,k \in \{a,c,d\}$ distinct, then $\tilde{\mu} \in \{\mu_{2},\mu_{5},\mu_{7}\}$. 

For $\nu = \alpha$ or $\nu = \beta$, we have that $\nu \in \Supp(r)$ if and only if
\begin{align*}
    \nu \in \Supp(\mu \cdot f_{a,c}^{2} + \mu' \cdot f_{a,c} f_{a,d} \, \omega_{A} + \mu'' \cdot f_{a,c}  f_{c,d} \, \omega_{A} + \mu''' \cdot f_{a,d}  f_{c,d} \, \omega_{A}).
\end{align*}
We have characterized in the preceding paragraph necessary and sufficient conditions on $\mu'$, $\mu''$, and $\mu'''$ for which 
\begin{align*}
    \nu \in \Supp(\mu' \cdot f_{a,c} f_{a,d} \, \omega_{A} + \mu'' \cdot f_{a,c}  f_{c,d} \, \omega_{A} + \mu''' \cdot f_{a,d}  f_{c,d} \, \omega_{A}).
\end{align*}
A similar computation shows that
\begin{align*}
    \nu \in \Supp(\mu\cdot f_{a,c}^{2})
\end{align*}
if and only if
\begin{align*}
    \mu_{1} := \frac{1}{x_{a}y_{a}x_{c}y_{c}}\cdot m \in \Supp(\mu).
\end{align*}

Define the sets 
\begin{align*}
    U_{m} &:= \{i \mid \mu_{i} \in \Supp(\mu'), i = 2,3 \} \cup \{i \mid \mu_{i} \in \Supp(\mu''), i = 4,5 \} \\
    &\phantom{:=} \cup \{i \mid \mu_{i} \in \Supp(\mu'''), i = 6,7 \} \\
    U_{\alpha} &:= (\{3,4,6\} \cap U_{m}) \cup \{i \mid \mu_{i} \in \Supp(\mu) \text{ and } i = 1 \} \\
    U_{\beta} &:= (\{2,5,7\} \cap U_{m}) \cup \{i \mid \mu_{i} \in \Supp(\mu) \text{ and } i = 1 \}.
\end{align*}
It follows from the above observations that the coefficient on $m$, $\alpha$, and $\beta$ in $r$ is $\# U_{m}$, $\# U_{\alpha}$, and $\# U_{\beta}$, respectively. Since $U_{m} = (U_{\alpha} \setminus \{1\}) \sqcup (U_{\beta} \setminus \{1\})$ and $\kk = \Z/2\Z$, it follows that $\#U_{m} + \#U_{\alpha} + \# U_{\beta} \equiv 0 \mod{2}$, which completes the proof.
\eproof

\blem
\label{lem:cycling_via_Q}
Let $n$ be a positive integer, and let $a$, $c$, and $d$ be distinct integers belonging to $[n]$. Let $A \subset [n] \setminus \{a,c,d\}$, and $B := \{a,c,d\}$. Put $P := P(A,B)$. Let $r \in P^{[2]} + L_{P}$. Let $p$ and $q$ be relatively prime squarefree monomials such that $p$ and $q$ are not divisible by $x_{j}$ or by $y_{j}$ for all $j \in A \cup \{a,c,d\}$. If for some $i \in \{a,c,d\}$ it holds that 
\begin{equation}
    \label{eqn:cycling_via_Q_1}
    \frac{x_{i}}{y_{i}} \cdot \frac{p}{q} \cdot m \in \Supp(r) \quad (\text{respectively, }  \frac{y_{i}}{x_{i}} \cdot \frac{p}{q} \cdot m \in \Supp(r)),
\end{equation}
then the same holds for all $i \in \{a,c,d\}$.
\elem

\bproof
For $i \in \{a,c,d\}$, we denote by $\alpha_{i}$ the monomial 
\begin{align*}
    \frac{x_{i}}{y_{i}} \cdot \frac{p}{q} \cdot m.
\end{align*}
After relabeling the variables as necessary, we may suppose without loss of generality that
\begin{align*}
    \frac{x_{a}}{y_{a}} \cdot \frac{p}{q} \cdot m \in \Supp(r).
\end{align*}
It follows that
\begin{align*}
r = \alpha + \mu \cdot f_{a,c}f_{a,d} \, \omega_{A} + \mu' \cdot f_{a,c}f_{c,d} \, \omega_{A} + \mu'' \cdot f_{a,d}f_{c,d} \, \omega_{A},
\end{align*}
where $\alpha \in P^{[2]}$ and $\mu$, $\mu'$, and $\mu''$ are polynomials in $R$. The constraints on $p$ imply that 
\begin{align*}
\frac{x_{a}}{y_{a}} \cdot \frac{p}{q} \cdot m \in \Supp( \mu \cdot f_{a,c}f_{a,d} \, \omega_{A} + \mu' \cdot f_{a,c}f_{c,d} \, \omega_{A} + \mu'' \cdot f_{a,d}f_{c,d} \, \omega_{A} ).
\end{align*}
By expanding the binomials $f_{a,c}f_{a,d}$ and comparing terms, we see that for $\tilde{\mu} \in \Supp(\mu)$ that
\begin{align*}
\frac{x_{a}}{y_{a}} \cdot \frac{p}{q} \cdot m \in \Supp( \tilde{\mu} \cdot f_{a,c}f_{a,d} \, \omega_{A} )
\end{align*}
if and only if 
\begin{align*}
    \tilde{\mu} = \mu_{1} := \frac{1}{x_{a}y_{a}y_{c}y_{d}} \cdot \frac{p}{q} \cdot \frac{m}{\omega_{A}}.
\end{align*}
Similarly, for $\tilde{\mu} \in \Supp(\mu')$ we have that
\begin{align*}
\frac{x_{a}}{y_{a}} \cdot \frac{p}{q} \cdot m \in \Supp( \tilde{\mu} \cdot f_{a,c}f_{c,d} \, \omega_{A} )
\end{align*}
if and only if 
\begin{align*}
    \tilde{\mu} = \mu_{2} := \frac{1}{y_{a}x_{c}y_{c}y_{d}} \cdot \frac{p}{q} \cdot \frac{m}{\omega_{A}}.
\end{align*}
Finally, for $\tilde{\mu} \in \Supp(\mu'')$ we have that
\begin{align*}
\frac{x_{a}}{y_{a}} \cdot \frac{p}{q} \cdot m \in \Supp( \tilde{\mu} \cdot f_{a,d}f_{c,d} \, \omega_{A} )
\end{align*}
if and only if 
\begin{align*}
    \tilde{\mu} = \mu_{3} := \frac{1}{y_{a}y_{c}x_{d}y_{d}} \cdot \frac{p}{q} \cdot \frac{m}{\omega_{A}}.
\end{align*}
We compute that
\begin{equation}
    \label{eqn:cycling_via_Q_2}
    \mu_{1}\cdot f_{a,c}f_{a,d} \, \omega_{A} = 
    \left( \frac{x_{a}}{y_{a}} + \frac{x_{c}}{y_{c}} + \frac{x_{d}}{y_{d}} + \frac{x_{c}x_{d}y_{a}}{y_{c}y_{d}x_{a}} \right) \cdot \frac{p}{q} \cdot m.
\end{equation}
The change of variables induced by swapping $a$ and $c$ applied to equation \eqref{eqn:cycling_via_Q_2} computes $\mu_{2} \cdot f_{a,c}f_{c,d}\, \omega_{A}$. Likewise, the change of variables induced by swapping $a$ and $d$ applied to equation \eqref{eqn:cycling_via_Q_2} computes $\mu_{3} \cdot f_{a,d}f_{c,d}\, \omega_{A}$. It now follows that $\alpha_{i}$ belongs to the support of $\mu_{2} \cdot f_{a,c}f_{a,d}\, \omega_{A}$ and $\mu_{3} \cdot f_{a,d}f_{c,d}\, \omega_{A}$ for all $i \in \{a,c,d\}$. 

A similar argument shows that for $i \in \{c,d\}$ that
\begin{align*}
    \alpha_{i} \in \Supp(\tilde{\mu} \cdot f_{a,c} f_{a,d} \, \omega_{A}) &\iff \tilde{\mu} = \mu_{1} \\
    \alpha_{i} \in \Supp(\tilde{\mu} \cdot f_{a,c} f_{c,d} \, \omega_{A}) &\iff \tilde{\mu} = \mu_{2} \\
    \alpha_{i} \in \Supp(\tilde{\mu} \cdot f_{a,d} f_{c,d} \, \omega_{A}) &\iff \tilde{\mu} = \mu_{3}.
\end{align*}
It follows that the coefficient on $\alpha_{i}$ in $r$ is the same for all $i \in \{a,c,d\}$, which completes the proof.
\eproof

\blem 
\label{lem:forbiddn_support_of_Q}
Let $n$ be a positive integer, and $a$, $c$, and $d$ be distinct positive integers belonging to $[n]$. Let $A \subset [n] \setminus \{a,c,d\}$, and $B := \{a,c,d\}$. Put $P := P(A,B)$. Let $r \in P^{[2]} + L_{P}$. Let $p$ and $q$ be relatively prime squarefree monomials. Suppose that 
\begin{enumerate}
    \item $p$ is not divisible by $x_{j}$ or by $y_{j}$ for $j \in A \cup \{a,c,d\}$, and
    \item $q$ is not divisible by $x_{j}$ or by $y_{j}$ for $j \in A \cup \{a\}$ or by $y_{c}$.
\end{enumerate}
Then, 
\begin{align*}
     \frac{x_{a}}{y_{c}y_{a}} \cdot \frac{p}{q} \cdot m \notin \Supp(r).
\end{align*}
Analogously, suppose that
\begin{enumerate}
    \item $p$ is not divisible by $x_{j}$ or by $y_{j}$ for $j \in A \cup \{a,c,d\}$, and
    \item $q$ is not divisible by $x_{j}$ or by $y_{j}$ for $j \in A \cup \{a\}$ or by $x_{c}$.
\end{enumerate}
Then, 
\begin{align*}
     \frac{y_{a}}{x_{c}x_{a}} \cdot \frac{p}{q} \cdot m \notin \Supp(r).
\end{align*}
\elem 

\bproof
We denote by $r'$ the monomial 
\begin{align*}
    \frac{x_{a}}{y_{c}y_{a}} \cdot \frac{p}{q} \cdot m.
\end{align*}
Since $r \in P^{[2]} + L_{P}$, we have that
\begin{align*}
    r = \alpha + \beta,
\end{align*}
where $\alpha \in P^{[2]}$ and $\beta \in L_{P}$. The constraints on $p$ imply that $r' \notin \Supp(\alpha)$. 

Next, we observe that every monomial of $f_{a,c}$ is divisible by $y_{a}$ or by $y_{c}$. The constraints on $p$ imply that $r'$ is not divisible by $y_{a}$ nor by $y_{c}$. Hence, it follows that $r' \notin \Supp(f_{a,c}f_{a,d})$ and that $r' \notin \Supp(f_{a,c}f_{c,d})$.

Lastly, we observe that every monomial of $f_{a,d}f_{c,d}$ is divisible by $y_{a}$, by $y_{c}$, or by $x_{a}x_{c}y_{d}^{2}$. Since $r'$ is not divisible by $y_{a}$ nor by $y_{c}$ nor by $y_{d}^{2}$, it follows that $r' \notin \Supp(f_{a,d}f_{c,d})$. 

This proves that $r' \notin \Supp(\beta)$. Hence, in particular, $r' \notin \Supp(r)$.
\eproof

\blem
\label{lem:elements_in_supp_of_Q}
Let $n$ be a positive integer, and let $a$, $c$, and $d$ be distinct integers belonging to $[n]$. Let $A \subset [n] \setminus \{a,c,d\}$, and $B := \{a,c,d\}$. Put $P := P(A,B)$. Let $r \in P^{[2]} + L_{P}$. Let $p$ and $q$ be relatively prime squarefree monomials such that $p$ and $q$ are not divisible by $x_{j}$ or by $y_{j}$ for all $j \in A \cup \{a,c,d\}$. If one element from the set
\begin{align*}
     T_{1} := \left\{ \frac{x_{c}}{x_{a}y_{c}} \cdot \frac{p}{q} \cdot m, \frac{1}{y_{a}} \cdot \frac{p}{q} \cdot m, \frac{x_{d}}{x_{a}y_{d}} \cdot \frac{p}{q} \cdot m \right\},
\end{align*}
respectively from the set 
\begin{align*}
     T_{2} := \left\{ \frac{y_{c}}{y_{a}x_{c}} \cdot \frac{p}{q} \cdot m, \frac{1}{x_{a}} \cdot \frac{p}{q} \cdot m, \frac{y_{d}}{y_{a}x_{d}} \cdot \frac{p}{q} \cdot m \right\},
\end{align*}
belongs to the support of $r$, then every element from the set $T_{1}$ (respectively from the set $T_{2}$) belongs to the support of $r$.
\elem 

\bproof
We will establish the statement for $T_{1}$ as the statement for $T_{2}$ is proven analogously. First, we observe that for all $r' \in T_{1}$ that $r'$ is not divisible by $x_{j}^{2}$ or $y_{j}^{2}$ for all $j \in A$, and $r'$ is not divisible by $x_{i}^{2}y_{j}^{2}$ for all $i,j \in \{a,c,d\}$ distinct. Consequently,
\begin{align*}
    r' \in \Supp(\mu \cdot f_{a,c}f_{a,d} \, \omega_{A} + \mu' \cdot f_{a,c}f_{c,d} \, \omega_{A} + \mu'' \cdot f_{a,d}f_{c,d} \, \omega_{A}).
\end{align*}
Second, we observe that for $r' \in T_{1}$, $r' \notin \Supp(\mu \cdot f_{a,c}f_{a,d} \, \omega_{A})$, because every monomial of $f_{a,c}f_{a,d}$ is divisible by $x_{a}^{2}$, $x_{a}y_{a}$, or $y_{a}^{2}$, whereas $r'$ is not divisible by these monomials. Consequently, 
\begin{align*}
    r' \in \Supp(\mu' \cdot f_{a,c}f_{c,d} \, \omega_{A} + \mu'' \cdot f_{a,d}f_{c,d} \, \omega_{A}).
\end{align*}
Next, we observe that if 
\begin{align*}
    r' = \frac{x_{c}}{x_{a}y_{c}} \cdot \frac{p}{q} \cdot m,
\end{align*}
then
\begin{align*}
    r' \in \Supp(\tilde{\mu}  \cdot f_{a,c}f_{c,d}\, \omega_{A})
\end{align*}
for $\tilde{\mu} \in \Supp(\mu')$ if and only if 
\begin{align*}
    \frac{x_{c}}{x_{a}y_{c}} \cdot \frac{p}{q} \cdot m = \mu' \cdot x_{c}y_{a}x_{c}y_{d} \, \omega_{A}.
\end{align*}
Equivalently, 
\begin{align*}
    \mu' = \mu_{1} := \frac{1}{x_{a}y_{a}x_{c}y_{c}y_{d}} \cdot \frac{p}{q} \cdot \frac{m}{\omega_{A}}.
\end{align*}
It can be verified that for every $r' \in T_{1}$,
\begin{align*}
    r' \in \Supp(\tilde{\mu} \cdot f_{a,c}f_{c,d}\, \omega_{A})
\end{align*}
for $\tilde{\mu} \in \Supp(\mu')$ if and only if $\tilde{\mu} = \mu_{1}$. Similarly, it can be verified that
\begin{align*}
    r' \in \Supp(\tilde{\mu}\cdot  f_{a,d}f_{c,d}\, \omega_{A})
\end{align*}
for $\tilde{\mu} \in \Supp(\mu'')$ if and only if
\begin{align*}
    \tilde{\mu} = \mu_{2} := \frac{1}{x_{a}y_{a}y_{c}x_{d}y_{d}} \cdot \frac{p}{q} \cdot \frac{m}{\omega_{A}}.
\end{align*}
We compute that 
\begin{equation}
    \label{eqn:elements_in_supp_of_Q_1}
    \mu_{1} \cdot f_{a,c}f_{c,d} \, \omega_{A} = \left( \frac{1}{y_{a}} + \frac{x_{d}y_{c}}{y_{a}y_{d}x_{c}} + \frac{x_{c}}{x_{a}y_{c}} + \frac{x_{d}}{x_{a}y_{d}} \right) \cdot \frac{p}{q} \cdot m.
\end{equation}
Finally, we compute that
\begin{equation}
    \label{eqn:elements_in_supp_of_Q_2}
    \mu_{2} \cdot f_{a,d}f_{c,d} \,\omega_{A} = \left( \frac{1}{y_{a}} + \frac{x_{c}y_{d}}{y_{a}y_{c}x_{d}} + \frac{x_{c}}{x_{a}y_{c}} + \frac{x_{d}}{x_{a}y_{d}} \right) \cdot \frac{p}{q} \cdot m.
\end{equation}
Since the coefficient field is $\Z/2\Z$, it follows for parity reasons, if $r' \in \Supp(r)$ for some $r' \in T_{1}$, then either $\mu_{1} \in \Supp(\mu')$ or $\mu_{2} \in \Supp(\mu'')$. In which case, equations \eqref{eqn:elements_in_supp_of_Q_1} and \eqref{eqn:elements_in_supp_of_Q_2} imply the claim.
\eproof

\blem
\label{lem:r_in_P_div_by_omega_S}
Let $n$ be a positive integer. Let $a$ and $c$ be distinct integers in $[n]$. Put $B := \{a,c\}$, and choose $A \subset [n] \setminus \{a,c\}$. Let $r \in P(A,B)^{[2]} + (f_{a,c} \, \omega_{A})$. Suppose that
\begin{enumerate}
    \item $r' \in \Supp(r)$,
    \item $r'$ is not divisible by $x_{j}^{2}$ or by $y_{j}^{2}$ for all $j \in A$,
    \item $r'$ is not divisible by $x_{a}^{2}y_{c}^{2}$ or by $x_{c}^{2}y_{a}^{2}$.
\end{enumerate}
Then,
\begin{enumerate}
    \item $x_{a}y_{c} \mid r'$ or $x_{c}y_{a} \mid r'$, and
    \item $\omega_{A} \mid r'$.
\end{enumerate}
\elem 

\bproof
Expanding $r$ in terms of the generators of $P_{a,c}^{[2]} + (f_{a,c}\, \omega_{A})$, we have that
\begin{align*}
    r = \sum_{j\in A} (\alpha_{j}\cdot x_{j}^{2} + \beta_{j}\cdot y_{j}^{2}) + \gamma \cdot f_{a,c}^{2} + \mu\cdot  f_{a,c} \, \omega_{A}
\end{align*}
for polynomials $\alpha_{j},\beta_{j},\gamma,\mu$. 
The constraints on $r'$ imply that $r' \in \Supp(\mu \cdot f_{a,c} \, \omega_{A})$.
\eproof

\blem
\label{lem:supp_Q_div_bin_supp}
Let $n$ be a positive integer, and let $a$, $c$, and $d$ be distinct integers belonging to $[n]$. Let $A \subset [n] \setminus \{a,c,d\}$, and let $B := \{a,c,d\}$. Put $P := P(A,B)$. Let $r \in P^{[2]} + L_{P}$.  Suppose that 
\begin{enumerate}
    \item $r' \in \Supp(r)$,
    \item $r'$ is not divisible by $x_{j}^{2}$ or $y_{j}^{2}$ for $j \in A$,
    \item $r'$ is not divisible by $x_{k}^{2}y_{l}^{2}$ for any $k,l \in B$ distinct.
\end{enumerate}
Then,
\begin{enumerate}
    \item $x_{k} \mid r'$ or $y_{k} \mid r'$ for each $k \in B$, and
    \item $\omega_{A} \mid r'$, 
    \item If in addition, $y_{d}^{2} \nmid r'$ (respectively $x_{d}^{2} \nmid r'$), then $y_{a}$ or $y_{c}$ (respectively $x_{a}$ or $x_{c}$) divides $r'$.
\end{enumerate}
\elem 

\bproof
It follows that 
\begin{equation}
    \label{eqn:supp_Q_div_bin_supp_1}
    r' \in \Supp(\mu \cdot f_{a,c}f_{a,d} \,\omega_{A} + \mu' \cdot f_{a,c}f_{c,d} \, \omega_{A} + \mu'' \cdot f_{a,d}f_{c,d} \, \omega_{A}).
\end{equation}
We observe that every monomial belonging to the support of $f_{a,c}f_{a,d}$ is divisible by $x_{k}$ or $y_{k}$ for each $k \in B$. Similarly, this holds for $f_{a,c}f_{c,d}$ and $f_{a,d}f_{c,d}$. This proves the first claim. The second claim follows immediately from equation \eqref{eqn:supp_Q_div_bin_supp_1}.

For the third claim, let us suppose that $y_{d}^{2} \nmid r'$. The case in which $x_{d}^{2} \nmid r'$ is proved similarly. We have that equation \eqref{eqn:supp_Q_div_bin_supp_1} holds. We observe that every monomial belonging to the support of $f_{a,c}$ is divisible by $y_{a}$ or $y_{c}$, and that every monomial belonging to the support of $f_{a,d}f_{c,d}$ is divisible by $y_{a}$, $y_{c}$, or $y_{d}^{2}$. This proves the third claim.
\eproof

\section{Reduction Step in Applying Fedder's Criterion}
\label{sec:reduction_step_fedder_criterion}

In this section, we establish a criterion for determining when a binomial edge ideal defines an F-pure quotient ring, Proposition \ref{prop:reduction_step_fedder_computation}. The remaining results in this section are used solely to establish this proposition.

We begin by establishing the setup for this section.

\begin{setup}
\label{setup:co_regular_not_f_pure}
Let $1 \leq m < n$ be positive integers. Let $R$ and $S$ denote the polynomial rings $\k[x_{1},\ldots,x_{m},y_{1},\ldots,y_{m}]$ and $\k[x_{1},\ldots,x_{n},y_{1},\ldots,y_{n}]$, respectively. Let $\m$ denote the ideal $( \{x_{i},y_{i} \mid 1 \leq i \leq m \}) \subset R$, and let $\n$ denote the ideal $( \{x_{i},y_{i} \mid m+1 \leq i \leq n \}) \subset S$. Let $X$ denote the $2 \times n$ matrix with $X_{1,i} = x_{i}$ and $X_{2,i} = y_{i}$ for $1 \leq i \leq n$. Let $Y$ denote the submatrix of $X$ obtained by restricting to the first $m$ columns. Let $G$ be a graph on $n$ vertices. Let $H$ denote the induced subgraph of $G$ on the vertices $\{1,\ldots,m\}$.
\end{setup}

\blem
\label{lem:distributivity_intersection_w_sum}
We utilize the setup established in \ref{setup:co_regular_not_f_pure}. Let $I$ and $K$ be ideals of $R$. Then,
\begin{align*}
    (IS + \n^{[2]}) \cap (KS + \n^{[2]}) = (I \cap K) S + \n^{[2]}.
\end{align*}
\elem

\bproof
The containment $\supset$ always holds.

We prove that the relation $\subset$ holds in this case. Pick $r \in (IS + \n^{[2]})$. Write $r = a + b$, where $a \in IS$ and $b \in \n^{[2]}$. Express $a$ as 
\begin{align*}
    \sum_{i=1}^{k} p_{i} g_{i},
\end{align*}
where $p_{i} \in S$ and $g_{i} \in I$ for $1 \leq i \leq k$. Since $I \subset R$, no monomial belonging to $\Supp(g_{i})$ belongs to $\n^{[2]}$ for $1 \leq i \leq k$. For $c \in \Supp(p_{i})$ and $c \in \n^{[2]}$, replace $a$ by $a - cg_{i}$ and $b$ by $b + cg_{i}$. After repeating this process we may assume that no monomial belonging to $\Supp(a)$ belongs to $\n^{[2]}$. Since $r \in KS + \n^{[2]}$ and $b \in \n^{[2]}$, it follows that $a = r - b \in KS+\n^{[2]}$. Hence, 
\begin{align*}
    a = ( \sum_{i=1}^{\ell} q_{i} h_{i} ) + v,
\end{align*}
where $q_{i} \in S$, $h_{i} \in K$ for $1 \leq i \leq \ell$, and $v \in \n^{[2]}$. Since no non-zero term of $a$ nor of $h_{i}$ belongs to $\n^{[2]}$ for any $1 \leq i \leq \ell$, it follows that a non-zero term of $q_{i}$ belonging to $\n^{[2]}$ will cancel with a non-zero term of $v$. After performing these cancellations, it must be the case that $v = 0$ since the support of $a$ is disjoint from $\n^{[2]}$. Thus, $a \in KS$, and the claim follows.
\eproof

\blem 
\label{lem:reduction_Q_n_to_Q_5}
We utilize the setup established in \ref{setup:co_regular_not_f_pure}. Let $A \subset [n] \setminus [m]$, and let $B \subset [n] \setminus A$. Let $Q := P(A,B)$. Put $B' := B \cap [m]$, and $Q' := P(\varnothing,B')$ Then
\begin{align*}
    \left( Q^{[2]}:_{S} J_{G}  \right) + \n^{[2]} \subset \left( Q'^{[2]}:_{R} J_{H}  \right)S + \n^{[2]}.
\end{align*}
\elem 

\bproof
Let $r \in Q^{[2]}:_{S} J_{G} $, and $f \in J_{H}$ a minimal generator. The relation $r \cdot f \in Q^{[2]}$ holds mod $\n^{[2]}$, and we obtain that 
\begin{align*}
    r \cdot f \in Q'^{[2]}S + \n^{[2]}.
\end{align*}
Hence, we have that
\begin{align*}
    r \in (Q'^{[2]}S + \n^{[2]}) :_{S} J_{H}S.
\end{align*}
Since the variables $x_{m+1},\ldots,x_{n},y_{m+1},\ldots,y_{n}$ do not appear in a minimal generating set of $J_{H}$ nor of $Q'$, it follows that 
\begin{align*}
    (Q'^{[2]} S + \n^{[2]}) :_{S} J_{H}S = \left( Q'^{[2]} S  :_{S} J_{H}S \right) + \n^{[2]}.
\end{align*}
Lastly, since $R \ra S$ is a flat map of rings, it follows that 
\begin{align*}
   \left( Q'^{[2]} S  :_{S} J_{H}S \right) + \n^{[2]} =  \left( Q'^{[2]}  :_{R} J_{H} \right)S + \n^{[2]},
\end{align*}
which completes the proof.
\eproof

\bprop
\label{prop:reduction_step_fedder_computation}
We utilize the setup established in \ref{setup:co_regular_not_f_pure}. Let $A \subset [n] \setminus [m]$, and let $B \subset [n] \setminus A$. We define $Q := P(A,B)$. Put $B' := B \cap [m]$, and $Q' := P(\varnothing,B')$. Suppose that $Q$ is a minimal prime of $J_{G}$, and let $\mc{C}$ be a subset of minimal primes of $J_{G}$ which does not contain $Q$. Let $K$ be an ideal of $R$ such that 
\begin{equation}
    \label{eqn:reduction_step_fedder_computation_assumption_1}
    \bigcap_{P \in \mc{C}} \left( P^{[2]}:_{S} J_{G} \right) \subset K S + \n^{[2]}.
\end{equation}
If 
\begin{equation}
    \label{eqn:reduction_step_fedder_computation_assumption_2}
    K \cap ( Q'^{[2]} :_{R} J_{H}) \subset \m^{[2]},
\end{equation}
then $J_{G}$ is not F-pure.
\eprop 

\bproof
We have that
\begin{align*}
    J_{G}^{[2]} :_{S} J_{G} 
    &\subset \left( \bigcap_{P \in \mc{C}} (P^{[2]} :_{S} J_{G}) \right) \cap \left( Q^{[2]} :_{S} J_{G} \right) & \\
    & \subset \left( KS + \n^{[2]} \right) \cap \left( Q^{[2]} :_{S} J_{G} \right) & & \eqref{eqn:reduction_step_fedder_computation_assumption_1} \\ 
    & \subset \left( KS + \n^{[2]} \right) \cap \left( (Q'^{[2]} :_{R} J_{H})S + \n^{[2]} \right) & & (\ref{lem:reduction_Q_n_to_Q_5}) \\
    &\subset \left( K \cap (Q'^{[2]} :_{R} J_{H}) \right)S + \n^{[2]} & & \eqref{lem:distributivity_intersection_w_sum} \\
    & \subset \m^{[2]}S + \n^{[2]}. & & \eqref{eqn:reduction_step_fedder_computation_assumption_2}
\end{align*}
Fedder's Criterion (Theorem \ref{thm:fedder_criterion}) implies that $J_{G}$ is not F-pure.
\eproof

\section{The ``Co-Regular Families" are Not F-Pure}
\label{sec:co_regular_not_f_pure}

In this section, we utilize the results of Sections \ref{sec:computation_colon_ideals} through \ref{sec:reduction_step_fedder_criterion} to prove that the co-regular families of graphs do not define F-pure ideals. Each co-regular family is treated in its own subsection. We have various technical results: \ref{lem:minimal_primes_anti_hole}, \ref{lem:anti_hole_bad_elts_in_pairs}, \ref{lem:Jn_min_primes} through \ref{cor:restrictions_supp_r}, \ref{lem:jn2_min_primes} through \ref{lem:Jn2_restriction_r'_numerator}, and \ref{lem:Jn2_two_elts_in_supp_r} through \ref{cor:Jn2_supp_r_exclusivity}. These technical preliminaries are used to establish the structure of $J_G^{[2]}:J_G$ in results: \ref{lem:anti_hole_bad_terms_char}, \ref{prop:J_contained_in_Q}, \ref{lem:jn_leftover_elements_not_killed_modulo_ideal}, \ref{prop:Jn_J_contained_in_Q}, \ref{lem:Jn2_supp_r_divisibility}, and \ref{prop:Jn2_J_contained_in_Q}. These structure results together with Proposition \ref{prop:reduction_step_fedder_computation} allows us to prove that co-regular families do not define F-pure ideals in Theorems \ref{thm:Cn_not_F_pure}, \ref{thm:Jn_not_F_pure}, \ref{thm:Jn1_not_F_pure}, and \ref{thm:Jn2_not_F_pure}.

\subsection{The Odd Anti-hole}

We start by establishing common notation that we will use in the remainder of this subsection.

\begin{setup}
\label{setup:anti_hole}
We denote by $\ell$ a positive integer greater than or equal to $2$, and we denote by $n$ the quantity $2\ell +1$. We denote by $G = \overline{C_n}$ the complement of the cycle $C_{n}$. For $1 \leq i \leq n$, we denote by $P_{i,i+2}$ the ideal
\begin{align*}
    P_{i,i+2} := P([n]\setminus\{i,i+1,i+2\},\{i,i+2\}).
\end{align*}
All indices are interpreted modulo n.  
We denote by $J$ the ideal 
\begin{align*}
    J := \bigcap_{i=1}^{n} \left( P_{i,i+2}^{[2]}:J_{G} \right).
\end{align*}
We denote by $m$ the monomial
\begin{align*}
    m &:= \prod_{j=1}^{n} x_{j}y_{j}.
\end{align*}
\end{setup}

\blem
\label{lem:minimal_primes_anti_hole}
With the setup as in \ref{setup:anti_hole}, the following are minimal prime ideals of $J_{G}$: 
\begin{itemize}
\item $P_{i,i+2}$ for $1 \leq i \leq n$, and
\item $P(\varnothing,[n])$.
\end{itemize}
\elem

\bproof
Follows from Proposition \ref{prop:min_primes_binomial_edge_ideal} since the set $[n] \setminus \{i,i+1,i+2\}$ is a cut set of $G$ for every $1 \leq i \leq n$.
\eproof

\noindent In fact, Lemma \ref{lem:minimal_primes_anti_hole} describes all of the minimal primes of $J_{G}$. However, we will not need this result.

\bcor
\label{lem:colon_ideal_of_anti_hole_minl_prime}
With the setup as in \ref{setup:anti_hole}, we have that for $1 \leq i \leq n$ 
\begin{align*}
P_{i,i+2}^{[2]}:J_{G} = P_{i,i+2}^{[2]} + (f_{i,i+2} \, \omega_{A}).
\end{align*}
\ecor 

\bproof
This follows from Theorem \ref{thrm:colon_minl_prime}, where we take $a = i$, $b = i+1$, $c = i+2$, and $d = n+1$ belonging to $[n+1]$, $S := [n+1] \setminus \{a,b,c,d\}$, and $T := \{a,c\}$.
\eproof

\blem
\label{lem:anti_hole_bad_elts_in_pairs}
With the setup as in \ref{setup:anti_hole}, let $r \in J$. Then, 
\begin{enumerate}
    \item $\frac{x_{1}}{y_{1}} \cdot m \in \Supp(r)$ if and only if $\frac{x_{2}}{y_{2}} \cdot m \in \Supp(r)$,
    \item $\frac{y_{1}}{x_{1}} \cdot m \in \Supp(r)$ if and only if $\frac{y_{2}}{x_{2}} \cdot m \in \Supp(r)$,
    \item $\frac{x_{1}y_{2}}{y_{1}x_{2}} \cdot m \in \Supp(r)$ if and only if $\frac{x_{2}y_{1}}{y_{2}x_{1}} \cdot m \in \Supp(r)$,
\end{enumerate}
\elem 

\bproof
We prove (1). Suppose that 
\begin{equation}
    \label{eqn:anti_hole_bad_elts_in_pairs_1}
    \frac{x_{1}}{y_{1}}\cdot m \in \Supp(r).
\end{equation}
Lemma \ref{lem:cycling} applied with $P = P_{1,3}$, $a = 1$, $c = 3$, $p = 1$, and $q = 1$ implies that equation \eqref{eqn:anti_hole_bad_elts_in_pairs_1} is equivalent to  
\begin{align*}
    \frac{x_{3}}{y_{3}}\cdot m \in \Supp(r).
\end{align*}
It follows from Lemmas \ref{lem:cycling} and \ref{lem:minimal_primes_anti_hole} that we may continue in this manner, and thus obtain that  \eqref{eqn:anti_hole_bad_elts_in_pairs_1} is equivalent to 
\begin{align*}
    \frac{x_{2}}{y_{2}}\cdot m \in \Supp(r).
\end{align*}
This proves (1). The proof of (2) is analogous.

We prove (3). Suppose that
\begin{equation}
    \label{eqn:anti_hole_bad_elts_in_pairs_2}
    \frac{x_{1}y_{2}}{y_{1}x_{2}}\cdot m \in \Supp(r).
\end{equation}
Lemma \ref{lem:cycling} applied with $P = P_{1,3}$, $a = 1$, $c = 3$, $p = y_{2}$, and $q = x_{2}$ implies that equation \eqref{eqn:anti_hole_bad_elts_in_pairs_2} is equivalent to  
\begin{equation}
    \label{eqn:anti_hole_bad_elts_in_pairs_3}
    \frac{y_{2}x_{3}}{x_{2}y_{3}}\cdot m \in \Supp(r).
\end{equation}
Now Lemma \ref{lem:cycling} applied with $P = P_{2,4}$, $a = 2$, $c = 4$, $p = x_{3}$, and $q = y_{3}$ implies that equation \eqref{eqn:anti_hole_bad_elts_in_pairs_3} is equivalent to 
\begin{align*}
    \frac{x_{3}y_{4}}{y_{3}x_{4}}\cdot m \in \Supp(r).
\end{align*}
It follows from Lemmas \ref{lem:cycling} and \ref{lem:minimal_primes_anti_hole} that we may continue in this manner, and thus obtain that equation \eqref{eqn:anti_hole_bad_elts_in_pairs_2} is equivalent to 
\begin{align*}
    \frac{x_{2}y_{1}}{y_{2}x_{1}} \cdot m \in \Supp(r).
\end{align*}
This proves (3).
\eproof

\blem
\label{lem:anti_hole_bad_terms_char}
Let the setup be as in \ref{setup:anti_hole}. We define the ideals
\begin{align*}
    \a_{0} &:= (\{x_{j}^{2},y_{j}^{2}\}_{j=6}^{n}) \\
    \a_{1} &:= (\{x_{j}^{2},y_{j}^{2}\}_{j=3}^{5}) + ( \{x_{j}^{3}, y_{j}^{3}\}_{j=1,2} ), \\
    \a_{2} &:= ( \{ x_{j}^{2}y_{j}, y_{j}^{2}x_{j}\}_{j = 1,2} ), \\
    \a_{3} &:= ( x_{1}^{2}x_{2}^{2},y_{1}^{2}y_{2}^{2} ).
\end{align*}
We define the ideal $L := \sum_{i=1}^{3} \a_{i}$. Let $r \in J$, and $r' \in \Supp(r)$ such that $r' \notin L + \a_{0}$. Then $r'$ belongs to the following set:
\begin{align*}
    \A_{1}\cup \A_{2} \cup \A_{3} \cup \A_{4},
\end{align*}
where $\A_{1} := \{h_{1,1} \}$ and $\A_{i} := \{h_{i,1},h_{i,2}\}$ for $2 \leq i \leq 4$, and
\begin{equation*}
\begin{aligned}
    h_{1,1} &:= 1\cdot m \quad \\
    h_{2,1} &:= \frac{x_{1}}{y_{1}} \cdot m
    & h_{2,2} &:= \frac{x_{2}}{y_{2}} \cdot m \\
    h_{3,1} &:= \frac{y_{1}}{x_{1}} \cdot m   
    & h_{3,2} &:= \frac{y_{2}}{x_{2}} \cdot m \\  
    h_{4,1} &:= \frac{x_{1}y_{2}}{y_{1}x_{2}} \cdot m 
    & h_{4,2} &:= \frac{x_{2}y_{1}}{y_{2}x_{1}} \cdot m.
\end{aligned}
\end{equation*}
\elem

\bproof
Since $r \in P_{i,i+2}^{[2]}:J_{G}$ for $i \in \{1,n\}$, Lemma \ref{lem:r_in_P_div_by_omega_S} implies that $r'$ is divisible by $\omega_{S}$ for $S = [n] \setminus \{1,2\}$. 

\ul{Case 1}. Suppose that $r'$ is not divisible by $x_{i}^{2}$ or by $y_{i}^{2}$ for $i = 1,2$. Then $r \in P_{3,5}^{[2]}:J_{G}$, and Lemma \ref{lem:r_in_P_div_by_omega_S} implies that $m \mid r'$, in which case, the constraints on $r'$ arising from $r' \notin L$ imply that $r' = h_{1,1}$.

\ul{Case 2}. Suppose that $x_{1}^{2} \mid r'$. Since $r' \notin \a_{2}$ and $r' \notin \a_{3}$, it follows that $y_{1}$ and $x_{2}^{2}$ do not divide $r'$. Because $r \in P_{n-1,1}^{[2]}:J_{G}$, Lemma \ref{lem:r_in_P_div_by_omega_S} implies that $y_{2}^{2} \mid r'$ or that $x_{2}y_{2} \mid r'$. When $y_{2}^{2} \mid r'$ it follows that $x_{2} \nmid r'$ (otherwise, $r' \in \a_{2}$), and hence $r' = h_{4,1}$. When $x_{2}y_{2} \mid r'$, it follows that $r' = h_{2,1}$.

\ul{Case 3.} The remaining cases where $x_{2}^{2}$, $y_{1}^{2}$, or $y_{2}^{2}$ divides $r'$ are considered analogously to case (2).
\eproof

\bprop
\label{prop:J_contained_in_Q} 
Let the setup be as in \ref{setup:anti_hole}. Let $L$ be the ideal as defined in Lemma \ref{lem:anti_hole_bad_terms_char}. Then 
\begin{align*}
    J + L + \a_{0} \subset (h_{1,1},h_{2,1}+h_{2,2},h_{3,1}+h_{3,2},h_{4,1}+h_{4,2}) + L + \a_{0}.
\end{align*}
\eprop  

\bproof
Let $r \in J$. Let $S := \Supp(r) \setminus (L + \a_{0})$. Define the sets
\begin{align*}
    S_{1} &= \{h_{1,1}\} \\
    S_{2} &= \{h_{2,1},h_{2,2}\} \\
    S_{3} &= \{h_{3,1},h_{3,2}\} \\
    S_{4} &= \{h_{4,1},h_{4,2} \}.    
\end{align*}
Lemmas \ref{lem:anti_hole_bad_elts_in_pairs} and \ref{lem:anti_hole_bad_terms_char} imply that there exists $1 \leq t \leq 4$ and $1 \leq i_{1} < \cdots < i_{t} \leq 4$ such that 
\begin{align*}
    S = S_{i_{1}} \sqcup \cdots \sqcup S_{i_{t}}.
\end{align*}
It follows that
\begin{align*}
    r = \sum_{j = 1}^{t} \sum_{g\in S_{i_{j}}} g \mod{(L+\a_{0})},
\end{align*}
which completes the proof.
\eproof

\bthm
\label{thm:Cn_not_F_pure}
Let $n \geq 5$ be an odd integer, and let $G$ be the graph $\ol{C_{n}}$. Then $J_{G}$ is not F-pure.
\ethm

\bproof
When $n = 5$, we verify with Macaulay2 that $J_{G}$ is not F-pure. Hence, we may suppose that $n \geq 7$ is an odd integer. We utilize the notation for $h_{i,j}$ and $L$ as they are defined in \ref{lem:anti_hole_bad_terms_char}. We will use Proposition \ref{prop:reduction_step_fedder_computation} to prove that $J_{G}$ is not F-pure. In the notation coming from Proposition \ref{prop:reduction_step_fedder_computation}, we have that $A = \varnothing$, $B = [n]$, $Q = P(A,B)$, $m = 5$, $B' = [5]$, $Q' = P(\varnothing,B')$, and $H$ is the graph on $[m]$ having edge set
\begin{align*}
    E(H) := \left\{ \{1,3\},\{1,4\},\{1,5\},\{2,4\},\{2,5\},\{3,5\} \right\}.
\end{align*}
We define the ideal
\begin{align*}
    K := (h_{1,1},h_{2,1}+h_{2,2},h_{3,1}+h_{3,2},h_{4,1}+h_{4,2}) + L.
\end{align*}
We define $\mc{C} := \{P_{i,i+2} \mid 1 \leq i \leq 2n+1\}$. Proposition \ref{prop:J_contained_in_Q} implies that the containment \eqref{eqn:reduction_step_fedder_computation_assumption_1} of Proposition \ref{prop:reduction_step_fedder_computation} is satisfied. We verify using Macaulay2 that equation \eqref{eqn:reduction_step_fedder_computation_assumption_2} of Proposition \ref{prop:reduction_step_fedder_computation} is satisfied. Thus, Proposition \ref{prop:reduction_step_fedder_computation} implies that $J_{G}$ is not F-pure.
\eproof

\subsection{The Graph $\mathbf{co\!-\!XF_1^{2n+3}}$}

We define the graph $\mathrm{co\!-\!XF}_1^{2n+3}$.
\bdefn
For $n \geq 0$, $\mathrm{XF}_1^{2n+3}$ denotes the graph on $2n+7$ vertices having edge set 
\begin{align*}
\{ \{1,i\}_{i = 2}^{2n+5}, \{i,i+1\}_{i=2}^{2n+4}, \{2n+5,2n+6\}, \{2, 2n+7\} \}.
\end{align*}
The graph $\mathrm{co\!-\!XF}_1^{2n+3} := \ol{\mathrm{XF}_1^{2n+3}}$; see Figure~\ref{graph:Jn} for a depiction of $\mathrm{XF}_1^{2n+3}$.
\edefn


\def\edJn{2}
\def\scale{0.75}

\def\ptsJn{0*\edJn/-1*\edJn/1/270,
    .781831*\edJn/-.62349*\edJn/2/0,
    .974928*\edJn/.222521*\edJn/3/0,
    .433884*\edJn/.900969*\edJn/4/0,
   -.974928*\edJn/.222521*\edJn/2n+4/180,
    -.781831*\edJn/-.62349*\edJn/2n+5/180,
    2*.781831*\edJn/2*-.62349*\edJn/2n+7/0,
    2*-.781831*\edJn/2*-.62349*\edJn/2n+6/180}
    
\def\ptsJnUnlabel{-.433884*\edJn/.900969*\edJn/5}

\def\edgesJn{
    1/2,
    1/3,
    1/4,
    1/5,
    1/2n+4,
    1/2n+5,
    2/3,
    3/4,
    2n+4/5,
    2n+4/2n+5,
    2n+5/2n+6,
    2/2n+7}

\begin{figure}[htb]
\begin{center}
\begin{tikzpicture}
	\foreach \x/\y/\z/\w in \ptsJn {
		\draw[fill = black!50] (\scale*\x,\scale*\y) circle [radius = 0.1] node[label = {[label distance = 0.05 cm]\w: $\z$}] (\z) {}; 
	}
\foreach \x/\y/\z in \ptsJnUnlabel {
		\draw[fill = black!50] (\scale*\x,\scale*\y) circle [radius = 0.1] node (\z) {}; 
	}
\foreach \x/\y in \edgesJn { \draw (\x) -- (\y); }  
\draw (4) -- (5) node [pos = 0.5, sloped, above] {{\large $\mathbf{\cdots}$}};
\end{tikzpicture}
\end{center}
\caption{The graph $\mathrm{XF}_1^{2n+3}$, $n\geq 0$} \label{graph:Jn}
\end{figure}

We establish common notation that we will use in the remainder of this subsection.

\begin{setup}
\label{setup:Jn}
Let $n \ge 1$ be an integer.\footnote{We will treat the case of $\mathrm{co\!-\!XF}_{1}^{3}$ separately.} We denote by $G$ the graph $\mathrm{co\!-\!XF}_1^{2n+3}$. For $2 \leq i \leq 2n+3$, we denote by $P_{i,i+2}$ the ideal
\begin{align*}
    P_{i,i+2} := P([2n+7]\setminus\{1,i,i+1,i+2\},\{i,i+2\}).
\end{align*}
We also define the ideals
\begin{align*}
Q_{1} &:= P(\{ j \mid 2 \leq j \leq 2n+3 \text{ or } j = 2n+7 \}, \{1,2n+4,2n+6\}) \\
Q_{2} &:= P( \{j \mid 4 \leq j \leq 2n+6 \}, \{1,3,2n+7\} ) \\
Q_{3} &:= P(\{2n+6,2n+7\}, \{j \mid 2 \leq j \leq 2n+5 \} ).
\end{align*}
We denote by $J$ the ideal 
\begin{align*}
    J := \bigcap_{i=2}^{2n+3} \left( P_{i,i+2}^{[2]}:J_{G} \right) \cap \bigcap_{i=1,2} \left( Q_{i}^{[2]}:J_{G} \right).
\end{align*}
We denote by $m$ the monomial
\begin{align*}
    m &:= \prod_{j=1}^{2n+7} x_{j}y_{j}.
\end{align*}
\end{setup}

\blem
\label{lem:Jn_min_primes}
With the setup as in \ref{setup:Jn}, we have that the following ideals are minimal primes of $J_{G}$:
\begin{enumerate}
\item $P_{i,i+2}$ for $2 \leq i \leq 2n+3$, and
\item $Q_{i}$ for $1 \leq i \leq 3$.
\end{enumerate}
\elem

\bproof
\begin{enumerate}
    \item The proof that $P_{i,i+2}$ is a minimal prime of $J_{G}$ is analogous to the proof in Lemma \ref{lem:minimal_primes_anti_hole}.
    \item We show that $Q_{1}$ is a minimal prime of $J_{G}$. First, we observe that if $S := \{j \mid 2 \leq j \leq 2n+3\} \cup \{2n+7\}$, then $G \setminus S$ consists of two connected components $G_{1}$ and $G_{2}$ where $V(G_{1}) = \{1,2n+4,2n+6\}$ and $E(G_{1}) = \{ \{1,2n+6\}, \{2n+4,2n+6\} \}$, and $G_{2}$ is the graph on the singleton vertex $\{2n+5\}$. It is clear that $S$ is a cut set of $G$, and hence $Q_{1}$ is a minimal prime of $J_{G}$ by Proposition \ref{prop:min_primes_binomial_edge_ideal}. Similarly, it is shown that $Q_{2}$ is a minimal prime of $J_{G}$. 
    \item We show that $Q_{3}$ is a minimal prime of $J_{G}$. If $S = \{2n+6,2n+7\}$, then $G \setminus S$ consists of two connected components. The set $G_{1}$ on the vertex set $1$, and $G_{2}$ on the vertices $\{j \mid 2 \leq j \leq 2n+5\}$. Adding $2n+6$ or $2n+7$ reconnects these connected components, and it follows that $S$ is a cut set of $G$.
\end{enumerate}
\eproof

\blem
\label{lem:colon_minimal_primes_Jn}
With the setup as in \ref{setup:Jn}, we have that
\begin{enumerate}
\item For $2 \leq i \leq 2n+3$ that 
\begin{align*}
    P_{i,i+2}^{[2]} : J_{G} = P_{i,i+2}^{[2]} + (f_{i,i+2} \, \omega_{A})
\end{align*}
\item For $i = 1,2$ we have that 
\begin{align*}
    Q_{i}^{[2]}:J_{G} = Q_{i}^{[2]} + L_{Q_{i}}.
\end{align*}
\end{enumerate}
\elem

\bproof
Follows from Theorem \ref{thrm:colon_minl_prime}. When applying Theorem \ref{thrm:colon_minl_prime} to $Q_{1}$, we let $a = 2n+4$, $b = 2n+5$, $c = 2n+6$, and $d = 1$. When applying Theorem \ref{thrm:colon_minl_prime} to $Q_{2}$, we let $a = 3$, $b = 2$, $c = 2n+7$, and $d = 1$.
\eproof

\bcor
\label{cor:Jn_degree_two_monomial_in_Supp}
With the setup as in \ref{setup:Jn}, let $r \in J$. 
\begin{enumerate}
\item If $m \in \Supp(r)$, then either
\begin{align*}
\frac{x_{1}y_{4}}{y_{1}x_{4}} \cdot m \in \Supp(r) \quad \text{ or }  \quad \frac{x_{1}y_{5}}{y_{1}x_{5}} \cdot m \in \Supp(r).
\end{align*}
\item If $m \notin \Supp(r)$, then 
\begin{align*}
    \frac{x_{1}y_{4}}{y_{1}x_{4}} \cdot m \in \Supp(r) \iff \frac{x_{1}y_{5}}{y_{1}x_{5}} \cdot m \in \Supp(r).
\end{align*}
\end{enumerate}
\ecor 

\bproof 
First we observe that Lemma \ref{lem:cycling} applied to $P_{3,5}$ where $p = x_{1}$ and $q = y_{1}$ implies that
\begin{align*}
    \frac{x_{1}y_{3}}{y_{1}x_{3}} \cdot m \in \Supp(r) \iff \frac{x_{1}y_{5}}{y_{1}x_{5}} \cdot m \in \Supp(r).
\end{align*}
Repeatedly applying Lemma~\ref{lem:cycling} we get the following
\begin{align*}
    \frac{y_{1}x_{3}}{x_{1}y_{3}} \cdot m \in \Supp(r)
    &\iff \frac{y_{1}x_{2n+5}}{x_{1}y_{2n+5}} \cdot m \in \Supp(r) & & (\ref{lem:cycling} \text{ iteratively}) \\
    &\iff \frac{y_{2n+4}x_{2n+5}}{x_{2n+4}y_{2n+5}} \cdot m \in \Supp(r) & & (\ref{lem:cycling_via_Q} \text{ via } Q_{1}) \\
    &\iff \frac{y_{2}x_{3}}{x_{2}y_{3}} \cdot m \in \Supp(r) & & (\ref{lem:cycling} \text{ iteratively}) \\
    &\iff \frac{x_{1}y_{2}}{y_{1}x_{2}} \cdot m \in \Supp(r) & & (\ref{lem:cycling_via_Q} \text{ via } Q_{2}) \\
    &\iff \frac{x_{1}y_{4}}{y_{1}x_{4}} \cdot m \in \Supp(r) & & (\ref{lem:cycling} \text{ via } P_{2,4}).
\end{align*}
The statement now follows from Lemma \ref{lem:Jn_m_in_Supp_f} applied to $Q_{2}$ as we have that 
\begin{align*}
    \card { \left\{ \frac{y_{1}x_{3}}{x_{1}y_{3}} \cdot m, \frac{x_{1}y_{3}}{y_{1}x_{3}} \cdot m, m\right \} \cap \Supp(r) } \equiv 0 \mod{2}.
\end{align*}
\eproof

\bcor 
\label{cor:Jn_cycling_xi_by_yi}
With the setup as in \ref{setup:Jn}, let $r \in J$. If 
\begin{equation}
    \label{eqn:Jn_cycling_xi_by_yi_1}
    \frac{x_{i}}{y_{i}} \cdot m \in \Supp(r) \quad  (\text{respectively, } \frac{y_{i}}{x_{i}} \cdot m \in \Supp(r) )
\end{equation}
for some $i \in \{1,4,5\}$, then the same holds for all $i \in \{1,4,5\}$.
\ecor 

\bproof
Lemma \ref{lem:cycling_via_Q} applied to $Q_{1}$ and $Q_{2}$ implies that 
\begin{align*}
    \frac{x_{1}}{y_{1}} \cdot m \in \Supp(r) \iff \frac{x_{2n+4}}{y_{2n+4}} \cdot m \in \Supp(r)  \text{ and } \frac{x_{3}}{y_{3}} \cdot m \in \Supp(r),
\end{align*}
respectively. Lemma \ref{lem:cycling} implies that 
\begin{align*}
\frac{x_{2n+4}}{y_{2n+4}} \cdot m \in \Supp(r)  \text{ and } \frac{x_{3}}{y_{3}} \cdot m \in \Supp(r)
\end{align*}
if and only if 
\begin{align*}
\frac{x_{4}}{y_{4}} \cdot m \in \Supp(r)  \text{ and } \frac{x_{5}}{y_{5}} \cdot m \in \Supp(r),
\end{align*}
respectively.
\eproof

\bcor 
\label{cor:Jn_restrictions_elements_in_support_of_colon}
Let the setup be as in \ref{setup:Jn}. If $r \in J$ and $r' \in \Supp(r)$, then 
\begin{align*}
    r' \notin \left\{ \frac{x_{i}}{x_{1}y_{1}y_{i}} \cdot m, \frac{y_{i}}{x_{1}y_{1}x_{i}} \cdot m, \frac{x_{i}}{y_{1}y_{i}} \cdot m, \frac{y_{i}}{x_{1}x_{i}} \cdot m \right\}_{i=4,5}.
\end{align*}
\ecor  

\bproof
By Lemma \ref{lem:forbiddn_support_of_Q} it follows for $r \in \bigcap_{i=1,2} (Q_{i}^{[2]}:J_{G})$ and $r' \in \Supp(r)$ that 
\begin{align*}
    r' \notin \left\{ \frac{x_{i}}{x_{1}y_{1}y_{i}} \cdot m, \frac{y_{i}}{x_{1}y_{1}x_{i}} \cdot m, \frac{x_{i}}{y_{1}y_{i}} \cdot m, \frac{y_{i}}{x_{1}x_{i}} \cdot m \right\}_{i=3,2n+4}.
\end{align*}
The claim now follows by applying Lemma \ref{lem:cycling}.
\eproof

\bcor
\label{cor:restrictions_supp_r}
Let the setup be as in \ref{setup:Jn}. If $r \in J$ and $r' \in \Supp(r)$, then
\begin{align*}
    r' \notin \left\{ \frac{x_{i}}{x_{1}y_{i}} \cdot m, \frac{y_{i}}{y_{1}x_{i}} \cdot m, \right \}_{i=4,5} \cup \left \{ \frac{1}{y_{1}} \cdot m, \frac{1}{x_{1}} \cdot m \right \}.
\end{align*}
\ecor 

\bproof
Lemma \ref{lem:cycling} implies that
\begin{equation}
    \label{eqn:elt_in_supp_r}
    \frac{x_{5}}{x_{1}y_{5}} \cdot m \notin \Supp(r) \iff \frac{x_{2n+5}}{x_{1}y_{2n+5}} \cdot m \notin \Supp(r)
\end{equation}
Lemma \ref{lem:elements_in_supp_of_Q} applied to $Q_{1}$ with $p = x_{2n+5}$, $q = y_{2n+5}$, $a = 1$, $c = 2n+4$, and $d = 2n+6$ implies that equation \eqref{eqn:elt_in_supp_r} is equivalent to 
\begin{equation}
    \label{eqn:elt_in_supp_r_2}
    \frac{x_{2n+5}}{y_{2n+5}}\cdot  \frac{y_{2n+6}}{y_{1}x_{2n+6}} \cdot m \notin \Supp(r).
\end{equation}
Lemma \ref{lem:cycling} implies that equation \eqref{eqn:elt_in_supp_r_2} is equivalent to 
\begin{equation}
    \label{eqn:elt_in_supp_r_3}
    \frac{1}{y_{1}} \cdot \frac{y_{2}x_{3}}{x_{2}y_{3}} \cdot m \notin \Supp(r).
\end{equation}
Now, Lemma \ref{lem:forbiddn_support_of_Q} applied via $Q_{2}$ with $a = 3$, $c =1$, $d = 2n+7$, $p = y_{2}$, and $q = x_{2}$ implies equation \eqref{eqn:elt_in_supp_r_3}. A similar argument shows that
\begin{align*}
    r' \notin \left \{\frac{x_{i}}{x_{1}y_{i}} \cdot m, \frac{y_{i}}{y_{1}x_{i}} \cdot m, \right \}_{i=4,5}.
\end{align*}

Lemma \ref{lem:elements_in_supp_of_Q} applied via $Q_{1}$ shows that 
\begin{align*}
    \frac{1}{y_{1}} \cdot m \notin \Supp(r) \iff \frac{x_{2n+4}}{x_{1}y_{2n+4}} \cdot m \notin \Supp(r).
\end{align*}
Lemma \ref{lem:cycling} implies that 
\begin{equation}
\label{eqn:elt_in_supp_r_4}
\begin{aligned}
    \frac{x_{2n+4}}{x_{1}y_{2n+4}} \cdot m \notin \Supp(r)  \iff \frac{x_{4}}{x_{1}y_{4}} \cdot m \notin \Supp(r).
\end{aligned}
\end{equation}
Equation \eqref{eqn:elt_in_supp_r_4} has been previously established in this proof. A similar argument proves that 
\begin{align*}
    \frac{1}{x_{1}} \cdot m \notin \Supp(r).
\end{align*}
\eproof

\blem 
\label{lem:jn_leftover_elements_not_killed_modulo_ideal}
Let the setup be as in \ref{setup:Jn}. We denote by $S$ the set of squarefree monomials of degree $2$ on the variables $\{x_{i},y_{i} \mid i = 1,4,5\}$. Define the ideals
\begin{align*}
    \a_{0} &:= (\{x_{j}^{2},y_{j}^{2}\}_{j=6}^{2n+7}) \\
    \a_{1} &:= (x_{2}^{2},y_{2}^{2},x_{3}^{2},y_{3}^{2}) + ( \{x_{j}^{3}, y_{j}^{3}\}_{j=1,4,5} ) \\
    \a_{2} &:= ( S \setminus  \{x_{1}y_{4}, x_{1}y_{5}\} )^{[2]} \\
    \a_{3} &:= ( \{ x_{j}^{2}y_{j}, y_{j}^{2}x_{j}\}_{j = 1,4,5} ),
\end{align*}
and $L := \sum_{i=1}^{3} \a_{i}$. Let $r \in J$ and $r' \in \Supp(r)$. If $r' \notin L + \a_{0}$, then $r'$ belongs to the following set:
\begin{align*}
    \A_{1}\cup \A_{2} \cup \A_{3}, 
\end{align*}
where $\A_{i} := \{h_{i,1},h_{i,2},h_{i,3}\}$ for $1 \leq i \leq 3$ and
\begin{align*}
    h_{1,1} &:= \frac{x_{1}y_{4}}{y_{1}x_{4}} \cdot m 
    & h_{1,2} &:= \frac{x_{1}y_{5}}{y_{1}x_{5}} \cdot m
    & h_{1,3} &:= 1\cdot m  \\
    h_{2,1} &:= \frac{x_{1}}{y_{1}} \cdot m 
    & h_{2,2} &:= \frac{x_{4}}{y_{4}} \cdot m
    & h_{2,3} &:= \frac{x_{5}}{y_{5}} \cdot m  \\
    h_{3,1} &:= \frac{y_{1}}{x_{1}} \cdot m     
    & h_{3,2} &:= \frac{y_{4}}{x_{4}} \cdot m 
    & h_{3,3} &:= \frac{y_{5}}{x_{5}} \cdot m.
\end{align*}
\elem 

\bproof
We have that
\begin{align*}
    r \in \bigcap_{i=3,4} P_{i,i+2}^{[2]}:J_{G}.
\end{align*}
Since $r' \notin L + \a_{0}$, Lemma \ref{lem:r_in_P_div_by_omega_S} implies that $\omega_{S} \mid r'$ for $S = [2n+7] \setminus \{1,4,5\}$.

\ul{Case 1}. Suppose that $r'$ is not divisible by $x_{i}^{2}$ nor by $y_{i}^{2}$ for $i = 4,5$. Then, Lemma \ref{lem:supp_Q_div_bin_supp} applied to $Q_{1}$ implies that  $\omega_{T} \mid r'$ where $T = [2n+7] \setminus \{1\}$, and that $x_{1} \mid r'$ or that $y_{1} \mid r'$. Thus, we have that
\begin{align*}
    r'  = \frac{d'}{x_{1}y_{1}} \cdot m
\end{align*}
for some $d' \in \{x_{1},y_{1},x_{1}y_{1},x_{1}^{2},y_{1}^{2}\}$. 
Corollary \ref{cor:restrictions_supp_r} implies that 
\begin{enumerate}
\item $d' = x_{1}y_{1}$, in which case $r' = h_{1,3}$,
\item $d' = x_{1}^{2}$, in which case $r' = h_{2,1}$.
\item $d' = y_{1}^{2}$, in which case $r' = h_{3,1}$.
\end{enumerate}

\ul{Case 2}. Suppose that $r'$ is divisible by $x_{4}^{2}$. Then $r'$ is not divisible by $y_{i}^{2}$ for $i=1,4,5$, $x_{1}^{2}$, or $x_{5}^{2}$ (otherwise, $r' \in \a_{2}$). The requirement that
\begin{equation}
	\label{eqn:Jn_r'_divisibility_1}
    r \in P_{2,4}^{[2]} : J_{G}
\end{equation}
implies via Lemma \ref{lem:r_in_P_div_by_omega_S} that $x_{5}y_{5} \mid r'$. Hence, 
\begin{align*}
    r' = \frac{d' x_{4}}{x_{1}y_{1}y_{4}} \cdot m
\end{align*}
for some monomial $d'$ with $\Supp(d') \in \{1,x_{1},y_{1},x_{1}y_{1}\}$. (We notice that $y_{4} \nmid d'$; otherwise, $r' \in \a_{3}$.) Corollaries \ref{cor:Jn_restrictions_elements_in_support_of_colon} and \ref{cor:restrictions_supp_r} imply that $d' = x_{1}y_{1}$, and thus that $r' = h_{2,2}$.

\ul{Case 3}. Suppose that $r'$ is divisible by $x_{5}^{2}$. This case is similar to case 2, and we find that $r' = h_{2,3}$.

\ul{Case 4}. Suppose that $r'$ is divisible by $y_{4}^{2}$. As in case 2 we have that 
\begin{align*}
    r' = \frac{d' y_{4}}{x_{1}y_{1}x_{4}} \cdot m
\end{align*}
for some monomial $d'$ with $d' \in \{1,x_{1},y_{1},x_{1}y_{1},x_{1}^{2} \}$. Corollary \ref{cor:Jn_restrictions_elements_in_support_of_colon} implies that $d' = x_{1}y_{1}$ (in which case $r' = h_{3,2}$) or that $d' = x_{1}^{2}$ (in which case $r' = h_{1,1}$).

\ul{Case 5.} Suppose that $r'$ is divisible by $y_{5}^{2}$. This case is similar to case 4, and we find that $r' = h_{1,2}$ or that $r' = h_{3,3}$.
\eproof

\bprop
\label{prop:Jn_J_contained_in_Q}
Let the setup be as in \ref{setup:Jn}, and let $L$ and $\a_{0}$ be defined as in Lemma \ref{lem:jn_leftover_elements_not_killed_modulo_ideal}. Then, 
\begin{align*}
    J + L + \a_{0} \subset (h_{1,1} + h_{1,3},h_{1,2}+h_{1,3},h_{2,1}+h_{2,2}+h_{2,3},h_{3,1}+h_{3,2}+h_{3,3}) + L + \a_{0}.
\end{align*}
\eprop

\bproof
Let $r \in J$. Let $S := \Supp(r) \setminus (L + \a_{0})$. Define the sets
\begin{align*}
    S_{1} &= \{h_{1,1},h_{1,3}\} \\
    S_{2} &= \{h_{1,2},h_{1,3}\} \\
    S_{3} &= \{h_{1,1},h_{1,2}\} \\
    S_{4} &= \{h_{2,1},h_{2,2},h_{2,3}\} \\
    S_{5} &= \{h_{3,1},h_{3,2},h_{3,3}\}.
\end{align*}
Corollaries \ref{cor:Jn_degree_two_monomial_in_Supp} and \ref{cor:Jn_cycling_xi_by_yi} and Lemma \ref{lem:jn_leftover_elements_not_killed_modulo_ideal} imply that there exists $1 \leq t \leq 5$ and $1 \leq i_{1} < \cdots < i_{t} \leq 5$ satisfying $\# \left( \{i_{1},\ldots,i_{t}\} \cap \{1,2,3\} \right) \leq 1$ such that 
\begin{equation}
    \label{eqn:Jn_reln_1}
    S = S_{i_{1}} \sqcup \cdots \sqcup S_{i_{t}}.
\end{equation}
It follows that
\begin{align*}
    r = \sum_{j = 1}^{t} \sum_{g\in S_{i_{j}}} g \mod{(L + \a_{0})}.
\end{align*}
For $1 \leq i \leq 5$ we have that 
\begin{align*}
    \sum_{g\in S_{i}} g \in (h_{1,1} + h_{1,3},h_{1,2}+h_{1,3},h_{2,1}+h_{2,2}+h_{2,3},h_{3,1}+h_{3,2}+h_{3,3}),
\end{align*}
which completes the proof.
\eproof

\bthm
\label{thm:Jn_not_F_pure}
Let $n \geq 0$ be a positive integer, and $G$ the graph $\mathrm{co\!-\!XF}_1^{2n+3}$. Then $J_{G}$ is not F-pure.
\ethm

\bproof
When $n = 0$, it can be verified via Macaulay2 that $J_{G}$ is not F-pure. Hence, it suffices to consider the case $n \geq 1$. We utilize the notation for $h_{i,j}$ and $L$ as they are defined in \ref{lem:jn_leftover_elements_not_killed_modulo_ideal}. We will use Proposition \ref{prop:reduction_step_fedder_computation} to prove that $J_{G}$ is not F-pure. Adopting the notation from Proposition \ref{prop:reduction_step_fedder_computation} we have that $A = \{2n+6,2n+7\}$, $B = [2n+7] \setminus (A\cup \{1\})$, $Q = P(A,B)$, $m = 5$, $B' = [5] \setminus \{1\}$, $Q' = P(\varnothing,B')$. We define the ideal
\begin{align*}
    K := (h_{1,1} + h_{1,3},h_{1,2}+h_{1,3},h_{2,1}+h_{2,2}+h_{2,3},h_{3,1}+h_{3,2}+h_{3,3}) + L.
\end{align*}
We define $\mc{C} := \{P_{i,i+2} \mid 2 \leq i \leq 2n+3\}\cup \{Q_{i} \mid i \in \{1,2\}\}$. Proposition \ref{prop:Jn_J_contained_in_Q} implies that the containment \eqref{eqn:reduction_step_fedder_computation_assumption_1} of Proposition \ref{prop:reduction_step_fedder_computation} is satisfied. We verify using Macaulay2 that the containment \eqref{eqn:reduction_step_fedder_computation_assumption_2} of Proposition \ref{prop:reduction_step_fedder_computation} is satisfied. Thus, Proposition \ref{prop:reduction_step_fedder_computation} implies that $J_{G}$ is not F-pure.
\eproof

\subsection{The Graph $\mathbf{co\!-\!XF_5^{2n+3}}$}

We define the graph $\mathrm{co\!-\!XF}_5^{2n+3}$.
\bdefn
\label{defn:Jn1}
For $n \geq 0$, $\mathrm{XF}_5^{2n+3}$ denotes the graph on $2n+8$ vertices having edge set 
\begin{align*}
\{ &\{1,i\}_{i = 3}^{2n+6}, \{2,i\}_{i = 3}^{2n+6}, \{i,i+1\}_{i=2}^{2n+5}, \{2, 2n+7\}, \{2n+6, 2n+7\},\\
&\{1,2n+8\}, \{3,2n+8\} \}.
\end{align*}
Then, $\mathrm{co\!-\!XF}_5^{2n+3} := \ol{\mathrm{XF}_5^{2n+3}}$.
\edefn


\def\edJnO{2}

\def\ptsJnO{.382683*\edJnO/-.92388*\edJnO/2/0,
.92388*\edJnO/-.382683*\edJnO/3/0,
.92388*\edJnO/.382683*\edJnO/4/0,
.382683*\edJnO/.92388*\edJnO/5/0,
-.92388*\edJnO/-.382683*\edJnO/2n+6/180,
-.382683*\edJnO/-.92388*\edJnO/1/180,
3*0.6532815*\edJn/3*-0.6532815*\edJn/2n+8/0,
3*-0.6532815*\edJn/3*-0.6532815*\edJn/2n+7/180
,
-.92388*\edJn/.382683*\edJn/2n+5/180
}
    
\def\ptsJnOUnlabel{-.382683*\edJnO/.92388*\edJnO/6}

\def\edgesJnO{
    2/3,
    2/4,
    2/5,
    2/6,
    2/2n+5,
    2/2n+5,
    3/4,
    4/5,
    6/2n+5,
    2n+5/2n+6,
    2n+5/1,
    3/1,
    4/1,
    5/1,
    6/1,
    2n+5/1,
    2n+6/2n+7,
    2/2n+7,
    3/2n+8,
    1/2n+8,
    1/2n+6}

\begin{figure}[htb]
\begin{center}
\begin{tikzpicture}
	\foreach \x/\y/\z/\w in \ptsJnO {
		\draw[fill = black!50] (\scale*\x,\scale*\y) circle [radius = 0.1] node[label = {[label distance = 0.05 cm]\w: $\z$}] (\z) {}; 
	}
 
\foreach \x/\y/\z in \ptsJnOUnlabel {
		\draw[fill = black!50] (\scale*\x,\scale*\y) circle [radius = 0.1] node (\z) {}; 
	}
 
\foreach \x/\y in \edgesJnO { \draw (\x) -- (\y); }  

\draw (5) -- (6) node [pos = 0.5, sloped, above] {{\large $\mathbf{\cdots}$}};

\end{tikzpicture}
\end{center}
\caption{The graph $\mathrm{XF}_5^{2n+3}$, $n\geq 0$} \label{graph:JnO}
\end{figure}

\bthm
\label{thm:Jn1_not_F_pure}
Let $n \geq 0$ be a positive integer, and $G$ the graph $\mathrm{co\!-\!XF}_5^{2n+3}$. Then, $J_{G}$ is not F-pure.
\ethm 

\bproof
After completing $G$ at the vertex $1$ we have that vertex $2$ is connected to the vertices $2n+7$ and $2n+8$. Now, the induced subgraph of $G$ on the vertices $[2n+8] \setminus \{1\}$ is isomorphic to the graph $\mathrm{co\!-\!XF}_1^{2n+3}$. As $\mathrm{co\!-\!XF}_1^{2n+3}$ is not F-pure for $n\geq 0$ (Theorem \ref{thm:Jn_not_F_pure}), it follows that $\mathrm{co\!-\!XF}_5^{2n+3}$ is not F-pure by Lemma \ref{lem:F_split_descends_to_induced_subgraph}.
\eproof

\subsection{The Graph $\mathbf{co\!-\!XF_6^{2n+2}}$}

We define the graph $\mathbf{co\!-\!XF}_6^{2n+2}$.
\bdefn
\label{defn:Jn2}
For $n \geq 0$, $\mathrm{XF}_6^{2n+2}$ denotes the graph on $2n+7$ vertices having edge set 
\begin{align*}
\{ \{1,i\}_{i = 2}^{2n+5}, \{2,i\}_{i = 3}^{2n+5}, \{i,i+1\}_{i=2}^{2n+4}, \{2,2n+6\}, \{2n+5,2n+6\}, \{1, 2n+7\}, \{3, 2n+7\} \}.
\end{align*}
Then, $\mathrm{co\!-\!XF}_6^{2n+2} := \ol{\mathrm{XF}_6^{2n+2}}$
\edefn


\def\edJnT{2}

\def\ptsJnT{-.433884*\edJnT/-.900969*\edJnT/1/180,
.433884*\edJnT/-.900969*\edJnT/2/0,
.974928*\edJnT/-.222521*\edJnT/3/0,
.781831*\edJnT/.62349*\edJnT/4/0,
-.781831*\edJnT/.62349*\edJnT/2n+4/180,
-.974928*\edJnT/-.222521*\edJnT/2n+5/180,
1.5*-1.408812*\edJnT/1.5*-1.12349*\edJnT/2n+6/180,
1.5*1.408812*\edJnT/1.5*-1.12349*\edJnT/2n+7/0
}
    
\def\ptsJnTUnlabel{0*\edJnT/1*\edJnT/5}

\def\edgesJnT{
    1/2,
    1/3,
    1/4,
    1/5,
    1/2n+4,
    1/2n+5,
    2/3,
    3/4,
    5/2n+4,
    2n+4/2n+5,
    2/4,
    2/5,
    2/2n+4,
    2/2n+5,
    2/2n+6,
    2n+5/2n+6,
    1/2n+7,
    3/2n+7}

\begin{figure}[htb]
\begin{center}
\begin{tikzpicture}
	\foreach \x/\y/\z/\w in \ptsJnT {
		\draw[fill = black!50] (\scale*\x,\scale*\y) circle [radius = 0.1] node[label = {[label distance = 0.05 cm]\w: $\z$}] (\z) {}; 
	}
 
\foreach \x/\y/\z in \ptsJnTUnlabel {
		\draw[fill = black!50] (\scale*\x,\scale*\y) circle [radius = 0.1] node (\z) {}; 
	}
 
\foreach \x/\y in \edgesJnT { \draw (\x) -- (\y); }  

\draw (4) -- (5) node [pos = 0.5, sloped, above] {{\large $\mathbf{\cdots}$}};

\end{tikzpicture}
\end{center}
\caption{The graph $\mathrm{XF}_6^{2n+2}$, $n\geq 0$} \label{graph:JnT}
\end{figure}

\brem
The labeling of the vertices in defining $\mathrm{co\!-\!XF}_6^{2n+2}$ in Definition \ref{defn:Jn2} is different than the labeling of the vertices as defined in Trotter \cite{trotter1992combinatorics}.
\erem 

We establish common notation that we will use in the remainder of this subsection.

\begin{setup}
\label{setup:Jn2}
By $n$ we denote a positive integer greater than or equal to $2$.\footnote{We will treat the cases of $\mathrm{co\!-\!XF}_6^{2}$ and $\mathrm{co\!-\!XF}_6^{4}$ separately.} We denote by $G$ the graph $\mathrm{co\!-\!XF}_6^{2n+2}$. For $3 \leq i \leq 2n+3$, we denote by $P_{i,i+2}$ the ideal
\begin{align*}
    P_{i,i+2} := P([2n+7]\setminus\{1,2,i,i+1,i+2\},\{i,i+2\}).
\end{align*}
We define the ideals
\begin{align*}
Q_{1} &:= P(\{ j \mid 5 \leq j \leq 2n+6 \}, \{2,4,2n+7\}) \\
Q_{2} &:= P( \{j \mid 3 \leq j \leq 2n+3 \text{ or } j = 2n+7 \}, \{1,2n+4,2n+6\} ) \\
Q_{3} &:= P(\{2n+6,2n+7\}, \{ 3 \leq j \leq 2n+5 \} ).
\end{align*}
We denote by $J$ the ideal 
\begin{align*}
    J := \bigcap_{i=3}^{2n+3} \left( P_{i,i+2}^{[2]}:J_{G} \right) \bigcap \bigcap_{i=1,2} \left( Q_{i}^{[2]}:J_{G} \right).
\end{align*}
We denote by $m$ the monomial
\begin{align*}
    m &:= \prod_{j=1}^{2n+7} x_{j}y_{j}.
\end{align*}
\end{setup}

\blem
\label{lem:jn2_min_primes}
With setup as in \ref{setup:Jn2}, we have that the following ideals are minimal primes of $J_{G}$:
\begin{enumerate}
\item $P_{i,i+2} := \left( f_{i,i+2},\{x_{j},y_{j}\}_{j\in[2n+7] \setminus \{ 1,2,i,i+1,i+2 \}} \right)$ where $3 \leq i \leq 2n+3$, and
\item $Q_{i}$ for $1 \leq i \leq 3$.
\end{enumerate}
\elem

\bproof
The proof is similar to the proof of Lemma \ref{lem:Jn_min_primes}
\eproof

\blem
\label{lem:colon_minl_primes_J_n_2}
With setup as in \ref{setup:Jn2}, we have that
\begin{enumerate}
\item For $3 \leq i \leq 2n+3$,
\begin{align*}
P_{i,i+2}^{[2]}:J_{G} = P_{i,i+2} + (f_{i,i+2} \, \omega_{A}).
\end{align*}
\item For $i = 1,2$, we have that
\begin{align*}
Q_{i}^{[2]}:J_{G} = Q_{i}^{[2]} + L_{Q_{i}}.
\end{align*}
\end{enumerate}
\elem

\bproof
Follows from Theorem \ref{thrm:colon_minl_prime}. When applying Theorem \ref{thrm:colon_minl_prime} to $Q_{1}$, we let $a = 4$, $b = 3$, $c = 2n+7$, and $d = 2$. When applying Theorem \ref{thrm:colon_minl_prime} to $Q_{2}$, we let $a = 2n+4$, $b = 2n+5$, $c = 2n+6$, and $d = 1$.
\eproof

\bcor
\label{cor:Jn2_restrictions_supp_r}
Let the setup be as in \ref{setup:Jn2}. If $r \in J$ and $r' \in \Supp(r)$, then
\begin{align*}
    r' \notin \left \{ \frac{1}{x_{1}x_{2}}\cdot m, \frac{1}{y_{1}y_{2}}\cdot m \right \}.
\end{align*}
\ecor 

\bproof
We have that
\begin{align*}
    \frac{1}{y_{1}y_{2}} \cdot m \in \Supp(r) 
    &\iff \frac{x_{4}}{y_{1}x_{2}y_{4}}\cdot m \in \Supp(r) & &  (\ref{lem:elements_in_supp_of_Q} : P = Q_{1}, a = 2, c = 4, q = y_{1}) \\
    &\iff \frac{x_{2n+4}}{y_{1}x_{2}y_{2n+4}} \in \Supp(r) & & (\ref{lem:cycling} \text{ iteratively}).
\end{align*}
Lemma \ref{lem:forbiddn_support_of_Q} applied to $Q_{2}$ with $a = 2n+4$, $c = 1$, and $q = x_{2}$ implies that 
\begin{align*}
    \frac{x_{2n+4}}{y_{1}x_{2}y_{2n+4}} \notin \Supp(r).
\end{align*}
Similarly, it is proven that 
\begin{align*}
    \frac{1}{x_{1}x_{2}} \notin \Supp(r).
\end{align*}
\eproof

\bcor
\label{cor:Jn2_structure_or_r'}
Let the setup be as in \ref{setup:Jn2}. Let $r \in J$. Let $q$ be a squarefree monomial with $\Supp(q) \subset \{x_{i},y_{i}\}_{i=1,2}$. If
\begin{align*}
    \frac{x_{4}}{qy_{4}}\cdot m \in \Supp(r),
\end{align*}
then $q$ is not divisible by $y_{1}$ or by $y_{2}$.
\ecor 

\bproof
Lemma \ref{lem:forbiddn_support_of_Q} applied via $Q_{1}$ implies that $y_{2}$ does not divide $q$. Lemma \ref{lem:cycling} implies that 
\begin{align*}
    \frac{x_{4}}{qy_{4}}\cdot m \in \Supp(r) \iff \frac{x_{2n+4}}{qy_{2n+4}}\cdot m \in \Supp(r).
\end{align*}
Lemma \ref{lem:forbiddn_support_of_Q} applied via $Q_{2}$ implies that $y_{1}$ does not divide $q$.
\eproof

\blem
\label{lem:Jn2_restriction_r'}
Let the setup be as in \ref{setup:Jn2}. Let $r \in J$. Let $q$ be a squarefree monomial with $\Supp(q) \subset \{x_{i}, y_{i}\}_{i=1,2}$. Moreover, suppose that $\#\Supp(q) \geq 3$. Then, 
\begin{align*}
\frac{x_{5}}{qy_{5}} \cdot m \notin \Supp(r).
\end{align*}
\elem 

\bproof
We define the elements 
\begin{align*}
    r' := \frac{x_{5}}{qy_{5}} \cdot m, \quad r'' := \frac{x_{3}}{qy_{3}} \cdot m, \quad \text{ and } \quad r''' := \frac{x_{2n+5}}{qy_{2n+5}} \cdot m.
\end{align*}
Suppose that $r' \in \Supp(r)$. Then, Lemma \ref{lem:cycling} implies that $r''$ and $r'''$ also belong to $\Supp(r)$. Applying Lemma \ref{lem:supp_Q_div_bin_supp} to $r''$ via $Q_{1}$ implies that $x_{2}$ or $y_{2}$ divides $r''$. Hence, $q$ is not divisible by $x_{2}$ or by $y_{2}$. Similarly, applying Lemma \ref{lem:supp_Q_div_bin_supp} to $r'''$ via $Q_{2}$ implies that $x_{1}$ or $y_{1}$ divides $r'''$. Therefore, $q$ is not divisible by $x_{1}$ or by $y_{1}$. This leads to the conclusion that $\# \Supp(q) \leq 2$.
\eproof 

\blem
\label{lem:Jn2_restriction_r'_numerator}
Let the setup be as in \ref{setup:Jn2}. Let $r \in J$. For $p \in \{1,y_{1},y_{2}\}$, we have that 
\begin{align*}
\frac{p x_{5}}{x_{1}x_{2}y_{5}} \cdot m \notin \Supp(r).
\end{align*}
\elem 

\bproof 
When $p = 1$, we have that
\begin{align*}
    &\phantom{\iff\,\,\,} \frac{x_{5}}{x_{1}x_{2}y_{5}}\cdot m \in \Supp(r)\\
    &\iff \frac{x_{3}}{x_{1}x_{2}y_{3}}\cdot m \in \Supp(r) & & (\ref{lem:cycling}) \\
    &\iff \frac{x_{3}}{x_{1}y_{3}} \cdot \frac{y_{4}}{y_{2}x_{4}}\cdot m \in \Supp(r) & & (\ref{lem:elements_in_supp_of_Q} : Q_{1}, a = 2, c = 4, d = 2n+7, p = x_{3}, q = x_{1}y_{3}) \\
    &\iff \frac{y_{2n+4}x_{2n+5}}{x_{1}y_{2}x_{2n+4}y_{2n+5}} \cdot m \in \Supp(r) & & (\ref{lem:cycling} \text{ iteratively}).
\end{align*}
Lemma \ref{lem:forbiddn_support_of_Q} applied to $Q_{2}$ with $a= 2n+4$, $c = 1$, $d = 2n+4$, $p = x_{2n+5}$, and $q = y_{2}y_{2n+5}$ implies that 
\begin{align*}
    \frac{y_{2n+4}x_{2n+5}}{x_{1}y_{2}x_{2n+4}y_{2n+5}} \cdot m \notin \Supp(r).
\end{align*}

When $p = y_{1}$, we have that
\begin{align*}
     \frac{y_{1}x_{5}}{x_{1}x_{2}y_{5}}\cdot m \in \Supp(r)
    &\iff \frac{y_{1}x_{2n+5}}{x_{1}x_{2}y_{2n+5}}\cdot m \in \Supp(r) & & (\ref{lem:cycling} \text{ iteratively}) \\
    &\iff \frac{y_{2n+4}x_{2n+5}}{x_{2}x_{2n+4}y_{2n+5}}\cdot m \in \Supp(r) & & (\ref{lem:cycling_via_Q} : Q_{2}) \\
    &\iff \frac{x_{3}y_{4}}{x_{2}y_{3}x_{4}}\cdot m \in \Supp(r) & & (\ref{lem:cycling} \text{ iteratively}).
\end{align*}
Lemma  \ref{lem:forbiddn_support_of_Q} applied to $Q_{1}$ with $a= 4$, $c = 2$, $d = 2n+7$, $p = x_{3}$, and $q = y_{3}$ implies that 
\begin{align*}
    \frac{x_{3}y_{4}}{x_{2}y_{3}x_{4}}\cdot m \notin \Supp(r).
\end{align*}

When $p = y_{2}$, it is shown that 
\begin{align*}
    \frac{y_{2}x_{5}}{x_{1}x_{2}y_{5}}\cdot m \notin \Supp(r)
\end{align*}
similarly as in the case where $p = y_{1}$.
\eproof

\bnota
For $f \in R$, $S \subset R$, and $T \subset R$, we define the sets
\begin{align*}
    f \cdot S &:= \{f g \mid g \in S \} \\
    S \cdot T &:= \{ a \cdot b \mid a \in S, \, b \in T \}.
\end{align*}
\enota

\blem 
\label{lem:Jn2_supp_r_divisibility}
Let the setup be as in \ref{setup:Jn2}. Define the ideals
\begin{align*}
    \a_{0} &:= (\{x_{j}^{2},y_{j}^{2}\}_{j=6}^{2n+7}) \\
    \a_{1} &:= ( x_{3}^{2},y_{3}^{2}) +  ( x_{1}^{2},x_{2}^{2},y_{4}^{2},y_{5}^{2} ) + (y_{1}^{3},y_{2}^{3},x_{4}^{3}, x_{5}^{3}) \\
    \a_{2} &:= ( x_{1}y_{1}^{2}, x_{2}y_{2}^{2}, x_{4}^{2}y_{4}, x_{5}^{2}y_{5}  ) \\
    \a_{3} &:= ( x_{1}x_{2}x_{4}^{2}, x_{1}x_{2}x_{5}^{2}, x_{4}^{2}x_{5}^{2} ) \\
    \a_{4} &:= ( x_{4}x_{5}y_{1}^{2}y_{4}y_{5}, x_{1}x_{4}x_{5}y_{2}^{2}y_{4}y_{5}, x_{4}x_{5}y_{1}y_{2}^{2}y_{4}y_{5}, y_{1}^{2}y_{2}^{2} ).
\end{align*}
and $L := \sum_{i=1}^{4} \a_{i}$. Let $r \in J$ and $r' \in \Supp(r)$. If $r' \notin L + \a_{0}$, then 
$r'$ belongs to the following set:
\begin{align*}
    \mc{A} \cup \mc{B}
\end{align*}
where 
\begin{align*}
    \mc{A} &:= h_{1,1} \cdot \{1 , x_{2}, y_{1}, x_{2}y_{1} \} \cup h_{1,2} \cdot \{ 1, x_{1} , y_{2}, x_{1}y_{2} \} \cup h_{1,3} \cdot \{1,x_{1},x_{2},y_{1},y_{2},x_{1}y_{2},x_{2}y_{1}\}
    \\ 
    \mc{B} &:= \{h_{2,1},h_{2,2} \} \cdot \{1,y_{1},y_{2},y_{1}y_{2} \}
\end{align*}
and
\begin{align*}
    h_{1,1} &:= \frac{1}{x_{2}y_{1}} \cdot m 
    & h_{1,2} &:=  \frac{1}{x_{1}y_{2}} \cdot m 
    & h_{1,3} &:= \frac{x_{4}}{x_{1}x_{2}y_{4}} \cdot m \\
    h_{2,1} &:= \frac{x_{5}}{x_{2}y_{1}y_{5}} \cdot m
    & h_{2,2} &:= \frac{x_{5}}{x_{1}y_{2}y_{5}} \cdot m
\end{align*}
\elem 

\bproof
Since $r \in P_{i,i+2}^{[2]}:J_{G}$ for $i = 3$ and $i = 4$ and $r' \notin L + \a_{0}$, Lemma \ref{lem:r_in_P_div_by_omega_S} implies that $\omega_{S} \mid r'$ for $S = [2n+7] \setminus \{1,2,4,5\}$.

\ul{Case 1}. Suppose that $r{'}$ is not divisible by $x_{4}^{2}$ nor by $x_{5}^{2}$. Since $n \geq 4$, $r \in P_{2n+3,2n+5}^{[2]}:J_{G}$, and $r' \notin L + \a_{0}$, Lemma \ref{lem:r_in_P_div_by_omega_S} implies that $x_{4}y_{4}x_{5}y_{5} \mid r'$. Consequently, we have that
\begin{align*}
    r' = \frac{d'}{x_{1}x_{2}y_{1}y_{2}} \cdot m
\end{align*}
for some monomial $d'$ with $\Supp(d') \subset \{x_{i},y_{i}\}_{i=1,2}$. Lemma \ref{lem:supp_Q_div_bin_supp} applied to $r'$ via $Q_{1}$ and $Q_{2}$ implies that $d'$ is divisible by $x_{i}$ or by $y_{i}$ for both $i = 1$ and $i = 2$. After accounting for the constraint that $r'$ does not belong to $\a_{1}$ nor to $\a_{2}$, we see that 
\begin{equation}
    \label{eqn:Jn2_supp_r_divisibility_set}
    d' \in \{x_{1},y_{1},x_{1}y_{1},y_{1}^{2}\} \cdot \{ x_{2},y_{2},x_{2}y_{2},y_{2}^{2} \}.
\end{equation}
Expanding the products in   \eqref{eqn:Jn2_supp_r_divisibility_set} and removing those terms which would force $r' \in L + \a_{0}$, we find that
\begin{align*}
    d' \in \{x_{1}x_{2},x_{1}y_{2}, x_{1}x_{2}y_{2} \} \cup \{y_{1}x_{2},y_{1}y_{2},y_{1}x_{2}y_{2}\} \cup \{x_{1}y_{1}x_{2},x_{1}y_{1}y_{2},x_{1}y_{1}x_{2}y_{2} \}.
\end{align*}
Lemma \ref{cor:Jn2_restrictions_supp_r} implies that $d' \neq x_{1}x_{2}$ and that $d' \neq y_{1}y_{2}$. For the remaining values of $d'$ we compute that\begin{align*}
    r' \in h_{1,1} \cdot \{1 , x_{2}, y_{1}, x_{2}y_{1} \} \cup h_{1,2} \cdot \{ 1, x_{1} , y_{2}, x_{1}y_{2} \}.
\end{align*}

\ul{Case 2}. Suppose that $r'$ is divisible by $x_{5}^{2}$. Since $r' \notin \a_{3}$, we conclude that $r'$ is not divisible by $x_{4}^{2}$. Lemma \ref{lem:r_in_P_div_by_omega_S} applied to $r'$ via $P_{5,7}$ implies that $x_{4}y_{4} \mid r'$. Hence we have that
\begin{align*}
    r' = \frac{d'\cdot x_{5}}{x_{1}x_{2}y_{1}y_{2}y_{5}} \cdot m
\end{align*}
for some monomial $d'$ where $\Supp(d') \subset \{x_{i},y_{i}\}_{i=1,2}$. The fact that $y_{5} \notin \Supp(d')$ follows from the condition that $r' \notin \a_{2}$. Suitable applications of Lemma \ref{lem:cycling} together with Lemma \ref{lem:supp_Q_div_bin_supp} applied to $r'$ via $Q_{1}$ and $Q_{2}$ implies that
\begin{align*}
    d' \in \{x_{1},y_{1},x_{1}y_{1},y_{1}^{2}\}\cdot \{x_{2},y_{2},x_{2}y_{2},y_{2}^{2}\}.
\end{align*}
Since $x_{1}x_{2}x_{5}^{2} \in \a_{3}$, we may suppose that $x_{1} x_{2} \nmid d'$. If $y_{1}y_{2} \mid d'$ and $\Supp(d') \subset \{y_{1},y_{2}\}$, then 
\begin{equation}
    \label{eqn:Jn2_supp_r_divisibility_impossibility}
    r' = \frac{p x_{5}}{x_{1}x_{2}y_{5}}\cdot m
\end{equation}
for $p \in \{1,y_{1},y_{2}\}$. However, Lemma \ref{lem:Jn2_restriction_r'_numerator} implies that equation \eqref{eqn:Jn2_supp_r_divisibility_impossibility} is impossible. Thus, we are reduced to considering 
\begin{align*}
    d' \in \{x_{1}y_{2},x_{1}y_{2}^{2},y_{1}x_{2},y_{1}x_{2}y_{2},x_{1}y_{1}y_{2},x_{1}y_{1}y_{2}^{2},y_{1}^{2}x_{2},y_{1}^{2}x_{2}y_{2}\}.
\end{align*}
For these values of $d'$ we compute that
\begin{align*}
    r' \in \{h_{2,1},h_{2,2} \}\cdot \{1,y_{1},y_{2},y_{1}y_{2} \}.
\end{align*}

\ul{Case 3}. Suppose that $r'$ is divisible by $x_{4}^{2}$ and that $r'$ is not divisible by $y_{2}^{2}$. Since $r' \notin \a_{3}$, it follows that $x_{5}^{2} \nmid r'$. Since $y_{2}^{2} \nmid r'$, Lemma \ref{lem:supp_Q_div_bin_supp} applied to $r'$ via $Q_{1}$ implies that $x_{5}y_{5} \mid r'$. Thus, 
\begin{align*}
    r' = \frac{d' \cdot x_{4}}{x_{1}x_{2}y_{1}y_{2}y_{4}} \cdot m
\end{align*}
for some monomial $d'$ with $\Supp(d') \subset \{ x_{i},y_{i} \}_{i=1,2}$. The fact that $y_{4} \notin \Supp(d')$ follows from the condition that $x_{4}^{2}y_{4} \in \a_{2}$. Since $r'$ is not divisible by $y_{4}$, Corollary \ref{cor:Jn2_structure_or_r'} implies that $y_{1}y_{2} \mid d'$. Thus, we have that
\begin{equation}
    \label{eqn:Jn2_supp_r_divisibility_d_prime_case_3}
    d' \in \{y_{1},x_{1}y_{1},y_{1}^{2}\} \cdot \{y_{2},x_{2}y_{2},y_{2}^{2}\}.
\end{equation}
We have that $d' \neq x_{1}y_{1}x_{2}y_{2}$ (since $x_{1}x_{2} y_{4}^{2} \in \a_{3}$), and we have that $d' \neq y_{1}^{2}y_{2}^{2}$ (since $y_{1}^{2}y_{2}^{2} \in \a_{4}$). For the remaining values of $d'$, we compute that
\begin{align*}
    r' \in h_{1,3} \cdot \{1,x_{1},x_{2},y_{1},y_{2},x_{1}y_{2},x_{2}y_{1}\}.
\end{align*}

\ul{Case 4.} Suppose that $r'$ is divisible by $x_{4}^{2}y_{2}^{2}$. Since $r' \notin \a_{2}$ it follows that $y_{4} \nmid r'$ and that $x_{2} \nmid r'$. Hence 
\begin{align*}
    r' = \frac{d' \cdot x_{4}y_{2}}{x_{1}x_{2}x_{5}y_{1}y_{4}y_{5}}\cdot m
\end{align*}
for some monomial $d'$ with $\Supp(d') \subset \{ x_{1},x_{5},y_{1}, y_{5} \}$. Applying Lemma \ref{lem:cycling} to $r'$ via $P_{4,6}$ implies that 
\begin{align*}
    r'' := \frac{d' \cdot x_{6}y_{2}}{x_{1}x_{2}x_{5}y_{1}y_{6}y_{5}}\cdot m \in \Supp(r).
\end{align*}
Since $x_{5}^{2} \nmid d'$ (otherwise, $r' \in \a_{3}$) and $y_{5}^{2} \nmid d'$ (otherwise, $r' \in \a_{1}$), Lemma \ref{lem:r_in_P_div_by_omega_S} applied to $r''$ via $P_{6,8}$ implies that $x_{5}y_{5} \mid r''$. Thus, we have that
\begin{align*}
    r' = \frac{d' \cdot x_{4}y_{2}}{x_{1}x_{2}y_{1}y_{4}}\cdot m
\end{align*}
for some monomial $d'$ with $\Supp(d') \subset \{ x_{1},y_{1} \}$. Lemma \ref{lem:cycling} implies that 
\begin{align*}
    r''' := \frac{d' \cdot x_{2n+4}y_{2}}{x_{1}x_{2}y_{1}y_{2n+4}}\cdot m \in \Supp(r).
\end{align*}
We observe that $y_{1}^{2} \nmid r'''$ (otherwise, $r' \in \a_{4}$), $y_{2n+4}\nmid r'''$, and $y_{2n+6}^{2} \nmid r'''$. Consequently, Lemma \ref{lem:supp_Q_div_bin_supp} applied to $r'''$ via $Q_{2}$ implies that $y_{1}$ divides $r'''$. Thus, $y_{1} \mid d'$, and we have that
\begin{align*}
    r' = \frac{d' \cdot x_{4}y_{2}}{x_{1}x_{2}y_{4}}\cdot m
\end{align*}
for some $d' \in \{1,x_{1}\}$. For $d' \in \{1,x_{1}\}$, we have that 
\begin{align*}
    r' \in h_{1,3} \cdot \{y_{2},x_{1}y_{2}\}.
\end{align*}
\eproof

\blem
\label{lem:Jn2_two_elts_in_supp_r}
Let the setup be as in \ref{setup:Jn2}. Let $r \in J$ and $p \in \{1,y_{1},y_{2},y_{1}y_{2}\}$. Then,
\begin{align*}
    \frac{p  x_{5}}{x_{2}y_{1}y_{5}} \cdot m \in \Supp(r) \iff \frac{p x_{5}}{x_{1}y_{2}y_{5}} \cdot m \in \Supp(r).
\end{align*}
\elem 

\bproof

When $p \in \{1,y_{1}\}$, we have that 
\begin{align*}
    &\phantom{\iff\,\,\,\,} \frac{px_{5}}{x_{2}y_{1}y_{5}}\cdot m \in \Supp(r)\\
    &\iff \frac{px_{3}}{x_{2}y_{1}y_{3}}\cdot m \in \Supp(r) & &(\ref{lem:cycling}) \\
    &\iff \frac{px_{3}y_{4}}{y_{1}y_{2}y_{3}x_{4}}\cdot m \in \Supp(r) & & (\ref{lem:elements_in_supp_of_Q} : Q_{1}, p' = px_{3}, q' = y_{1}y_{3}, a = 2, c = 4) \\
    &\iff \frac{py_{2n+4}x_{2n+5}}{y_{1}y_{2}x_{2n+4}y_{2n+5}}\cdot m \in \Supp(r) & & (\ref{lem:cycling} \text{ iteratively}).
\end{align*}

Specializing to the case that $p = 1$, we have that
\begin{equation}
    \label{eqn:Jn2_two_elts_in_supp_r_0}
    \frac{y_{2n+4}x_{2n+5}}{y_{1}y_{2}x_{2n+4}y_{2n+5}}\cdot m \in \Supp(r)
    \iff \frac{x_{2n+5}}{x_{1}y_{2}y_{2n+5}}\cdot m \in \Supp(r)
\end{equation}
by applying Lemma \ref{lem:elements_in_supp_of_Q} to $Q_{2}$ with $p' = x_{2n+5}$, $q' = y_{2}y_{2n+5}$, $a = 1$, and $c = 2n+4$. Finally, Lemma \ref{lem:cycling} applied to equation \eqref{eqn:Jn2_two_elts_in_supp_r_0} implies that 
\begin{align*}
    \frac{x_{2n+5}}{x_{1}y_{2}y_{2n+5}}\cdot m \in \Supp(r) \iff \frac{x_{5}}{x_{1}y_{2}y_{5}} \cdot m \in \Supp(r).
\end{align*}

Specializing to the case that $p = y_{1}$, we have that
\begin{align*}
    &\phantom{\iff\,\,\,\,} \frac{y_{2n+4}x_{2n+5}}{y_{2}x_{2n+4}y_{2n+5}}\cdot m \in \Supp(r)\\
    &\iff \frac{y_{1}x_{2n+5}}{x_{1}y_{2}y_{2n+5}}\cdot m \in \Supp(r) & &(\ref{lem:cycling_via_Q} : Q_{2}, p' = x_{2n+5}, q' = y_{2}y_{2n+5}) \\
    &\iff \frac{y_{1}x_{5}}{x_{1}y_{2}y_{5}} \cdot m \in \Supp(r) & & (\ref{lem:cycling} \text{ iteratively}). \\
\end{align*}

\ul{Case $p = y_{2}$:} The proof is analogous to the case where $p = y_{1}$.

\ul{Case $p = y_{1}y_{2}$:} 
We have that
\begin{align*}
\frac{y_{2}x_{5}}{x_{2}y_{5}}\cdot m \in \Supp(r) 
&\iff \frac{y_{2}x_{3}}{x_{2}y_{3}}\cdot m \in \Supp(r) & & (\ref{lem:cycling}) \\
&\iff \frac{y_{4}x_{3}}{x_{4}y_{3}}\cdot m \in \Supp(r) & &(\ref{lem:cycling_via_Q} : Q_{1}) \\
&\iff \frac{y_{2n+4}x_{2n+5}}{x_{2n+4}y_{2n+5}}\cdot m \in \Supp(r) & &(\ref{lem:cycling} \text{ iteratively}) \\
&\iff \frac{y_{1}x_{2n+5}}{x_{1}y_{2n+5}}\cdot m \in \Supp(r) & &(\ref{lem:cycling_via_Q} : Q_{2}) \\
&\iff \frac{y_{1}x_{5}}{x_{1}y_{5}} \cdot m \in \Supp(r) & &(\ref{lem:cycling} \text{ iteratively}).
\end{align*}
\eproof

\blem
\label{lem:Jn2_three_elts_in_supp_r}
Let the setup be as in \ref{setup:Jn2}. We utilize notation as introduced in Lemma \ref{lem:Jn2_supp_r_divisibility}. Let $r \in J$, then 
\begin{enumerate}
    \item $h_{1,1} \in \Supp(r)$ if and only if $h_{1,2} \in \Supp(r)$ if and only if $h_{1,3} \in \Supp(r)$,
    \item $x_{1} h_{1,2} \in \Supp(r)$ if and only if $x_{1} h_{1,3} \in \Supp(r)$,
    \item $y_{1} h_{1,1} \in \Supp(r)$ if and only if $y_{1} h_{1,3} \in \Supp(r)$,
    \item $x_{2} h_{1,1} \in \Supp(r)$ if and only if $x_{2} h_{1,3} \in \Supp(r)$, and
    \item $y_{2} h_{1,2} \in \Supp(r)$ if and only if $y_{2} h_{1,3} \in \Supp(r)$.
\end{enumerate}
\elem 

\bproof
(1) First, we observe that 
\begin{align*}
    \frac{1}{x_{1}y_{2}}\cdot m \in \Supp(r) \iff \frac{x_{4}}{x_{1}x_{2}y_{4}}\cdot m \in \Supp(r)
\end{align*}
by applying Lemma \ref{lem:elements_in_supp_of_Q} to $h_{1,2}$ with $Q_{1}$ and $p = 1$, $q = x_{1}$, $a = 2$, $c = 4$, and $d = 2n+7$.
Next, we observe that 
\begin{align*}
    &\phantom{\iff\,\,\,\,} \frac{1}{x_{2}y_{1}}\cdot m \in \Supp(r) \\
    &\iff \frac{x_{2n+4}}{x_{1}x_{2}y_{2n+4}} \cdot m \in \Supp(r) & & (\ref{lem:elements_in_supp_of_Q} : Q_{2}, p = 1, q = x_{2}, a = 1, c = 2n+4) \\
    &\iff \frac{x_{4}}{x_{1}x_{2}y_{4}} \cdot m \in \Supp(r) & & (\ref{lem:cycling} \text{ iteratively}).
\end{align*}

The proof of statements (2) and (4) are similar to the proof of (1), and we do not repeat these proofs.

The proof of statement (3) is similar to the proof of (5). Thus, we give the proof of (5) below. We have that
\begin{align*}
    y_{2}h_{1,2} = \frac{1}{x_{1}}\cdot m \in \Supp(r)
    &\iff \frac{y_{2n+4}}{y_{1}x_{2n+4}}\cdot m \in \Supp(r) & & (\ref{lem:elements_in_supp_of_Q}:\; Q_{2}, a=1,c= 2n+4) \\
    &\iff \frac{y_{4}}{y_{1}x_{4}}\cdot m \in \Supp(r) & & (\ref{lem:cycling} \text{ iteratively}) \\
    &\iff \frac{y_{2}}{y_{1}x_{2}}\cdot m \in \Supp(r) & &(\ref{lem:cycling_via_Q} : Q_{1}) \\
    &\iff \frac{y_{2}x_{2n+4}}{x_{1}x_{2}y_{2n+4}}\cdot m \in \Supp(r) & & (\ref{lem:elements_in_supp_of_Q}:\; Q_{2}, a=1, c = 2n+4) \\
    &\iff \frac{y_{2}x_{4}}{x_{1}x_{2}y_{4}}\cdot m \in \Supp(r) & & (\ref{lem:cycling} \text{ iteratively}).
\end{align*}
\eproof

\bcor 
\label{cor:Jn2_supp_r_exclusivity}
Let the setup be as in \ref{setup:Jn2}. We utilize notation as introduced in Lemma \ref{lem:Jn2_supp_r_divisibility}. Let $r \in J$, then
\begin{align*}
    \card{ \{m, x_{1}y_{2}h_{1,3}, x_{2}y_{1}h_{1,3}\} \cap \Supp(r)  } \equiv 0 \mod{2}.
\end{align*}
\ecor 

\bproof
Lemma \ref{lem:Jn_m_in_Supp_f} applied to $Q_{2}$ implies that
\begin{align*}
    \card{ \{m, \frac{x_{1}y_{2n+4}}{y_{1}x_{2n+4}}\cdot m, \frac{y_{1}x_{2n+4}}{x_{1}y_{2n+4}}\cdot m \} \cap \Supp(r) } \equiv 0 \mod{2}.
\end{align*}
Lemma \ref{lem:cycling} implies that 
\begin{align*}
    \frac{y_{1}x_{2n+4}}{x_{1}y_{2n+4}}\cdot m \in \Supp(r) \iff \frac{y_{1}x_{4}}{x_{1}y_{4}}\cdot m \in \Supp(r).
\end{align*}
We observe that
\begin{align*}
    \frac{x_{1}y_{2n+4}}{y_{1}x_{2n+4}}\cdot m \in \Supp(r)
    &\iff \frac{x_{1}y_{4}}{y_{1}x_{4}}\cdot m \in \Supp(r) & &(\ref{lem:cycling} \text{ iteratively}) \\
    &\iff \frac{x_{1}y_{2}}{y_{1}x_{2}}\cdot m \in \Supp(r) & &(\ref{lem:cycling_via_Q}: Q_{1}) \\
    &\iff \frac{x_{2n+4}y_{2}}{y_{2n+4}x_{2}}\cdot m \in \Supp(r) & &(\ref{lem:cycling_via_Q}: Q_{2}) \\
    &\iff \frac{x_{4}y_{2}}{y_{4}x_{2}}\cdot m \in \Supp(r) & &(\ref{lem:cycling} \text{ iteratively})
\end{align*}
which completes the proof.
\eproof

\bprop 
\label{prop:Jn2_J_contained_in_Q}
Let the setup be as in \ref{setup:Jn2}. We utilize notation as introduced in Lemma \ref{lem:Jn2_supp_r_divisibility}. Then,
\begin{align*}
    J + L + \a_{0} \subset ( h_{1,1} + h_{1,2} + h_{1,3}, h_{2,1} + h_{2,2} ) + L + \a_{0}.
\end{align*}
\eprop 

\bproof
Let $r \in J$. Let $S := \Supp(r) \setminus (L+\a_{0})$. Define the sets
\begin{align*}
    S_{1} &= \{h_{1,1},h_{1,2},h_{1,3}\} & S_{6} &= \{ m,x_{1}y_{2}h_{1,3} \} & S_{9} &= \{ h_{2,1}, h_{2,2} \}  \\
    S_{2} &= \{x_{1} h_{1,2}, x_{1} h_{1,3}\} &  S_{7} &= \{ m,x_{2}y_{1}h_{1,3} \} & S_{10} &= \{ y_{1}h_{2,1}, y_{1}h_{2,2} \} \\
    S_{3} &= \{y_{1} h_{1,1}, y_{1} h_{1,3}\} &  S_{8} &= \{ x_{1}y_{2}h_{1,3}, x_{2}y_{1}h_{1,3} \} & S_{11} &= \{ y_{2}h_{2,1}, y_{2}h_{2,2} \} \\
    S_{4} &= \{x_{2} h_{1,1}, x_{2} h_{1,3}\} & & & S_{12} &= \{ y_{1}y_{2}h_{2,1}, y_{1}y_{2}h_{2,2} \} \\
    S_{5} &= \{y_{2} h_{1,2}, y_{2} h_{1,3}\} 
\end{align*}
Then, Lemmas \ref{lem:Jn2_supp_r_divisibility}, \ref{lem:Jn2_two_elts_in_supp_r}, \ref{lem:Jn2_three_elts_in_supp_r}, and \ref{cor:Jn2_supp_r_exclusivity} imply that there exists $1 \leq t \leq 12$ and $1 \leq i_{1} < \cdots < i_{t} \leq 12$ satisfying $\# \left( \{i_{1},\ldots,i_{t}\} \cap \{6,7,8\} \right) \leq 1$ such that 
\begin{equation}
    \label{eqn:Jn2_reln_1}
    S = S_{i_{1}} \sqcup \cdots \sqcup S_{i_{t}}.
\end{equation}
We observe that for the elements of $S_{2}$ that we have the relation 
\begin{equation}
    \label{eqn:Jn2_reln_2}
    x_{1} (h_{1,1} + h_{1,2} + h_{1,3}) = x_{1} h_{1,2} + x_{1} h_{1,3} \mod{(L + \a_{0})}.
\end{equation}
A similar relation holds for the elements of $S_{i}$ for $3 \leq i \leq 7$. The elements of $S_{8}$ satisfy the relation
\begin{equation}
    \label{eqn:Jn2_reln_3}
    (x_{1}y_{2} + x_{2}y_{1})(h_{1,1} + h_{1,2} + h_{1,3}) = x_{1}y_{2}h_{1,3} + x_{2}y_{1}h_{1,3} \mod{(L + \a_{0})}.
\end{equation}
It follows that
\begin{align*}
    r = \sum_{j = 1}^{t} \sum_{g\in S_{i_{j}}} g \mod{(L + \a_{0})}.
\end{align*}
Now, equations \eqref{eqn:Jn2_reln_2} and \eqref{eqn:Jn2_reln_3} imply that
\begin{align*}
    \sum_{g\in S_{i_{j}}} g \in (h_{1,1}+h_{1,2}+h_{1,3},h_{2,1} + h_{2,2}),
\end{align*}
which completes the proof.
\eproof

\bthm
\label{thm:Jn2_not_F_pure}
Let $n\geq 0$ be a positive integer, and $G$ the graph $\mathrm{co\!-\!XF}_6^{2n+2}$. Then, $J_{G}$ is not F-pure.
\ethm

\bproof
When $n = 0$ and $n = 1$, it can be verified via Macaulay2 that $J_{G}$ is not F-pure. Hence, it suffices to consider the cases $n \geq 3$. We utilize the notation for $h_{i,j}$ and $L$ as they are defined in Lemma \ref{lem:Jn2_supp_r_divisibility}. We will use Proposition \ref{prop:reduction_step_fedder_computation} to prove that $J_{G}$ is not F-pure. Adopting the notation from Proposition \ref{prop:reduction_step_fedder_computation} we have that $A = \{2n+6,2n+7\}$, $B = [2n+7] \setminus (A\cup \{1,2\})$, $Q = P(A,B)$, $m = 7$, $B' = [7] \setminus \{1,2\}$, $Q' = P(\varnothing,B')$. We define the ideal
\begin{align*}
    K := ( h_{1,1} + h_{1,2} + h_{1,3}, h_{2,1} + h_{2,2} ) + L + (\{ x_{i}^{2},y_{i}^{2} \}_{i=6,7} ).
\end{align*}
We define $\mc{C} := \{P_{i,i+2} \mid 3 \leq i \leq 2n+3\}\cup \{Q_{i} \mid i \in \{1,2\}\}$. Proposition \ref{prop:Jn2_J_contained_in_Q} implies that the containment  \eqref{eqn:reduction_step_fedder_computation_assumption_1} of Proposition \ref{prop:reduction_step_fedder_computation} is satisfied. We verify using Macaulay2 that the containment \eqref{eqn:reduction_step_fedder_computation_assumption_2} of Proposition \ref{prop:reduction_step_fedder_computation} is satisfied. Thus, Proposition \ref{prop:reduction_step_fedder_computation} implies that $J_{G}$ is not F-pure.
\eproof

\section{Applications of Matsuda's Conjecture}
\label{sec:application_matsuda_theorem}

In this section, we deduce as a consequence of Theorem \ref{thm:matsudas_conjecture_is_true} that unmixed binomial edge ideals of K\"onig type are weakly closed, and that weakly closed graphs are stable under vertex deletion and completion.

\subsection{K\"onig type, Unmixed Binomial Edge Ideals Are Weakly Closed}

In \cite{herzog2022graded} Herzog, Hibi, and Moradi introduce the notion of a graded ideal of K\"onig type. 
\bdefn
$S = \kk[x_1,\ldots,x_n]$, and let $I \subseteq S$ be a graded ideal of height $h$. We say that $I$ is of \textbf{K\"onig type}, if there exists a sequence $\ul{f} = f_1,\ldots,f_h$ of homogeneous polynomials
that form part of a minimal system of generators of $I$ and a monomial order $<$
on $S$ such that $\mathrm{in}_<(f_1),\ldots,\mathrm{in}_<(f_h)$ is a regular sequence.
\edefn
They characterize the binomial edge ideals of K\"onig type via the following theorem.

\bthm[{\cite[Theorem 3.5]{herzog2022graded}}]
\label{thm:bei_konig_type}
Let $G$ be a graph. Then, $J_{G}$ is of K\"onig type if and only if $G$ contains vertex disjoint paths $P_{1}, \ldots, P_{t}$ such that 
\begin{align*}
    \sum_{i=1}^{t} \card{E(P_{i})} = \hgt J_{G}.
\end{align*}
\ethm 

Binomial edge ideals of K\"onig type satisfy various ``rigidity" properties. For instance, the property of being Cohen-Macaulay is independent of characteristic for binomial edge ideals of K\"onig type {\cite[Corollary 3.8]{herzog2022graded}}. We will show next that \textit{unmixed} binomial edge ideals of K\"onig type are F-pure independent of the base field so long as it has positive characteristic. For us, by an \textbf{unmixed} ideal we mean that all of the associated primes of the ideal have the same height.

\bprop 
\label{prop:konig_unmixed_bei_F_pure}
Let $G$ be a graph on $n$ vertices. Let $J_{G}$ be an unmixed ideal of K\"onig type. If the characteristic of the base field is $p >0$, then $R_{G}$ is F-pure.
\eprop 

\bproof
Let $P_{1},\ldots,P_{t}$ be as in Theorem \ref{thm:bei_konig_type}. Let $\alpha_{1},\ldots,\alpha_{g}$ denote the binomials corresponding to the edges of the $P_{i}$, where $g = \hgt J_{G}$. By {\cite[Corollary 3.3]{pandey2024linkage}}, if $R/\a$ is F-injective, where $\a = (\alpha_{1},\ldots,\alpha_{g})$, then $R/J_{G}$ is F-pure as $\{\alpha_{i}\}_{i=1}^{g}$ form a regular sequence in $J_G$ of length $\hgt J_{G}$. The ideal $\a$ is even F-pure. Since the $\alpha_{i}$ form a regular sequence, it follows from Proposition~\ref{prop:reg:sequence} that
\begin{align*}
    \a^{[p]} : \a = \left( \prod_{i=1}^{g} \alpha_{i}^{p-1} \right) + \a^{[p]}.
\end{align*}
The lead term of $\prod_{i=1}^{g} \alpha_{i}^{p-1}$ with respect to the lexicographic term order induced by $x_{1} > x_{2} > \cdots > x_{n} > y_{1} > y_{2} > \cdots > y_{n}$ does not belong to $\m^{[p]}$. Thus, 
\begin{align*}
    \a^{[p]}:\a \not \subset \m^{[p]}.
\end{align*}
Hence, Fedder's criterion (Theorem \ref{thm:fedder_criterion}) implies that $R/\a$ is F-pure.
\eproof 

\bcor 
\label{cor:unmixed_Konig_type_are_wc}
Let $G$ be a graph on $n$ vertices. Suppose that the binomial edge ideal $J_{G}$ is an unmixed ideal of K\"onig type. Then $G$ is weakly closed.
\ecor 

\bproof
Note that $J_{G}$ being unmixed (resp., of K\"onig type) are combinatorial statements independent of the base field; see, for instance, Proposition \ref{prop:min_primes_binomial_edge_ideal} (resp., \cite{laclair2023invariants}). Hence, we may assume that the base field has characteristic $2$. Proposition \ref{prop:konig_unmixed_bei_F_pure} implies that $R_{G}$ is F-pure, and Theorem \ref{thm:matsudas_conjecture_is_true} implies that $G$ is weakly closed.
\eproof

We note that one cannot remove the hypothesis that $J_{G}$ is unmixed. Indeed, for the big claw, its binomial edge ideal is of K\"onig type (see, {\cite[Theorem 5.6]{laclair2023invariants}}), but it is not weakly closed by Theorem \ref{thm:gallai}. (The big claw is in fact $T_{2}$ in Theorem \ref{thm:gallai}.) An interesting question that the authors do not know the answer to is to what extent the converse of Corollary \ref{cor:unmixed_Konig_type_are_wc} holds.

\bques
If $G$ is a graph which is weakly closed (with $J_{G}$ unmixed), then is $J_{G}$ of K\"onig type?
\eques 

For further reading on binomial edge ideals of K\"onig type, we refer the reader to \cite{herzog2022graded}, \cite{laclair2023invariants}, and \cite{williams2023lf}.

\subsection{Weakly Closed is Preserved Under Vertex Deletion and Completion}

In this section, we prove that the property of being weakly closed is preserved under vertex completion and deletion. This is an important observation as there is a short exact sequence relating the binomial edge ideal of a graph to that of its completion and deletion.

\bprop[{\cite[Lemma 4.8]{ohtani2011graphs}}, {\cite[Formula 1]{kumar2020depth}}]
\label{lem:exact_sequence_binomial_edge_ideal}
Let $G$ be a graph and $v \in G$ a non-simplicial vertex. Then, there is a short exact sequence of $R$-modules
\begin{equation}
\label{eqn:ses_binomial_edge_ideal}
    0 \ra R/J_{G} \ra R/J_{G_{v}} \oplus R_{v}/J_{G \setminus v} \ra R_{v}/J_{G_{v} \setminus v} \ra 0
\end{equation}
where $R_{v} = \k[X_{i},Y_{i} : i \in [n] \setminus \{v\}]$.
\eprop

This short exact sequence  \eqref{eqn:ses_binomial_edge_ideal} has proven useful in studying the Castelnuovo--Mumford regularity \cite{kumar2020regularity}, \cite{kumar2021binomial}, \cite{kumar2020depth}, \cite{laclair2024regularity}, \cite{laclair2024new} and Cohen--Macaulay property of binomial edge ideals \cite{bolognini2018binomial}, \cite{bolognini2022cohen}.

\bprop 
\label{prop:wc_preserved_under_completion_deletion}
Let $G$ be a weakly closed graph, and $v$ a vertex of $G$. Then the graphs $G \setminus v$, $G_{v}$, and $G_{v} \setminus v$ are weakly closed.
\eprop 

\bproof
Since $G$ is a weakly closed graph, $R_{G}$ is F-pure whenever the base field $\k$ has characteristic two (Theorem \ref{thm:weakly_closed_are_F_pure}). Lemma \ref{lem:F_pure_preserved_vertex_deletion_completion} implies that $G \setminus v$, $G_{v}$, and $G_{v} \setminus v$ define F-pure binomial edge ideals. Since the base field has characteristic two, Theorem \ref{thm:matsudas_conjecture_is_true} implies that the graphs $G \setminus v$, $G_{v}$, and $G_{v} \setminus v$ are all weakly closed.
\eproof 

\section{Further Questions} \label{sec:further_questions}

The results of this paper raise the interesting open problem of determining for which characteristics the binomial edge ideals of graphs from the ``co-regular" families are F-pure. Using Macaulay2, we were able to calculate the F-purity for several small ``co-regular" graphs provided that the characteristic is small, which we present in Table \ref{table:results}. In the first column, we denote the graph whose binomial edge ideal we are considering. In the first row, we denote the characteristic of the residue field. We denote by ``Y", ``N", ``?" for when $R_G$ is F-pure, is not F-pure, or is unknown to be F-pure, respectively.

\begin{table}[h]
\centering
\begin{tabular}{|c||c|c|c|c|c|}
\hline
& \multicolumn{4}{c}{\,\,Characteristic}& \\ \hline
      Graph  &2& 3 & 5 & 7 & 11 \\ \hline \hline
    \xrowht[()]{7pt} $\ol{C_5}$ &N& Y & Y & Y & Y \\ \hline
    \xrowht[()]{7pt} $\ol{C_7}$ &N& Y & ? & ? & ? \\ \hline
   \xrowht[()]{7pt}  $\mathrm{co\!-\!XF}_1^3$ 
     &N& N & ? & ? & ? \\ \hline
    \xrowht[()]{7pt} $\mathrm{co\!-\!XF}_5^3$ &N& N & ? & ? & ? \\ \hline
    \xrowht[()]{7pt} $\mathrm{co\!-\!XF}_6^2$ &N& N & ? & ? & ? \\ \hline
\end{tabular}
\caption{F-Purity of ``Co-Regular" families in small characteristic}
\label{table:results}
\end{table}

In \cite{rauh2013generalized}, Rauh provides a primary decomposition of generalized binomial edge ideals, which is analogous to the primary decomposition of binomial edge ideals, and proves that generalized binomial edge ideals are always radical. Later, Seccia generalized Matsuda's theorem \ref{thm:weakly_closed_are_F_pure} to generalized binomial edge ideals, proving that when the graph is weakly closed graph the associated generalized binomial edge ideal is a Knutson ideal and hence defines an F-pure quotient ring in every positive characteristic \cite{seccia2023binomial}. Thus, a natural direction for future research is to investigate analogues of Theorems \ref{thm:matsudas_conjecture_is_true} and \ref{thm:intro_non_f_purity} for the class of generalized binomial edge ideals. 

Ene, Herzog, Hibi, and Mohammadi introduced another class of ideals associated to any pure simplicial complex, the determinantal facet ideals \cite{ene2013determinantal}. In \textit{loc. cit.}, it is shown that these ideals are not always radical. However, for certain families of determinantal facet ideals that generalize the closed binomial edge ideals, Benedetti, Seccia, and Varbaro proved that the corresponding quotient rings are F-pure in every positive characteristic \cite{benedetti2022hamiltonian}. Further classes of combinatorially defined ideals that could be amenable to investigations of their F-purity via the techniques introduced in this paper are parity binomial edge ideals \cite{kahle2016parity} and permanental binomial edge ideals \cite{herzog2015ideal}.

\section{Macaulay2 Computations} \label{sec:M2_computations}

We include the Macaulay2 computations present in this paper at the repository: $\newline$ \noindent \url{https://github.com/adamlaclair/F-Purity-of-Binomial-Edge-Ideals}

\section*{Acknowledgements}
The authors would like to thank the anonymous referee for their helpful feedback which improved this paper. The authors would also like to thank Anna Brosowsky, Kyle Maddox, Matthew Mastroeni, Lance Edward Miller, Vaibhav Pandey, Claidu Raicu, Lisa Seccia, Anurag Singh, and Matteo Varbaro for helpful discussions. The authors would like to thank Justin Fong for pointing out some typos. The first author would like to thank Uli Walther for helpful feedback on earlier drafts of this paper. LaClair was partially supported by National Science Foundation grant DMS--2100288 and by Simons Foundation Collaboration Grant for Mathematicians \#580839.  McCullough was partially supported by National Science Foundation grant  DMS--2401256.

\bibliographystyle{abbrv}
\bibliography{bibliography}

\end{document}